\documentclass[oneside, hidelinks, 12pt, a4paper,reqno]{amsart}
\pdfoutput=1
\newif\ifpersonal

\personaltrue 

\ifpersonal
\newcommand*{\personal}[1]{\textcolor[rgb]{0,0,1}{(Personal: #1)}}
\newcommand*{\todo}[1]{\textcolor{red}{(Todo: #1)}}
\else
\newcommand*{\personal}[1]{\ignorespaces}
\newcommand*{\todo}[1]{\ignorespaces}
\fi

\usepackage[toc,page]{appendix}
\usepackage[foot]{amsaddr}


\usepackage{amsthm,mathtools,bm,tensor,stmaryrd,thmtools}
\newcommand*{\boldone}{\text{\usefont{U}{bbold}{m}{n}1}}
\usepackage[T1]{fontenc}\newcommand{\yo}{\text{\usefont{U}{min}{m}{n}\symbol{'110}}}
\DeclareFontFamily{U}{min}{}
\DeclareFontShape{U}{min}{m}{n}{<-> dmjhira}{}
\usepackage[bitstream-charter]{mathdesign}
\usepackage{XCharter}
\usepackage{enumerate,comment,braket,xspace,tikz-cd,csquotes}
\usepackage[driver=dvips,centering,vscale=0.7,hscale=0.8]{geometry}
\usepackage{url}
\usepackage{microtype}
\usepackage[greek.polutoniko,english]{babel}
\usepackage[style=authoryear,backend=biber,sorting=nyt,citestyle=ieee-alphabetic, bibstyle=alphabetic,texencoding=ascii,bibencoding=utf8,maxnames=5,minnames=3,maxbibnames=99]{biblatex}
\usepackage{setspace}
\usepackage{enumitem}
\usepackage{eucal}
\usepackage[scr=boondoxo]{mathalfa}
\usepackage{listings}
\usepackage{colordvi}
\usepackage{tikz,pgf}
\usepackage{pict2e}
\usetikzlibrary{arrows,decorations.pathmorphing,backgrounds,positioning,fit,matrix,calc}
\tikzset{curve/.style={settings={#1},to path={(\tikztostart)
			.. controls ($(\tikztostart)!\pv{pos}!(\tikztotarget)!\pv{height}!270:(\tikztotarget)$)
			and ($(\tikztostart)!1-\pv{pos}!(\tikztotarget)!\pv{height}!270:(\tikztotarget)$)
			.. (\tikztotarget)\tikztonodes}},
	settings/.code={\tikzset{quiver/.cd,#1}
		\def\pv##1{\pgfkeysvalueof{/tikz/quiver/##1}}},
	quiver/.cd,pos/.initial=0.35,height/.initial=0}
\tikzset{tail reversed/.code={\pgfsetarrowsstart{tikzcd to}}}
\tikzset{2tail/.code={\pgfsetarrowsstart{Implies[reversed]}}}
\tikzset{2tail reversed/.code={\pgfsetarrowsstart{Implies}}}
\tikzset{Rightarrow/.style={double equal sign distance,>={Implies},->},
triple/.style={-,preaction={draw,Rightarrow}},
quadruple/.style={preaction={draw,Rightarrow,shorten >=0pt},shorten >=1pt,-,double,double
distance=0.2pt}}
\usepackage{epsfig}
\usepackage{epstopdf}
\usepackage{tikz-cd}
\usepackage{fancyhdr}
\usepackage{graphicx}
\usepackage{fullpage}
\makeatletter
\newcommand{\mylabel}[2]{#2\def\@currentlabel{#2}\label{#1}}
\makeatother
\usepackage[colorlinks,bookmarks]{hyperref}
\hypersetup{
	colorlinks=true,
	linkcolor=red,
        linktoc=page,
	anchorcolor=black,
	filecolor=blue,
	citecolor = blue, 
	menucolor=., 
	urlcolor=cyan,
	breaklinks=true,
    unicode=true,
}
\usepackage[capitalise]{cleveref}
\usepackage{etoolbox}
\usepackage{enumitem}
\usepackage{comment}
\usepackage{color}
\usepackage[totoc]{idxlayout}
\usepackage[lastexercise]{exercise}
\usepackage{xcolor}
\usepackage{breakcites}
\usetikzlibrary{positioning}
\newcommand{\bDelta}{\bm{\Delta}}
\makeatletter
\newcommand{\twoarrows}[3][0.2ex]{%
  \mathrel{\mathpalette\twoarrows@{{#1}{#2}{#3}}}%
}
\newcommand{\twoarrows@}[2]{\twoarrows@@#1#2}
\newcommand{\twoarrows@@}[4]{%
  \vcenter{\offinterlineskip\m@th
    \ialign{\hfil##\hfil\cr
      $#1#3$\cr
      \noalign{\vskip#2}
      $#1#4$\cr
    }%
  }%
}
\makeatother

\makeatletter
\newcommand*{\doublerightarrow}[2]{\mathrel{
		\settowidth{\@tempdima}{$\scriptstyle#1$}
		\settowidth{\@tempdimb}{$\scriptstyle#2$}
		\ifdim\@tempdimb>\@tempdima \@tempdima=\@tempdimb\fi
		\mathop{\vcenter{
				\offinterlineskip\ialign{\hbox to\dimexpr\@tempdima+1em{##}\cr
					\rightarrowfill\cr\noalign{\kern.5ex}
					\rightarrowfill\cr}}}\limits^{\!#1}_{\!#2}}}
\newcommand*{\triplerightarrow}[1]{\mathrel{
		\settowidth{\@tempdima}{$\scriptstyle#1$}
		\mathop{\vcenter{
				\offinterlineskip\ialign{\hbox to\dimexpr\@tempdima+1em{##}\cr
					\rightarrowfill\cr\noalign{\kern.5ex}
					\rightarrowfill\cr\noalign{\kern.5ex}
					\rightarrowfill\cr}}}\limits^{\!#1}}}
\makeatother
\usepackage{hologo}
\makeatletter
\def\@tocline#1#2#3#4#5#6#7{\relax
	\ifnum #1>\c@tocdepth 
	\else
	\par \addpenalty\@secpenalty\addvspace{#2}%
	\begingroup \hyphenpenalty\@M
	\@ifempty{#4}{%
		\@tempdima\csname r@tocindent\number#1\endcsname\relax
	}{%
		\@tempdima#4\relax
	}%
	\parindent\z@ \leftskip#3\relax \advance\leftskip\@tempdima\relax
	\rightskip\@pnumwidth plus4em \parfillskip-\@pnumwidth
	#5\leavevmode\hskip-\@tempdima
	\ifcase #1
	\or\or \hskip 1em \or \hskip 2em \else \hskip 3em \fi%
	#6\nobreak\relax
	\dotfill\hbox to\@pnumwidth{\@tocpagenum{#7}}\par
	\nobreak
	\endgroup
	\fi}
\makeatother

\newlist{defenum}{enumerate}{1} 
\setlist[defenum]{label=$(\arabic*)$, ref=\thedefn.$(\arabic*)$}
\crefalias{defenumi}{definition} 
\newlist{lemenum}{enumerate}{1} 
\setlist[lemenum]{label=$(\arabic*)$, ref=\thelemma.$(\arabic*)$}
\crefalias{lemenumi}{lemma} 
\newlist{assumpenum}{enumerate}{1} 
\setlist[assumpenum]{label=$(\arabic*)$, ref=\theassumption.$(\arabic*)$}
\crefalias{assumpenumi}{assumption} 
\newlist{propenum}{enumerate}{1} 
\setlist[propenum]{label=$(\arabic*)$, ref=\theprop.$(\arabic*)$}
\crefalias{propenumi}{proposition}
\newlist{exmpenum}{enumerate}{1} 
\setlist[exmpenum]{label=$(\arabic*)$), ref=\theexample.$(\arabic*)$}
\crefalias{exmpenumi}{example}
\newlist{thmenum}{enumerate}{1} 
\setlist[thmenum]{label=$(\arabic*)$, ref=\thetheorem.$(\arabic*)$}
\crefalias{thmenumi}{theorem}
\newlist{intrenum}{enumerate}{1}
\setlist[intrenum]{label=$(\arabic*)$,ref=$(\arabic*)$}
\crefalias{intrenumi}{section}
\newlist{corenum}{enumerate}{1} 
\setlist[corenum]{label=$(\arabic*)$, ref=\thecorollary.$(\arabic*)$}
\crefalias{corenumi}{corollary}
\newlist{factenum}{enumerate}{1} 
\setlist[factenum]{label=$(\arabic*)$, ref=\thefact.$(\arabic*)$}
\crefalias{factenumi}{fact}
\newlist{remarkenum}{enumerate}{1} 
\setlist[remarkenum]{label=$(\arabic*)$, ref=\theremark.$(\arabic*)$}
\crefalias{remarkenumi}{remark}
\newlist{axiomenum}{enumerate}{1} 
\setlist[axiomenum]{label=$(\arabic*)$, ref=\theaxiom.$(\arabic*)$}
\crefalias{axiomenumi}{axiom} 
\newlist{questenum}{enumerate}{1} 
\setlist[questenum]{label=$(\arabic*)$, ref=\thequestion.$(\arabic*)$}
\crefalias{questenumi}{question} 
\newtheorem{innercustomthm}{Theorem}
\newenvironment{customthm}[1]
{\renewcommand\theinnercustomthm{#1}\innercustomthm}
  {\endinnercustomthm}
  
   \makeatletter
\newcommand{\adjunction}{\@ifstar\named@adjunction\normal@adjunction}
\newcommand{\normal@adjunction}[4]{%
  #1\colon #2%
  \mathrel{\vcenter{%
    \offinterlineskip\m@th
    \ialign{%
      \hfil$##$\hfil\cr
      \longrightharpoonup\cr
      \noalign{\kern-.3ex}
      \smallbot\cr
      \longleftharpoondown\cr
    }%
  }}%
  #3 \noloc #4%
}
\newcommand{\named@adjunction}[4]{%
  #2%
  \mathrel{\vcenter{%
    \offinterlineskip\m@th
    \ialign{%
      \hfil$##$\hfil\cr
      \scriptstyle#1\cr
      \noalign{\kern.1ex}
      \longrightharpoonup\cr
      \noalign{\kern-.3ex}
      \smallbot\cr
      \longleftharpoondown\cr
      \scriptstyle#4\cr
    }%
  }}%
  #3%
}
\newcommand{\longrightharpoonup}{\relbar\joinrel\rightharpoonup}
\newcommand{\longleftharpoondown}{\leftharpoondown\joinrel\relbar}
\newcommand\noloc{%
  \nobreak
  \mspace{6mu plus 1mu}
  {:}
  \nonscript\mkern-\thinmuskip
  \mathpunct{}
  \mspace{2mu}
}
\newcommand{\smallbot}{%
  \begingroup\setlength\unitlength{.15em}%
  \begin{picture}(1,1)
  \roundcap
  \polyline(0,0)(1,0)
  \polyline(0.5,0)(0.5,1)
  \end{picture}%
  \endgroup
}
\makeatother

\newcommand{\infinity}{\mbox{\footnotesize $\infty$}}

\newcommand{\oblv}{\operatorname{oblv}}

\newcommand{\Einf}{\mathbb{E}_{{\scriptstyle\infty}}}

\newcommand{\LMod}{{\operatorname{LMod}}}
\newcommand{\LModtwo}{\mathbf{LMod}}
\newcommand{\Modtwo}{\mathbf{Mod}}

\newcommand{\RMod}{{\operatorname{RMod}}}

\newcommand{\op}{\operatorname{op}}

\newcommand{\PSt}{\operatorname{PSt}}

\newcommand{\Perf}{{\operatorname{Perf}}}

\newcommand{{\Cn}}{\operatorname{Cn}}

\newcommand{\BGm}{\mathbf{B}\Gm}

\newcommand{\Sphere}{\mathbb{S}}
\DeclareMathOperator*{\colim}{colim}

\newcommand{\Mapin}{{\underline{\smash{\Map}}}}

\newcommand{\scrA}{\mathscr{A}}
\newcommand{\scrC}{\mathscr{C}}
\newcommand{\scrD}{\mathscr{D}}

\newcommand{\scrF}{\mathscr{F}}
\newcommand{\scrG}{\mathscr{G}}
\newcommand{\scrO}{\mathscr{O}}
\newcommand{\scrX}{\mathscr{X}}
\newcommand{\scrY}{\mathscr{Y}}

\newcommand{\hsp}{\hspace{0.05cm}}

\DeclareMathOperator{\fib}{fib}

\DeclareMathOperator{\Spec}{Spec}

\DeclareMathOperator{\Qcoh}{QCoh}
\DeclareMathOperator{\QCoh}{QCoh}
\DeclareMathOperator{\Coh}{Coh}
\DeclareMathOperator{\IndCoh}{IndCoh}
\newcommand{\cSpec}{\operatorname{cSpec}}
\newcommand{\CX}{\operatorname{C}^{\bullet}(X;\Bbbk)}
\newcommand{\CY}{\operatorname{C}^{\bullet}(Y;\Bbbk)}
\newcommand{\CK}[1][1]{\operatorname{C}^{\bullet}(K(\pi,#1);\Bbbk)}

\newcommand{\aff}{\operatorname{aff}}
\newcommand{\COX}{\operatorname{C}^{\bullet}(\Omega_*X;\Bbbk)}
\newcommand{\ShvCat}{\operatorname{ShvCat}}
\newcommand{\ShvCattwo}{\mathbf{ShvCat}}

\newcommand{\Mod}{{\operatorname{Mod}}}

\newcommand{\lp}{\left(}
\newcommand{\rp}{\right)}
\newcommand{\scrU}{\mathscr{U}}
\newcommand{\scrV}{\mathscr{V}}
\newcommand{\CatUinfty}[1][1]{\operatorname{Cat}_{({\scriptstyle\infty},#1)}}

\newcommand{\CatVinfty}[1][1]{\smash{\widehat{\operatorname{Cat}}_{({\scriptstyle\infty},#1)}}}

\newcommand{\CatVinftycocom}[1][1]{\smash{\widehat{\operatorname{Cat}}}^{\operatorname{rex}}_{({\scriptstyle\infty},#1)}}

\newcommand{\CatVtwoinfty}[1][1]{\smash{\widehat{\mathbf{Cat}}_{({\scriptstyle\infty},#1)}}}

\newcommand{\CatWinfty}[1][1]{\smash{\widehat{\operatorname{CAT}}_{({\scriptstyle\infty},#1)}}}

\newcommand{\infinitone}{(\infinity,1)}

\newcommand{\PrLU}[1][1]{\operatorname{Pr}^{\operatorname{L}}_{({\scriptstyle\infty},#1)}}

\newcommand{\IndCohL}{\operatorname{IndCoh}^{\operatorname{L}}}
\newcommand{\PrLUtwo}[1][1]{\mathbf{Pr}^{\operatorname{L}}_{({\scriptstyle\infty},#1)}}
\newcommand{\PrLUtwotiny}[1][1]{\mathbf{Pr}^{\operatorname{L}}_{({\scriptscriptstyle\infty},#1)}}

\newcommand{\Ind}{{\operatorname{Ind}}}

\newcommand{\LC}{\operatorname{LocSys}}
\newcommand{\LocSysCat}{\operatorname{LocSysCat}}
\newcommand{\LocSysCattwo}{\mathbf{LocSysCat}}

\newcommand{\LinAPrLU}[1][1]{\operatorname{Lin}_{\scrA}\operatorname{Pr}^{\operatorname{L}}_{({\scriptstyle\infty},#1)}}

\newcommand{\LinAPrLUtwo}[1][1]{\mathbf{Lin}_{\scrA}\mathbf{Pr}^{\operatorname{L}}_{({\scriptstyle\infty},#1)}}
\newcommand{\LinAPrLUtwotiny}[1][1]{\mathbf{Lin}_{\scrA}\mathbf{Pr}^{\operatorname{L}}_{({\scriptscriptstyle\infty},#1)}}

\newcommand{\LinkPrLU}[1][1]{{{\operatorname{Lin}_{\Bbbk}}{\operatorname{Pr}^{\operatorname{L}}_{({\scriptstyle\infty},#1)}}}}
\newcommand{\LinkPrLUtwo}[1][1]{{\mathbf{Lin}_{\Bbbk}\mathbf{Pr}^{\operatorname{L}}_{({\scriptstyle\infty},#1)}}}

\newcommand{\LinkPrLUtwotiny}[1][1]{\mathbf{Lin}_{\Bbbk}\mathbf{Pr}^{\operatorname{L}}_{({\scriptscriptstyle\infty},#1)}}

\newcommand*{\longhookrightarrow}{\ensuremath{\lhook\joinrel\relbar\joinrel\rightarrow}}

\makeatletter

\def\enddoc@text{}
\makeatother
\newcounter{savedchapter}
\preto\appendix{\setcounter{savedchapter}{\arabic{chapter}}}
\newcommand\resumechapters{
	\setcounter{chapter}{\arabic{savedchapter}}
	\setcounter{section}{0}
	\gdef\@chapapp{\chaptername}
	\gdef\thechapter{\@arabic\c@chapter}
}
\makeatother

\newcommand{\scrS}{\mathscr{S}}

\newcommand{\Ebb}{\mathbb{E}}

\newcommand{\Ga}{\mathbb{G}_{\operatorname{a},\Bbbk}}
\newcommand{\Gm}{\mathbb{G}_{\operatorname{m},\Bbbk}}

\newcommand{\ZZ}{\mathbb{Z}}

\newcommand{\NN}{\mathbb{N}}

\newcommand{\QQ}{\mathbb{Q}}

\newcommand{\CC}{\mathbb{C}}

\newcommand{\PP}{\mathbb{P}}

\newcommand{\Map}{{\operatorname{Map}}}

\newcommand{\Sym}{{\operatorname{Sym}}}

\newcommand{\Fun}{{\operatorname{Fun}}}
\newcommand{\Funin}{{\underline{\smash{\Fun}}}}
\newcommand{\Funtwo}{\mathbf{Fun}}

\newcommand{\FuninL}{\Funin^{\operatorname{L}}}
\newcommand{\FuninLtwo}{{\underline{\smash{\mathbf{Fun}}}}^{\operatorname{L}}}

\newcommand{\LocSys}{{\operatorname{LocSys}}}

\newcommand{\CAlg}{{\operatorname{CAlg}}}

\newcommand{\Alg}{{\operatorname{Alg}}}

\numberwithin{equation}{subsection}
\theoremstyle{plain}
\newtheorem{theorem}[equation]{Theorem}
\crefname{theorem}{Theorem}{Theorems}

\crefname{thm}{Theorem}{Theorems}
\newtheorem{lemma}[equation]{Lemma}
\crefname{lemma}{Lemma}{Lemmas}
\newtheorem{proposition}[equation]{Proposition}
\crefname{proposition}{Proposition}{Propositions}

\crefname{prop}{Proposition}{Propositions}
\newtheorem{corollary}[equation]{Corollary}
\crefname{corollary}{Corollary}{Corollaries}

\crefname{conj}{Conjecture}{Conjectures}
\newtheorem{question}[equation]{Question}
\crefname{question}{Question}{Questions}

\crefname{slogan}{Slogan}{Slogans}

\theoremstyle{definition}
\newtheorem{defn}[equation]{Definition}
\crefname{defn}{Definition}{Definitions}
\newtheorem{parag}[equation]{}
\crefname{parag}{Paragraph}{Paragraphs}

\newtheorem{remark}[equation]{Remark}
\crefname{remark}{Remark}{Remarks}

\crefname{axiom}{Axiom}{Axioms}
\newtheorem{construction}[equation]{Construction}
\newtheorem{porism}[equation]{Porism}

\newtheorem{warning}[equation]{Warning}
\newtheorem{notation}[equation]{Notation}

\newtheorem{fact}[equation]{Fact}

\newtheorem{exmp}[equation]{Example}

\newtheorem*{warning*}{Warning}
\newtheorem*{assumption*}{Standing assumption}
\newcommand{\Sp}{\mathscr{S}\mathrm{p}}

\newcommand\restr[2]{{
		\left.\kern-\nulldelimiterspace
		#1
		\vphantom{\|}
		\right|_{#2} 
}}
\usepackage{lipsum}
\usepackage{fancyhdr}
\providecommand{\abstract}{}
\usepackage{abstract}
\DeclareGraphicsExtensions{.png,.jpg,.pdf,.mps}
\usepackage{fancyhdr}
\usepackage{lastpage}
\usepackage{multirow}
\usepackage{faktor}

\usepackage{array}
\usepackage{makecell}

\newcommand{\stackspace}{1.7}
\newcommand{\stack}[2][1cm]{\;\tikz[baseline, yshift=.65ex]%
	{\foreach \k [evaluate=\k as \r using (.5*#2+.5-\k)*\stackspace] in {1,...,#2}{%
			\ifodd\k{\draw[->](0,\r pt)--(#1,\r pt);}%
			\else{\draw[<-](0,\r pt)--(#1,\r pt);}\fi
	}}\;}
    \newcommand{\stackrev}[2][1cm]{\;\tikz[baseline, yshift=.65ex]%
	{\foreach \k [evaluate=\k as \r using (.5*#2+.5-\k)*\stackspace] in {1,...,#2}{%
			\ifodd\k{\draw[<-](#1,\r pt) -- (0,\r pt);}%
			\else{\draw[->](#1,\r pt) -- (0,\r pt);}\fi
	}}\;}

\usepackage{doc}

\begin{document}	
	\title{Higher Koszul duality and $n$-affineness}
	\author{James Pascaleff}
	\address[James Pascaleff]{University of Illinois\\West Green Street, 1409\\61801, Urbana, IL, United States}\email{\href{mailto:jpascale@illinois.edu}{jpascale@illinois.edu}}
	\author{Emanuele Pavia}
	\address[Emanuele Pavia]{SISSA\\Via Bonomea 265\\34136 Trieste, TS, Italy}	\email{\href{mailto:epavia@sissa.it}{epavia@sissa.it}}
	\author{Nicolò Sibilla}
	\address[Nicolò Sibilla]{SISSA\\Via Bonomea 265\\34136 Trieste, TS, Italy}
	\email{\href{mailto:nsibilla@sissa.it}{nsibilla@sissa.it}}
\maketitle

\begin{abstract}

In this paper we study $\mathbb{E}_n$-Koszul duality in the topological setting, and the closely related question of \emph{$n$-affineness} for Betti stacks. 
The $\mathbb{E}_n$-Koszul  dual of the algebra of chains on the $n$-fold loop space of a space $X$ is the algebra of cochains on $X$.  It was expected  that $\mathbb{E}_n$-Koszul duality should induce a kind of Morita equivalence between categories of iterated modules, but even the precise formulation of such a statement was not known. We give a rigorous formulation, and a proof, of such an $\mathbb{E}_n$-Koszul duality in the topological setting as an equivalence of $(\infinity,n)$-categories. 
Conceptually, our main innovation is highlighting  the  coaffine stack defined by the \emph{cospectrum} of $\mathrm{C}^{\bullet}(X;\Bbbk)$ as a key geometric object supporting Koszul duality.    
Our result is new already in the classical case $n=1$, although it can be seen to recover  well known formulations of $\mathbb{E}_1$-Koszul duality as a Morita equivalence of module categories (up to  appropriate completions of the $t$-structures). We also investigate (higher) affineness  properties of  Betti stacks.  We give a complete characterization of $n$-affine Betti stacks, in terms of the $0$-affineness of their iterated loop space. As a consequence, we prove that $n$-truncated Betti stacks are $n$-affine; and that $\pi_{n+1}(X)$ is an obstruction to $n$-affineness.

 \end{abstract}
\tableofcontents
{\section*{Introduction}
\addtocontents{toc}{\protect\setcounter{tocdepth}{0}}
\renewcommand\thesubsection{I.\arabic{subsection}}
\renewcommand\theequation{\thesubsection.\arabic{equation}}
\setcounter{equation}{0}
 
In this paper we study local systems of higher categories over spaces. Our main goal is formulating and proving  $\Ebb_k$-Koszul duality between the algebras of chains on the   iterated loop space of a space $X$, and the algebra of cochains on $X$, as an equivalence of $(\infinity,k)$-categories.

To understand the context of our results, let us recall the classical picture of   Koszul duality in the topological setting. Let $\Bbbk$ be a field of characteristic zero. If $X$ is a connected space, local systems of $\Bbbk$-vector spaces over $X$ are determined by monodromy data, in the sense that the abelian category of such local systems is equivalent to the category of representations of $\pi_{1}(X)$. 
Understanding the higher cohomology of local systems requires more information that is not captured by $\pi_1(X)$, and in fact depends on full homotopy type of $X$. More precisely,  
the stable $(\infinity,1)$-category of complexes of sheaves of vector spaces on $X$ whose cohomology sheaves are local systems is equivalent to the stable $(\infinity,1)$-category of complexes of modules over $\mathrm{C}_{\bullet}(\Omega_*X;\Bbbk)$, the algebra of chains on the based loop space of $X$. In formulas, there is an equivalence \begin{equation}
\label{fundintro}
\tag{$\ast$}
\LocSys(X;\Bbbk) \simeq \LMod_{\mathrm{C}_{\bullet}(\Omega_*X;\Bbbk)} 
\end{equation} which we  refer to as the \emph{monodromy equivalence}.

If $X$ is simply connected, the category of local systems on $X$ admits also an alternative description.  This is a consequence of 
  \emph{Koszul duality}, the main object of study of this article.  
 The $\Ebb_{1}$-Koszul dual of $\mathrm{C}_{\bullet}(\Omega_*X;\Bbbk)$ is $\CX$, the algebra of cochains on $X$. Under certain finiteness hypotheses, the reciprocal duality also holds, and moreover there is a tight relationship between the $(\infinity,1)$-categories of complexes of modules over these two algebras. Hence, under these hypotheses, local systems over $X$ also admit a description in terms of $\CX$.

Passing to the $k$-categorical level, local systems of complexes of vector spaces are replaced by local systems of $\Bbbk$-linear $(\infinity,k)$-categories,  the loop space is replaced by the $k$-fold iterated loop space $\Omega_{*}^{k}X$, and Koszul duality of $\Ebb_{1}$-algebras is replaced by Koszul duality of the $\Ebb_{k}$-algebras $\mathrm{C}_{\bullet}(\Omega^{k}_*X;\Bbbk)$ and $\CX$. In this paper, relying on the    groundwork set down in the companion paper \cite{Pascaleff_Pavia_Sibilla_Local_Systems}, we generalize the classical picture of  $\mathbb{E}_1$-Koszul duality   to the setting of $\mathbb{E}_k$-Koszul duality. Along the way, we also explore the closely related question of  $k$-affineness for Betti stacks. 


The next section of this introduction provides a more detailed discussion of Koszul duality and outlines the main ideas underlying our approach to $\mathbb{E}_n$-Koszul duality. Readers primarily interested in our results may wish to proceed directly to \cref{mainresults}, where we present an overview of the paper's structure and state our main theorems.

\subsection{Koszul duality}
\label{intro1}


We start by recalling in some greater detail the  classical picture of   Koszul duality in the topological setting.   A standard reference  is  \cite{koszuldualitypatterns}, although our exposition will depart in some respects from that paper. 
Let $X$ be a pointed and simply connected finite CW complex, and let $\Omega_*X$ be the space of loops in $X$ based at the chosen point. The complexes $\mathrm{C}_{\bullet}{\lp\Omega_*X;\Bbbk\rp}$ and $\CX$ naturally acquire the structure of augmented differential graded algebras, and they are \textit{Koszul dual} one to the other. Classically, this means that we have the following two closely related statements: 
  \begin{intrenum}
 \item\label{intro:1} The algebra of endomorphisms of the augmentation of $\mathrm{C}_{\bullet}{\lp\Omega_*X;\Bbbk\rp}$ is  equivalent to $\mathrm{C}^{\bullet}(X;\Bbbk)$, and vice versa. In symbols:
 $$
 \mathrm{C}_{\bullet}{\lp\Omega_*X;\Bbbk\rp}\simeq  \Mapin_{\mathrm{C}^{\bullet}(X;\Bbbk)}{\lp\Bbbk,\hsp\Bbbk\rp} \quad \text{and} \quad 
 \mathrm{C}^{\bullet}(X;\Bbbk) \simeq \Mapin_{\mathrm{C}_{\bullet}{\lp\Omega_*X;\Bbbk\rp}}{\lp\Bbbk,\hsp\Bbbk\rp}.
 $$
\item\label{intro:morita}The functor between $\LMod_{\mathrm{C}_{\bullet}{\lp\Omega_*X;\Bbbk\rp} }$ and 
$\LMod_{\mathrm{C}^{\bullet}{\lp X;\Bbbk\rp} }$  
given by  
\begin{equation}
\label{2eqintro}
\Mapin_{\mathrm{C}_{\bullet}{\lp\Omega_*X;\Bbbk\rp}}{\lp \Bbbk,\hsp - \rp}:  \LMod_{ \mathrm{C}_{\bullet}{\lp\Omega_*X;\Bbbk\rp} } \longrightarrow 
\LMod_{ \mathrm{C}^{\bullet}{\lp X;\Bbbk\rp} }
\end{equation}
is \emph{almost}, but not quite, a Morita equivalence,  as we explain next. \end{intrenum}

Under the equivalence  $
\LocSys(X;\Bbbk) \simeq \LMod_{\mathrm{C}_{\bullet}(\Omega_*X;\Bbbk)} 
$ the augmentation module is sent to the constant local system $\underline{\Bbbk}_X$,  
and the functor $\Mapin_{\mathrm{C}_{\bullet}{\lp\Omega_*X;\Bbbk\rp}}{\lp \Bbbk,\hsp - \rp}$ corresponds to the enhanced global sections  \begin{equation}
\label{globsecintro}
\Gamma(X, -): \LocSys(X;\Bbbk) \longrightarrow \Mod_{\mathrm{C}^{\bullet}{\lp X;\Bbbk\rp}}.
\end{equation}
As it turns out, the enhanced global section functor \eqref{globsecintro}, and thus  functor \eqref{2eqintro}, are almost never  equivalences.    
Using the terminology of algebraic geometry, we can express this by saying that finite CW complexes, or more precisely their \emph{Betti stacks}, 
are almost never \emph{affine} (see \cref{warning:wrongkoszul}). Here we understand affineness precisely as the property that global sections define an equivalence between the stable $(\infinity,1)$-category of quasi-coherent sheaves, and the stable $(\infinity,1)$-category of modules over the global sections of the structure sheaf. 

The failure of Koszul duality to give rise to an actual Morita equivalence is one the main subtleties of the theory.  
There are several ways to obviate this, and turn \ref{intro:morita} into a rigorous mathematical statement. 
It is possible to show  that functor \eqref{2eqintro} does restrict to an equivalence between categories of appropriately bounded modules: more precisely, there is an equivalence 
\begin{equation}
\label{bounded} \LMod^-_{ \mathrm{C}_{\bullet}{\lp\Omega_*X;\Bbbk\rp} } \simeq 
\LMod^+_{ \mathrm{C}^{\bullet}{\lp X;\Bbbk\rp} }
\end{equation}
between  \emph{bounded above}  $\mathrm{C}_{\bullet}{\lp\Omega_*X;\Bbbk\rp}$-modules, and \emph{bounded below} $\mathrm{C}^{\bullet}(X;\Bbbk)$-modules (compare with \cite[Theorem $12.6$]{koszuldualitypatterns}). Alternatively, we can modify the notion of module we work with. Namely, the  functor \eqref{2eqintro} induces an equivalence \begin{equation}
\label{eqqeqintro}
\LMod_{\operatorname{C}_{\bullet}(\Omega_*X;\Bbbk)} 
\simeq  
\IndCoh_{\operatorname{C}^\bullet(X, \Bbbk)}
\end{equation}
where the right-hand side is the $(\infinity,1)$-category of \emph{ind-coherent} modules over $\CX$, which we define formally in \cref{sec:nkoszulcoaffine} of the main text: suffice it to say for the moment that, in this setting, this is the $(\infinity,1)$-category generated by the augmentation module. It is this latter formulation of $\mathbb{E}_1$-Koszul duality  which will be particularly  relevant  for our approach to $\mathbb{E}_k$-Koszul duality. 

From now on, in order to be consistent with the notations of the paper, we fix an integer $n\geqslant0$ and set $k\coloneqq n+1$. Let $X$ be a pointed and $(n+1)$-connected finite CW complex. Much as before, we can associate to $X$ two augmented algebras: except now these will be $\mathbb{E}_{n+1}$- rather than $\mathbb{E}_1$-algebras. On the one hand, the $(n+1)$-th iterated loop space
$$
\Omega_*^{n+1}X\coloneqq\Omega_* \ldots \Omega_* X
$$is a $\mathbb{E}_{n+1}$-space; thus, $\mathrm{C}_\bullet(\Omega_*^{n+1}X; \Bbbk)$ carries a $\mathbb{E}_{n+1}$-product. In this setting we have  an $(\infinity,n)$-categorical analogue of the monodromy equivalence, which relates local systems of presentable $(\infinity,n)$-categories and  iterated modules over $\mathrm{C}_\bullet(\Omega_*^{n+1}X; \Bbbk)$. This is the main result of the companion paper 
\cite{Pascaleff_Pavia_Sibilla_Local_Systems}, and we will refer to it as the \emph{categorified monodromy equivalence}: 
\begin{align}
\label{catmon}
\tag{$\ast\ast$}
(n+1)\LocSysCattwo^n(X)&\simeq  (n+1)\mathbf{LMod}_{n\LModtwo_{\Omega_*^{n+1}X}(\scrS)}\PrLUtwo[n].
\end{align} 
On the other hand, the algebra of $\Bbbk$-valued cochains $\CX$ on $X$ is naturally a $\Einf$-algebra, so we can regard it in particular as an $\mathbb{E}_{n+1}$-algebra. The key claim is that these two algebras are  $\mathbb{E}_{n+1}$-Koszul dual to each other.

 Applying \cite[Theorem $4.4.5$]{dagx} one can  \textit{almost} deduce an $\Ebb_n$-analog of statement \ref{intro:1}. Indeed, using \cite[Example $5.3.1.5$ and Lemma $5.3.1.11$]{ha}, one can prove that the Koszul dual of an augmented $\Ebb_{n+1}$-algebra $A\to\Bbbk$ is the morphism object
\[
A^!\coloneqq\Mapin_{\Mod^{\Ebb_{n+1}}_A}{\lp A,\hsp\Bbbk\rp}.
\]
However, no analog of statement \ref{intro:morita} has been established in the literature.  In fact, as far as we are aware of, even how to properly formulate \ref{intro:morita} in the $\mathbb{E}_{n+1}$-setting  was not known. 

In this paper we prove an $\Ebb_{n+1}$-analog of statement \ref{intro:morita}, and this will yield in particular an  $\Ebb_{n+1}$-analog of statement \ref{intro:1}. Conceptually,  our  main innovation consists in reinterpreting  statement \ref{intro:morita}, and in particular equivalence \eqref{eqqeqintro}, from a novel perspective which makes categorification possible.  
The key idea is that instead of viewing $\CX$ merely as a $\Einf$-algebra, we can do algebraic geometry with it. The algebra of cochains $\CX$ can be endowed with a structure of a commutative dg algebra 
concentrated in negative degrees; the contrary of what we require of a derived affine scheme. Following To\"{e}n (\cite{toenchampsaffines}) and Lurie (\cite{dagviii}), however,  we can still view such an algebra as the algebra of functions on a kind of geometric object: namely a \emph{coaffine stack}, which is called the cospectrum of $\CX$, and denoted $\cSpec(\CX)$.

 Quasi-coherent sheaves on $\cSpec(\CX)$ can be viewed as a \emph{renormalization} of the category of $\CX$-modules. Under our assumptions on $X$,  they coincide with ind-coherent modules
\begin{equation}
\label{eqqeqintro11/2}
 \QCoh(\cSpec(\CX))  \simeq \IndCoh_{\operatorname{C}^\bullet(X, \Bbbk)}.
\end{equation}
This yields a new formulation of the \emph{almost} Morita equivalence which is at the heart of $\mathbb{E}_1$-Koszul duality.  Namely, combining \eqref{eqqeqintro} and \eqref{eqqeqintro11/2}  we obtain  an equivalence 
\begin{equation}
\label{eqqeqintro2}
\LMod_{\operatorname{C}_{\bullet}(\Omega_*X;\Bbbk)} 
\simeq  
\QCoh(\cSpec(\CX)).
\end{equation}
We stress that, as far as we know, this formulation 
of $\mathbb{E}_1$-Koszul duality as equivalence (\ref{eqqeqintro2}) is new. It is  our guiding principle in our exploration of higher Koszul duality. Indeed, via (\ref{eqqeqintro2}), the notion of ind-coherent  $\CX$-module required to turn Koszul duality into an actual Morita equivalence is encoded  in the geometry of $\cSpec(\CX)$. The great advantage over other  formulations of Koszul duality is that equivalence \eqref{eqqeqintro2} is well adapted to categorification.

\subsection{Main results}
\label{mainresults}
We shall give next an overview  of the contents of the paper, and state our main results. In \cref{mainsetup} we survey briefly all preliminary material which will be required in the remainder of the paper. This includes the theory of presentable $(\infinity,n)$-categories, which was recently introduced by Stefanich in \cite{stefanich2020presentable}, and forms the technical  backbone of many of our constructions. Additionally we recall the main results we obtained in the companion paper \cite{Pascaleff_Pavia_Sibilla_Local_Systems}, and most importantly the categorified monodromy equivalence \eqref{catmon}.
 
As we have explained, the reason that stating   $\mathbb{E}_1$-Koszul duality  an equivalence of categories is  subtle is that Betti stacks are rarely affine (see the discussion around the definition of the enhanced global sections functor in  (\ref{globsecintro})). This feature persists in the $\mathbb{E}_n$-setting, where  the relevant concept   is  a categorification of ordinary affineness, i.e. the notion of  \emph{$n$-affineness}  which was introduced by Gaitsgory \cite{1affineness} and Stefanich \cite{stefanichthesis}. We devote   \cref{sec:naffineness1} to  studying  the question of $n$-affineness for Betti stacks. We obtain a complete characterization of $n$-affine Betti stacks, which has however the drawback of not being explicit: it reduces the question of $n$-affineness of a Betti stack $X_{\operatorname{B}}$, which is $n$-categorical in nature, to a purely $1$-categorical condition on the Betti stack of the iterated loop space 
$\Omega_*^nX$.
\begin{customthm}{A}[\cref{prop:naffineifloopn-1} and \cref{remark:decategorified}]
\label{thm:main3}
Let $X$ be a space with choices of base points on each connected component of $\Omega^m_*X$, for $0\leqslant m \leqslant n$. Then its Betti stack $X_{\operatorname{B}}$ is $n$-affine if and only if the global section functor 
$$
\Gamma(\Omega^n_*X,-)  \colon \LocSys(\Omega_*^nX ; \Bbbk) \longrightarrow \Mod_{\Bbbk}
$$ 
is monadic.
\end{customthm}
The criterion for $n$-affineness of \cref{thm:main3} is however difficult to check in practice. To obviate this shortcoming we extract one sufficient condition, and one necessary condition, which are both easily verifiable.
\begin{customthm}{B}[\cref{thm:naffinenessntruncated} 
 and \cref{obstruction}]
\label{thm:main4}
Let $X$ be a space, and let $\Bbbk$ be a field of characteristic zero. 
\begin{thmenum}
\item\label{thm:notnecessary}
If $X$ is $n$-truncated, then its Betti stack $X_{\operatorname{B}}$ is $n$-affine.
\item Suppose that $\pi_{n+1}(X)$ does not vanish for some choice of a base point in $X$. Then the Betti stack $X_{\operatorname{B}}$ is not  $n$-affine over $\Bbbk$.
\end{thmenum}
\end{customthm}
In \cref{sec:nkoszulcoaffine} we turn our attention to   $\Ebb_n$-Koszul duality. As we explained, one of our main ideas is that 
the \emph{cospectrum} of the coconnective cdga of cochains on $X$, $\CX$, plays a key role in the study of Koszul duality for this class of algebras. 
The  main result of  \cref{sec:nkoszulcoaffine} is a categorification of \eqref{eqqeqintro2}.

\begin{customthm}{C}[\cref{thm:mainkoszuln}]
\label{mainintroE}
Let $n\geqslant 1$ be an integer, let $\Bbbk$ be a field of characteristic zero, and let $X$ be a pointed $(n+1)$-connected space satisfying appropriate finiteness conditions.  Then there is a natural equivalence of $(\infinity,n+1)$-categories 
\begin{equation}
\label{nkoszulintro}
(n+1)\ShvCattwo^{n}(\cSpec(\CX)) \simeq (n+1)\LocSysCattwo^{n}(X;\Bbbk).
\end{equation}
Combining this with the categorified monodromy equivalence  \eqref{catmon}, 
we obtain a commutative square of  equivalences of $(\infinity,n+1)$-categories

\begin{equation}
\label{mdiag2}
\begin{tikzpicture}[scale=0.75,baseline=(current  bounding  box.center)]
\node (a) at(-5,0){$(n+1)\mathbf{Mod}_{n\Modtwo^{n-1}_{\mathrm{C}_{\bullet}(\Omega_*^{n+1}X;\Bbbk)}}{\lp (n+1)\PrLUtwo[n]\rp}$};
\node (b) at (5,0){$(n+1)\ShvCattwo^{n}(\cSpec(\CX))$};
\node (c) at (-5,-2.5){$(n+1)\LocSysCattwo^{n}(X;\Bbbk)$};
\node (d) at (5,-2.5){$(n+1)\ShvCattwo^{n}(X_{\mathrm{B}})$};
\draw[<->,font=\scriptsize] (a) to node[above]{$\simeq$}(b);
\draw[<->,font=\scriptsize] (a) to node[above,rotate=90]{$\simeq$}(c);
\draw[<->,font=\scriptsize] (c) to node[above]{$\simeq$}(d);
\draw[<->,font=\scriptsize] (b) to node[above,rotate=-90]{$\simeq$}(d);
\end{tikzpicture}
\end{equation}
\end{customthm}

Let us explain the notations in the statement of \cref{mainintroE}. If $\mathfrak{X}$ is a stack the $(n+1)$-category $(n+1)\ShvCattwo^{n}(\mathfrak{X})$ is the $(n+1)$-category of quasi-coherent sheaves of (presentably $\Bbbk$-linear) $n$-categories over $S$. This is a categorification of the usual category of quasi-coherent sheaves: when $n=1$, this was defined in \cite{1affineness}, while for arbitrary $n$ it has been recently introduced in \cite{stefanichthesis}. In the bottom right corner of diagram (\ref{mdiag2}), $X_{\mathrm{B}}$ is the Betti stack associated with the space $X$.

\cref{mainintroE} is the main theorem of this article. 
 The top equivalence of diagram (\ref{mdiag2}) is a $n$-fold categorification of equivalence \eqref{eqqeqintro2}. It provides a complete solution to the problem of generalizing the (almost)  Morita equivalence \ref{intro:morita} to the setting of   $\Ebb_n$-Koszul duality in the topological  context.

 \begin{remark}
The appearance of $\cSpec(\CX)$ in the statement might seem only a technical artifact of our approach; on the contrary, we believe that our work clarifies the true nature of Koszul duality in this setting.   
There is a canonical map of stacks
$$
\mathrm{aff}_X\colon X_{\operatorname{B}} \longrightarrow \cSpec(\CX)
$$
called the \emph{affinization map}. Equivalence \eqref{nkoszulintro} is given precisely by the pull back along $\mathrm{aff}_X$. The real content of Koszul duality in this setting is therefore that, under appropriate connectivity assumptions on $X$, the theory of (higher) local systems does not distinguish between $X_{\operatorname{B}}$ and $\cSpec(\CX)$. We believe this to be a more transparent statement already in the classical case $n=1$ where, as we discussed, the standard formulation of the duality between $\CX$ and $\mathrm{C}_{\bullet}(\Omega_*X;\Bbbk)$ requires otherwise artificial size restrictions, or $t$-structure renormalizations.
\end{remark}
We conclude the paper proving some easy consequences  of our results.  
In particular, 
\cref{mainintroE} allows us to test the $n$-affineness of some coaffine stacks by using the criterion of \cref{thm:main4}. This has the  following interesting consequence. 
\begin{customthm}{D}[\cref{cor:1affinenessga}]
\label{thm:introD}
For $k\geqslant1$ and $\Bbbk$ a field of characteristic zero, the stack $\mathbf{B}^{n+2k}\Ga$ is $n$-affine, while the stack $\mathbf{B}^{n+2k+1}\Ga$ is not $n$-affine.
\end{customthm}
When $n=1$, \cref{thm:introD} generalizes \cite[Theorem 2.5.7]{1affineness}, which proved the $1$-affineness of $\mathbf{B}\Ga$, $\mathbf{B}^2\Ga$ and $\mathbf{B}^3\Ga$ and proved the non-$1$-affineness of $\mathbf{B}^4\Ga.$
\subsection{Future directions}
While \cref{thm:mainkoszuln} provides a clean and complete answer to the problem of Koszul duality between (categorified) modules over $\Ebb_k$-Koszul dual algebras coming from algebraic topology, some questions are left open. 
\begin{question}
What can be said when $\Bbbk$ is a field of characteristic $p>0$?
\end{question}
In some ways, the answer to this question is subtle. Up to replacing $\Einf$-rings with derived commutative rings (in the sense of \cite{raksit2020}), the main characters appearing in \cref{thm:mainkoszuln} make sense in this setting as well. Moreover, When $k=0$ or $k\geqslant2$ the classical statement and/or our techniques still apply to this setting as well.

However, when $k=1$, our proof breaks down: the main technical input is \cref{lemma:locfullyfaithful}, which heavily relies on rational homotopy theory and the structure of rational H-spaces. It is not clear whether this issue can be solved using tools from $p$-adic homotopy theory developed in \cite{mandell_p_adic}: in any case, this will require to move to an algebraic or separable closure of $\mathbb{F}_p$.

On another note, one could be tempted to check whether it is possible to obtain a similar description of Koszul duality between modules for \textit{arbitrary} $\Ebb_k$-algebras, and not only those coming from algebraic topology. Namely:
\begin{question}
Let $\Bbbk$ be a field of characteristic zero, let $A$ be a $\Ebb_k$-Koszul $\Bbbk$-algebra with $\Ebb_k$-Koszul dual $A^{!_{\Ebb_k}}$ satisfying suitable finiteness assumption. Can one relate the $k$-categories of iterated modules on $A$ and $A^{!_{\Ebb_k}}$?
\end{question}
If $A$ is a $(k,\infinity)$-bialgebra (that is, it is a $\Ebb_k$-algebra inside the category of cocommutative coalgebras in $\Mod_{\Bbbk}$), then its Koszul dual $A^{!_{\Ebb_k}}$ is actually the $\Bbbk$-linear dual of the $\Ebb_k$-Koszul dual $\Einf$-coalgebra $\smash{\widetilde{A}}^{!_{\Ebb_k}}$ of $A$; in particular, it is still a $\Einf$-algebra. So, with the exception of the usual \cref{lemma:locfullyfaithful}, our techniques cover this problem as well.

However, to our knowledge, there is no known categorification of the category of comodules over a $\Ebb_k$-coalgebra, nor can our statement be reproduced \textit{verbatim} in this setting. Indeed, arbitrary $\Ebb_k$-algebras and coalgebras cannot be embedded naturally in the category of derived stacks, so the formalism of derived algebraic geometry cannot be applied straight-forwardly.
}
 
\subsection*{Notations and conventions}
\begin{itemize}
\item We will use throughout the language of $(\infinity,1)$-categories and higher homotopical algebra, as developed in \cite{htt,ha}, from which we borrow most of the notations and conventions.
\item Our work  relies on intrinsically derived and homotopical concepts. So we shall simply write ``limits'', ``colimits'', ``tensor product'', suppressing adjectives such as ``homotopy'' or ``derived'' in our notations. Similarly, we shall simply write ``categories'' instead of ``$\infinitone$-categories'', and ``$n$-categories'' instead of ``$(\infinity,n)$-categories''. 
\item We will work with \textit{local systems} and \textit{sheaves} of categories, and it will be important pay attention to size issues. We fix a sequence of nested universes $\scrU \in\scrV\in\mathscr{W}\in\ldots$. We shall say that a category $\scrC$ is \textit{small} if it is $\scrU$-small, that $\scrC$ is \textit{large} if it is $\scrV$-small without being $\scrU$-small, that $\scrC$ is \textit{very large} if it is $\mathscr{W}$-small without being $\scrV$-small, and that $\scrC$ is \textit{huge} if it is not even $\mathscr{W}$-small. When dealing with categories of (possibly decorated) categories, we shall adopt the following notations in order to distinguish the size: large categories of categories will be denoted with a normal font; very large categories of categories will be denoted with $\smash{\widehat{(-)}}$; huge categories of categories will be denoted with $\smash{\widehat{(-)}}$ and capital letters.\\
For example, $\CatUinfty$ is the large category of small categories, while $\CatVinfty$ is the very large category of large categories, and $\CatWinfty$ is the huge category of very large categories.
\item We shall denote the large category of small spaces by $\scrS$. In particular, by \textit{space} we always mean \textit{small space}.
\item The large category $\PrLU$ of large presentable categories and the very large category of $\CatVinftycocom$ of large cocomplete categories are both symmetric monoidal categories: $\Ebb_k$-algebras inside $\PrLU$ and $\CatVinftycocom$ are (respectively) presentable and cocomplete categories endowed with an $\Ebb_k$-monoidal structure that commutes with colimits separately in each variable. In order to streamline our notations, in the rest of our paper we shall refer to an $\Ebb_k$-algebra in $\PrLU$ as a \textit{presentably $\Ebb_k$-monoidal category}, and to an $\Ebb_k$-algebra in $\CatVinftycocom$ as a \textit{cocompletely $\Ebb_k$-monoidal category}; in the case $k=\infinity$ we shall simply write \textit{symmetric monoidal} in place of $\Einf$-monoidal.
\item In a similar fashion, for any $k\in\NN_{\geqslant1}\cup\left\{\infinity\right\}$ and a cocompletely (resp. presentably) $\Ebb_k$-monoidal \infinity-category $\scrA$, we shall say that a category $\scrC$ is \textit{cocompletely} (resp. \textit{presentably}) left tensored over $\scrA$ if it is a left $\scrA$-module in $\CatVinftycocom$ (resp. in $\PrLU$). This formula amounts to the datum of a cocomplete (or presentable) category $\scrC$ which is left tensored over $\scrA$ in such a way that the tensor action functor commutes with colimits separately in each argument.
\item We shall deal with higher (i.e., $n$-)categories, and in particular with $(n+1)$-categories of (possibly decorated) $n$-categories: we follow the conventions introduced and adopted in \cite{Pascaleff_Pavia_Sibilla_Local_Systems}. We shall denote $n$-categories with a bold font, and in order to avoid confusion concerning the ``categorical height'' we are working at, we shall adopt the following highly non-standard notation as well: if we want to refer to the (very large) higher category of large $m$-categories seen as a $n$-category, we shall write $n\CatVtwoinfty[m]$. In the particular case $n=1$, we shall drop both the bold font and the $1$ before our notations, and simply write $\CatVinfty[m]$. For example, $3\CatVtwoinfty[2]$ is the very large $3$-category of all large $2$-categories, while $2\CatVtwoinfty[2]$ is its underlying $2$-category, and $\CatVinfty[2]$ is its underlying $1$-category.
\item Most of the times we will consider categories which are enriched over some preferred category (e.g., modules in spectra which are enriched over themselves, or presentably enriched categories which are enriched over themselves, and so forth). At the same time, we will need to consider the underlying spaces of maps between objects in such categories. For this reason, when $\scrC$ is enriched over a category $\scrA$, we will denote as $\Map_{\scrC}(-,-)$ the space of maps in $\scrC$, and as $\Mapin_{\scrC}(-,-)$ 
the morphism object of $\scrA$ providing the enrichment, so as to to highlight whether we are seeing a morphism object as a space or as something more structured. If $\scrC$ is a higher category of categories (e.g., $\scrC=\CatVinftycocom$ or $\scrC=\PrLU$) we will also use $\Funin(-,-)$, possibly with decorations, to mean the category of structure-preserving functors which serves as the internal object of morphisms in $\scrC$.
\end{itemize}
\subsection*{Acknowledgements}
This project took quite a long time to complete, and along the way we benefited enormously from conversations and email exchanges with friends and colleagues that helped us navigate the many subtle issues in the theory of Koszul duality. We owe special thanks to Francesco Battistoni, Ivan Di Liberti, John Francis, Andrea Gagna, Yonatan Harpaz, David Nadler, Guglielmo Nocera, Mauro Porta, Pavel Safronov, German Stefanich. We are deeply grateful to an anonymous referee for their  many useful comments and for highlighting a gap in the proof of \cref{lemma:locfullyfaithful}.\\
The second author wishes to dedicate this work to the loving memory of his father, who passed away during the completion of this project.
\addtocontents{toc}{\protect\setcounter{tocdepth}{2}}
\section{Main setup}
\label{mainsetup}
In this section we collect the main constructions and definitions concerning presentable $n$-categories (\cite{stefanich2020presentable}) and $n$-affineness of stacks (\cite{1affineness,stefanichthesis}). 
We discuss the main properties of presentable $n$-categories and of local systems of presentable $n$-categories over spaces in \cref{subsec:higher_cat}; we will then focus on the concept of $n$-affineness in \cref{subsec:n_affineness}. 

\subsection{Miscellanea on higher categories}
\label{subsec:higher_cat}
In this section, we fix our notations and collect the main results established in \cite{Pascaleff_Pavia_Sibilla_Local_Systems} concerning presentable $n$-categories and local systems of presentable $n$-categories. 
We try to avoid the subtler technicalities and try to convey the main ideas -- that is, that presentable $n$-categories provide the right generalization of the concept of presentable categories to the $n$-categorical setting, and that local systems of presentable $n$-categories are controlled by monodromy data just like ordinary local systems.

In what follows, $\kappa_0$ is the first large cardinal of our theory (i.e., the cardinal such that $\kappa_0$-small objects are precisely ordinary small objects).
\begin{parag}
Let $\scrA$ be a presentably symmetric monoidal category. Following Stefanich \cite{stefanich2020presentable}, we can define for any $n\geqslant1$ an $(n+1)$-category of presentable $\scrA$-linear $n$-categories that we denote as $(n+1)\LinAPrLUtwo[n]$. Its underlying category is explicitly constructed as follows: if $\CatVinftycocom[n]$ denotes the category of cocomplete $n$-categories, then the underlying category $\LinAPrLU[n]$ is the sub-category of $\kappa_0$-compact objects inside the cocomplete category of iterated modules over $\scrA$ inside $\CatVinftycocom[n]$. More explicitly:
\[
\LinAPrLU[n]\coloneqq\Mod_{n\LinAPrLUtwotiny[n-1]}{\lp\CatVinftycocom[n]\rp}^{\kappa_0}.
\]
Using the fact that $n\LinAPrLUtwo[n-1]$ is itself an $n$-category, we can enhance  $\LinAPrLU[n]$ to an  $(n+1)$-category, which we denote $(n+1)\LinAPrLUtwo[n]$.\\

When $\scrA=\Mod_{\Bbbk}$ is the category of modules over a commutative ring spectrum $\Bbbk$, we shall simply write $(n+1)\LinkPrLUtwo[n]$.
\end{parag}
\begin{fact}[{\cite[Section $5$]{stefanich2020presentable}}]
\label{fact:properties}
The $(n+1)$-category $(n+1)\LinAPrLUtwo[n]$ enjoys the following properties.
\begin{factenum}
\item \label{fact:prl2=prl}For $n=1$, this is just the $2$-categorical incarnation of the category of $\scrA$-modules inside the category $\PrLU$ of presentable categories. In particular, this agrees with the usual notion of presentably $\scrA$-linear categories.
\item\label{fact:prln_symm_monoidal}For all $n\geqslant1$, the $(n+1)$-category $(n+1)\LinAPrLUtwo[n]$ is a commutative algebra object inside the symmetric monoidal category $\smash{\widehat{\mathrm{Cat}}}_{(\scriptstyle\infty,n+1)}^{\mathrm{rex}}$. 
This means that it admits all colimits and it is equipped with a symmetric monoidal structure which preserves them separately in each variable. The unit of such symmetric monoidal structure is provided by the presentably $\scrA$-linear $n$-category $n\LinAPrLUtwo[n-1]$.
\item\label{fact:limits_of_presentable_n} It admits all limits of left adjointable diagrams. That is, given a diagram of presentably $\scrA$-linear $n$-categories
\[
I\longrightarrow(n+1)\LinAPrLUtwo[n]
\]
such that every arrow $i\to j$ in $I$ yields an $n$-functor $n\bm{\scrC}(i)\to n\bm{\scrC}(j)$ which admits a left adjoint, then the limit of such diagram exists and agrees with the colimit of opposite diagram
\[
I^{\op}\longrightarrow(n+1)\LinAPrLUtwo[n]
\]
where each $n$-functor is replaced with its left adjoint.
\item \label{fact:presentable_n_fold_modules}For $A$ an $\Ebb_k$-algebra object in a presentably symmetric monoidal category $\scrA$, and for a presentable category $\scrC$ tensored over $\scrA$, for all integers $0\leqslant n\leqslant k-1$ there exists an $(n+1)$-category of $n$-fold categorical $A$-modules inside $\scrC$ inductively defined as
\[
(n+1)\mathbf{LMod}^n_A(\scrC)\coloneqq(n+1)\LModtwo_{n\LModtwo^{n-1}_A(\scrC)}{\lp(n+1)\LinAPrLUtwo[n]\rp}.
\]
When $n=0$, we shall interpret $1\LModtwo^0_A(\scrC)$ to be the ordinary category $\LMod_A(\scrC)$ of left $A$-modules inside $\scrC$. For all integers $n$ as above, this is again a presentable $\scrA$-linear $(n+1)$-category.\\
When $A$ is a $\Ebb_k$-ring spectrum and $\scrA=\scrC=\Sp$ is the category of spectra, we shall simply write $(n+1)\LModtwo^n_A$ instead of $(n+1)\LModtwo^n_A(\Sp)$.
\end{factenum}
\end{fact}
\begin{parag}
Recall now from \cite[Section $1$]{Pascaleff_Pavia_Sibilla_Local_Systems} that for any small space $X$ and for any cocomplete category $\scrC$ we have a well defined category of \textit{$\scrC$-valued local systems on $X$}
\[
\LocSys(X;\scrC)\coloneqq\Fun(X,\hsp\scrC),
\]
which is equivalent to the category of hypersheaves over $X$ which are locally hyperconstant (\cite{haineportateyssier}). Applying this machinery to the cocomplete category $(n+1)\LinAPrLUtwo[n]$, we obtain an $(n+1)$-category of local systems of presentable $\scrA$-linear $n$-categories that we denote as
\[
(n+1)\LocSysCattwo^n(X;\scrA)\coloneqq(n+1)\Funtwo{\lp X,\hsp (n+1)\LinAPrLUtwo[n]\rp}.
\]
When $\scrA\coloneqq\Mod_{\Bbbk}$ is the category of modules over some commutative ring spectrum $\Bbbk$, we shall denote such $(n+1)$-category simply as $(n+1)\LocSysCattwo^n(X;\Bbbk)$.
\end{parag}
\begin{parag}
Whenever a category $\scrA$ is cocomplete, it is naturally tensored over the category of spaces: the action of a topological monoid $G$ on an object $A$ of $\scrA$ is expressed through the colimit of the diagram of shape $G$ and constant value in $A$. It is known that for a presentable category $\scrC$, local systems over a connected space $X$ with values in $\scrC$ are equivalently described as $\Omega_*X$-modules inside $\scrC$. This can be generalized to  cocomplete categories and yields the equivalence
\begin{align}
\label{lemma:koszulduality}
\LocSys(X;\scrC)\simeq\LMod_{\Omega_*X}(\scrC).
\end{align} 
for any cocomplete category $\scrC$. 
\end{parag}
The main result of \cite{Pascaleff_Pavia_Sibilla_Local_Systems} relates $n$-categorical local systems on an $n$-connected space $X$ and monodromy data, expressed as an $(n+1)$-fold $\Omega_*^{n+1}X$-module structure on the fiber.
\begin{theorem}[{\cite[Theorem $3.2.24$]{Pascaleff_Pavia_Sibilla_Local_Systems}}]
	\label{conj:infinityn}
Let $n\geqslant 1$ be an integer, let $X$ be a pointed $n$-connected space (i.e., $\pi_k(X)\cong 0$ for every $k\leqslant n$), and let $\scrA$ be a presentably symmetric monoidal category. Then there exist equivalences of $(n+1)$-categories\begin{align*}
(n+1)\LocSysCattwo^n(X;\scrA)&\simeq(n+1)\mathbf{LMod}_{\Omega_*X}{\lp(n+1){\mathbf{Lin}_{\scrA}\mathbf{Pr}^{\mathrm{L}}}_{({\scriptstyle\infty}, n)}\rp}\\&\simeq
(n+1)\mathbf{LMod}^n_{\Omega_*^{n+1}X}(\scrA).
\end{align*}
\end{theorem}
\begin{remark}[{\cite[Lemma $2.2$]{Pascaleff_Pavia_Sibilla_Local_Systems}}]
\label{lemma:rectificationofmodules}
When the category $\scrA$ is cocompletely symmetric monoidal with unit $\boldone_{\scrA}$ then the action of the space $G$ can be expressed in terms of an algebra inside $\scrA$. Indeed, an object $A$ of $\scrA$ is a left $G$-module if and only if it is a left $(G\otimes\boldone_{\scrA})$-module, where $\otimes$ denotes the cocontinuous action of $\scrS$ over $\scrA$.

For example, when $\scrA=\Mod_{\Bbbk}$ is the category of $\Bbbk$-modules over a commutative ring spectrum $\Bbbk$, then the $\Bbbk$-algebra $G\otimes \Bbbk$ agrees with the $\Bbbk$-algebra of $\Bbbk$-valued chains $\mathrm{C}_{\bullet}(G;\Bbbk)$ with its Pontrjagin product. Therefore, a left $G$-module structure on a $\Bbbk$-module $M$ is the same as a $\mathrm{C}_{\bullet}(G;\Bbbk)$-module structure.

When $\scrA=\LinkPrLU[n]$ is the category of presentable $\Bbbk$-linear $n$-categories, then the presentable monoidal $n$-category $G\otimes n\LinAPrLUtwo[n-1]$ is equivalent to the $n$-category $n\LocSysCattwo^{n-1}(G;\Bbbk)$ of presentable $\scrA$-linear local systems over $G$, equipped with the Day convolution tensor product (\cite[Lemma $3.2.28$]{Pascaleff_Pavia_Sibilla_Local_Systems}).

Applying these observations to \cref{conj:infinityn}, we see that when $X$ is $n$-connected then one can express local systems of presentable $\Bbbk$-linear $n$-categories over $X$ as presentable $\mathrm{C}_{\bullet}(\Omega^{n+1}_*X;\Bbbk)$-linear $n$-categories.
\end{remark}
For future reference, we also report the following key result on the ambidexterity of limits and colimits of presentable $n$-categories indexed over groupoids/spaces.
\begin{lemma}[{\cite[Lemma $3.2.29$]{Pascaleff_Pavia_Sibilla_Local_Systems}}]
\label{lemma:limitsandcolimitsovergroupoids}
Let $\scrA$ be a presentably symmetric monoidal category, let $X$ be a space, and let $F\colon X\to(n+1)\LinAPrLUtwo[n]$ be a diagram of presentable $\scrA$-linear $n$-categories of shape $X$. Then, there is a natural equivalence $\lim F\simeq\colim F$ in $(n+1)\LinAPrLUtwo[n]$.
\end{lemma}
\subsection{The notion of \texorpdfstring{$n$}{n}-affineness}
\label{subsec:n_affineness}
Here we introduce the notion of $n$-affineness for arbitrary (pre)stacks. We shall use our results in this section as a stepping stone for our study of general $n$-affineness properties of Betti stacks in \cref{sec:naffineness1} below.   We refer the reader to the Introduction for a thorough discussion of the relationship between $n$-affineness and higher Koszul duality.  Before introducing the general case for arbitrary $n$, we start by reviewing  basic definitions and results established in \cite{1affineness}.

\begin{construction}[{\cite[Section $1.1$]{1affineness}}]
	\label{constr:1affine}
Let $\Bbbk$ be an $\Einf$-ring spectrum, and denote by $\operatorname{Aff}_{\Bbbk}$ the category of affine schemes over $\Bbbk$, i.e., the opposite category of the category $\smash{\CAlg_{\Bbbk}^{\geqslant0}}$ of connective and commutative $\Bbbk$-algebras. Let $\mathrm{PSt}_{\Bbbk}$ be the category of \textit{prestacks over $\Bbbk$}, i.e., the category of accessible presheaves over the category $\operatorname{Aff}_{\Bbbk}.$ The functor$$\operatorname{ShvCat}\colon\mathrm{PSt}_{\Bbbk}\longrightarrow\operatorname{Lin}_{\Bbbk}\CatVinftycocom$$is by definition the right Kan extension of the functor$$\operatorname{Lin}_{(-)}\PrLU\colon\operatorname{Aff}_{\Bbbk}^{\op}\simeq\CAlg^{\geqslant0}_{\Bbbk}\longrightarrow\operatorname{Lin}_{\Bbbk}\CatVinftycocom$$along the Yoneda embedding $\yo\colon\operatorname{Aff}_{\Bbbk}^{\op}\to\operatorname{PSt}^{\op}_{\Bbbk}$. 
\begin{defn}
\label{def:shvcat}
Let $\mathscr{X}$ be a prestack. Then the category $\operatorname{ShvCat}(\mathscr{X})$ is the \textit{category of quasi-coherent sheaves of ($\Bbbk$-linear) categories over $\mathscr{X}$}. 
\end{defn}

If $\mathscr{F}$ is a quasi-coherent sheaf of categories over $\mathscr{X}$, we have a well defined functor$$\Gamma{\lp-,\hsp\mathscr{F}\rp}\colon\lp\operatorname{Aff}_{\Bbbk/\mathscr{X}}\rp^{\op}\longrightarrow\LinkPrLU$$which we can right Kan extend to get the functor$$\Gamma{\lp-,\hsp\mathscr{F}\rp}\colon\operatorname{PSt}_{\Bbbk/\mathscr{X}}^{\op}\longrightarrow\LinkPrLU.$$By fixing the prestack to be $\mathscr{X}$ itself, for any quasi-coherent sheaf of categories over $\mathscr{X}$ the $\Bbbk$-linear category of its global section is actually acted on by the stable category $\Qcoh(\mathscr{X})$. Hence, we deduce the existence of a \textit{global section functor}
	\begin{align}
	\label{functor:globalsections}
	\Gamma^{\operatorname{enh}}(\mathscr{X},-)\colon\operatorname{ShvCat}(\mathscr{X})\longrightarrow\operatorname{Lin}_{\Qcoh(\mathscr{X})}\PrLU
	\end{align} which is right adjoint to the sheafification functor
\begin{align}
\label{functor:sheafification}
\operatorname{Loc}_{\mathscr{X}}\colon\operatorname{Lin}_{\Qcoh(\mathscr{X})}\PrLU\longrightarrow\operatorname{ShvCat}(\mathscr{X}).
\end{align}
The latter acts on objects by sending a presentably $\Qcoh(\mathscr{X})$-linear category $\scrC$ to the quasi-coherent sheaf of categories obtained by sheafifying the assignment  $$\Spec(S)\mapsto \Qcoh(S)\otimes_{\Qcoh(\mathscr{Y})}\scrC.$$
\end{construction}

\begin{defn}[{\cite[Definition $1.3.7$]{1affineness}}]
	\label{def:1affine}
A prestack $\mathscr{X}$ is \textit{$1$-affine} if $\Gamma^{\operatorname{enh}}(\mathscr{X},-)$ and $\operatorname{Loc}_{\mathscr{X}}$ are mutually inverse equivalences.
\end{defn}
\begin{remark}
\label{remark:on0affineness}
\cref{def:1affine} has to be interpreted as a generalization of affineness, in the following sense. When $X=\Spec (R)$ is an affine scheme, then there is a canonical equivalence of stable categories
\begin{align}
\label{eq:0affine}
\QCoh(X) \simeq \Mod_{\Gamma(X,\scrO_X)}.
\end{align}
In \cite{1affineness}, stacks for which the equivalence \eqref{eq:0affine} holds are called \textit{weakly $0$-affine}. Let us remark that, actually, the class of weakly $0$-affine stacks (in this sense) sits between the class of affine schemes and an even weaker notion of $0$-affineness. Indeed, if $\scrX$ is an arbitrary stack one could define a notion of $0$-affineness by asking that the global sections functor
\[
\Gamma(\scrX,-)\colon\QCoh(\scrX)\longrightarrow\Mod_{\Bbbk}
\]
is monadic: let us call a stack $\scrX$ that satisfies this condition    \emph{almost 0-affine}. If $\scrX$ is a weakly $0$-affine stack in the sense of Gaitsgory, then the global sections are trivially monadic over $\Mod_{\Bbbk}$: in virtue of the Schwede--Shipley recognition principle for stable categories of modules in spectra (\cite[Proposition $7.1.2.6$]{ha}), if $\scrX$ is weakly $0$-affine  then $\scrO_{\scrX}$ has to be a compact generator of $\QCoh(\scrX)$. This means that the functor
\[\Gamma(\scrX,-)\simeq\Mapin_{\QCoh(\scrX)}{\lp\scrO_{\scrX},-\rp}\]
has to reflect equivalences and preserve all colimits. However, this latter condition is \textit{weaker}. for $\Gamma(\scrX,-)$ to be monadic it is sufficient that it reflects equivalences and preserves only a special class of colimits -- namely, colimits of $\Gamma(\scrX,-)$-split simplicial objects (\cite[Theorem $4.7.3.5$]{ha}). This implies that if $\scrX$ is almost $0$-affine then $\scrO_{\scrX}$ has to be a generator (although this is not a sufficient condition): but it might very well fail to be a \emph{compact} generator. 

 A particularly easy example is the following: take $\scrC$ to be a countable product of copies of $\Mod_{\Bbbk}$, which can be interpreted as the category of local systems over the discrete space $\ZZ$. \cref{lemma:connected} implies that the functor of global sections is monadic, yet the structure sheaf $\scrO_{\ZZ}$ (which corresponds to the constant sequence $(\Bbbk)_{n\in\ZZ}$) is not compact. Indeed,  the global  sections functor is equivalent to the functor
\[
(M_n)_{n\in\ZZ}\mapsto\Mapin_{\scrC}(\scrO_{\ZZ},M_n)\simeq\prod_{n\in\ZZ}\Mapin_{\Bbbk}(\Bbbk,M_n)\simeq\prod_{n\in\ZZ}M_n,
\]
and in general infinite colimits do not commute with infinite products in stable categories. Another non-trivial example of the difference between these two notions is provided by $\CC\PP^{\scriptstyle\infty}$ (\cref{cor:BCPinfty}). 

This issue  persists in the categorified setting. This means that \cref{def:1affine} has to be interpreted as a ``strong'' notion of $1$-affineness, as is a direct categorification of  Gaitsgory's  notion of $0$-affineness. By the same token,  we could define \textit{almost} $1$-affineness by requiring the mere monadicity of the global sections functor. These two notions are, in general, genuinely different. However, as we shall explain in \cref{porism:conservativity} below, for Betti stacks the situation is  simpler,  as almost $1$-affineness implies $1$-affineness. This will entail a significant simplification of some of our arguments.  
\end{remark}
We now present the general definition of $n$-affineness for arbitrary $n$. For any $n\geqslant 0$, we can construct a functor\begin{align*}(n+1)\Modtwo_{(-)}^n\colon\mathrm{Aff}_{\Bbbk}^{\op}\simeq\CAlg_{\Bbbk}^{\geqslant0}&\longrightarrow \LinkPrLU[n+1]\\\Spec(R)&\mapsto (n+1)\Modtwo^{n}_R\end{align*}sending an affine scheme $\Spec(R)$ to its $(n+1)$-category of $n$-fold $R$-modules. When $n\geqslant1$, this is another name for the $(n+1)$-category of $R$-linear presentable $n$-categories of \cref{fact:presentable_n_fold_modules}.
\begin{defn}[{\cite[Definition $14.2.4$]{stefanichthesis}}]
\label{def:nqcoh}
Let $\scrX$ be a prestack defined over a commutative ring spectrum $\Bbbk$, and let $n\geqslant1$ be an integer. The \textit{$(n+1)$-category $(n+1)\ShvCattwo^n(\scrX)$ of quasi-coherent sheaves of ($\Bbbk$-linear presentable) $n$-categories} over $\scrX$ is defined as the right Kan extension of the functor $(n+1)\Modtwo^{n}_{(-)}$ along the inclusion $\mathrm{Aff}^{\op}_{\Bbbk}\subseteq\PSt_{\Bbbk}^{\op}.$
\end{defn}
This means that $(n+1)\ShvCattwo^n(\scrX)$ is the limit computed inside $\LinkPrLUtwo[n+1]$ $$(n+1)\ShvCattwo^n(\scrX)\coloneqq\lim_{\substack{\Spec(R)\to\scrX\\ R\in\CAlg^{\geqslant0}_{\Bbbk}}}(n+1)\Modtwo^{n}_R\simeq\lim_{\substack{\Spec(R)\to\scrX\\ R\in\CAlg^{\geqslant0}_{\Bbbk}}}(n+1)\mathbf{Lin}_R\PrLUtwo[n].$$In particular, \cref{def:nqcoh} agrees with \cref{constr:1affine} when $n=1$.
\begin{notation}
\label{notation:when_n_is_zero}
Let $\scrX$ be any stack over $\Bbbk$. In order to  streamline our exposition, in the following it will be convenient to make sense of the notation $(n+1)\ShvCattwo^n(\scrX)$ also for $n=0$. In this case, it will be understood as $\QCoh(\scrX)$.
\end{notation}
Just like for the case $n=1$, we have a naturally defined global sections $(n+1)$-functor$$(n+1)\boldsymbol{\Gamma}(\scrX,-)\colon (n+1)\ShvCattwo^n(\scrX)\longrightarrow (n+1)\LinkPrLUtwo[n]$$which for any quasi-coherent sheaf of $n$-categories $n\bm{\scrF}$ over $\scrX$ computes the limit over all local sections$$\lim_{\substack{\Spec(R)\to\scrX\\ R\in\CAlg^{\geqslant0}_{\Bbbk}}}(n+1)\boldsymbol{\Gamma}(\Spec(R),n\bm{\scrF})$$inside $(n+1)\LinkPrLUtwo[n]$.
\begin{defn}
\label{def:naffineness}
Let $\scrX$ be a prestack, and let $n\geqslant1$ be an integer. We say that $\scrX$ is \textit{$n$-affine} if the global sections $(n+1)$-functor$$(n+1)\boldsymbol{\Gamma}(\mathscr{X},-)\colon (n+1)\ShvCattwo^n(\scrX)\longrightarrow (n+1)\LinkPrLUtwo[n]$$is a monadic morphism in the $(\infinity,2)$-category $(n+2)\LinkPrLUtwo[n+1]$ of $\Bbbk$-linear presentable $(n+1)$-categories.
\end{defn}
\begin{remark}
\label{remark:shvcatleftrightkan}
For $n\geqslant2$, the right Kan extension defining the $(n+1)$-category of quasi-coherent sheaves of $n$-categories in \cref{def:nqcoh} is also a \textit{left} Kan extension. More generally: let $f\colon\scrX\to\scrY$ be a morphism of prestacks over a commutative ring spectrum $\Bbbk$, and let $n\geqslant2$ be an integer. Then the pullback functor $f^*\colon\ShvCat^n(\scrY)\to\ShvCat^n(\scrX)$ is part of an ambidextrous adjunction (\cite[Corollary $14.2.10$]{stefanichthesis}). In particular, \cite[Theorem $5.5.14$]{stefanich2020presentable} implies that the limit inside $(n+2)\LinkPrLUtwo[n+1]$ along the pullback $(n+1)$-functors corresponds to the colimit inside $(n+2)\LinkPrLUtwo[n+1]$ along the colimit-preserving pushforward $(n+1)$-functors. This also holds for $n=1$ if the morphism $f$ is assumed to be affine schematic.

So, let $n\geqslant2$ and $f$ be as above. Then the $(n+1)$-functor $f_\ast$ sends a quasi-coherent sheaf of $n$-categories $n\bm{\scrF}$ over $\scrX$ to the quasi-coherent sheaf of $n$-categories over $\scrY$ whose local sections on an affine scheme $\Spec(R)$ over $\scrY$ are described by the formula
\[
n\bm{\Gamma}(\Spec(R),f_\ast(n\bm{\scrF}))\simeq\lim_{\substack{\Spec(S)\to\Spec(R)\times_{\scrY}\scrX}}n\bm{\Gamma}(\Spec(S),n\bm{\scrF}).
\]
Such limit is computed along pullback $(n+1)$-functors, which are both right and left adjoint to the pushforward $(n+1)$-functors. So, it can be equivalently computed as the \textit{colimit} along the latter, i.e.,
\[
n\bm{\Gamma}(\Spec(R),f_\ast(n\bm{\scrF}))\simeq\colim_{\substack{\Spec(S)\to\Spec(R)\times_{\scrY}\scrX}}n\bm{\Gamma}(\Spec(S),n\bm{\scrF}).
\]
Both the above limit and colimit are computed inside $(n+1)\LinkPrLUtwo[n].$
\end{remark}
\begin{remark}
\label{remark:monadicity}
Thanks to \cite[Proposition $14.3.6$]{stefanichthesis}, we know that for $n\geqslant2$ the monadicity requirement for $n$-affineness can be checked at the level of the underlying categories -- i.e., the $(n+1)$-functor $(n+1)\boldsymbol{\Gamma}(\scrX,-)$ is a monadic morphism of presentable $\Bbbk$-linear $(n+1)$-categories if and only if the underlying functor of ordinary categories$$\Gamma(\scrX,-)\colon\ShvCat^n(\scrX)\longrightarrow\LinkPrLU[n]$$is monadic. If $n=1$, \cref{def:naffineness} recovers \cref{def:1affine} (\cite[Remark $14.3.8$]{stefanichthesis}). As explained in \cref{remark:on0affineness}, this is is a stronger requirement than asking simply for the monadicity of the global sections functor: indeed, this is equivalent to asking for it to be a \textit{colimit-preserving} monadic right adjoint.

The weak and strong notion of $n$-affineness  in fact coincide as soon as  $n\geqslant 2$. That is, when $n \geqslant 2$, if  the global sections functor is monadic, it is automatically a monadic morphism in  $(n+2)\LinkPrLUtwo[n+1]$.  Hence, it induces a $\Bbbk$-linear equivalence of presentable $(n+1)$-categories
\[
(n+1)\ShvCattwo^n(\scrX)\simeq(n+1)\Modtwo_{n\ShvCattwo^{n-1}(\scrX)}{\lp(n+1)\PrLUtwo[n]\rp}.
\] 
This follows immediately from the fact that pullbacks and pushforwards between presentable $(n+1)$-categories of quasi-coherent sheaves of $n$-categories form an \textit{ambidextrous} adjunction in $(n+2)\LinkPrLUtwo[n+1]$, as already mentioned in \cref{remark:shvcatleftrightkan}. 
\end{remark}
\section{Betti stacks and \texorpdfstring{$n$}{n}-affineness}
\label{sec:naffineness1}
\subsection{Betti stacks and their sheaves of categories}
\begin{notation}
In this section, unless otherwise specified, $\Bbbk$ always denotes an arbitrary connective commutative ($\Einf$-)ring spectrum. 
\end{notation}
Let $\operatorname{St}_{\Bbbk}$ be the category of \textit{stacks over $\Bbbk$}, i.e., the full subcategory of $\operatorname{PSt}_{\Bbbk}$ spanned by those prestacks which are hypercomplete sheaves with respect to the étale topology over $\operatorname{Aff}_{\Bbbk}$. The natural functor $\operatorname{Aff}_{\Bbbk}\to\left\{*\right\}$ induces a functor$$(-)_{\mathrm{B}}\colon\operatorname{Shv}(\left\{*\right\})\simeq\scrS\longrightarrow\operatorname{St}_{\Bbbk}$$which corresponds to sending a space $X$ to the sheafification of the constant prestack$$\operatorname{Aff}_{\Bbbk}^{\op}\longrightarrow\left\{*\right\}\overset{X}{\longrightarrow}\scrS.$$
\begin{defn}
	\label{def:bettistack}
The functor $(-)_{\mathrm{B}}\colon\scrS\to\operatorname{St}_{\Bbbk}$ is the \textit{Betti stack functor}. 
\end{defn}
\begin{parag}
	\label{parag:qcohbetti}
Betti stacks are intimately linked to the theory of local systems over spaces. Indeed, for every space $X$ we can consider the category $\Qcoh{\lp X_{\mathrm{B}}\rp}$ of quasi-coherent sheaves over its associated Betti stack, and for every affine scheme $\Spec(R)$ over $\Bbbk$ one has a symmetric monoidal equivalence of stable categories$$\Qcoh{\lp X_{\mathrm{B}}\times\Spec(R)\rp}\simeq\LC(X;R),$$where on the right hand side we are considering the point-wise tensor product, as proved in \cite[Proposition $3.1.1$]{portasala}. In particular, taking $R$ to be $\Bbbk$ in the above formula yields$$\Qcoh{\lp X_{\mathrm{B}}\times\Spec(\Bbbk)\rp}\simeq\Qcoh{\lp X_{\mathrm{B}}\rp}\simeq\LC(X;\Bbbk).$$Analogously, at a categorified level, it turns out that quasi-coherent sheaves of categories over $X_{\mathrm{B}}$ recover categorical local systems over $X$.
\end{parag}
\begin{lemma}\label{lemma:shvcatn}
Let $n\geqslant1$. For any space $X$ and for any base commutative ring spectrum $\Bbbk$, we have an equivalence of $(n+1)$-categories$$(n+1)\ShvCattwo^n{\lp X_{\mathrm{B}}\rp}\simeq(n+1)\LocSysCattwo^n{\lp X;\Bbbk\rp}.$$
\end{lemma}
\begin{proof}
The functor $(-)_{\mathrm{B}}\colon\scrS\to\operatorname{St}_{\Bbbk}$ is the pullback functor in a geometric morphism between categories of sheaves, hence it obviously commutes with colimits and finite limits. Presenting $X$ as a colimit of its contractible cells, we have equivalences$$X_{\mathrm{B}}\simeq\colim_{\left\{*\right\}\to X} \left\{*\right\}_{\mathrm{B}}\simeq\colim_{\Spec(\Bbbk)\to X_{\mathrm{B}}}\Spec(\Bbbk).$$Since the functor $(n+1)\ShvCattwo^n(-)$ is a sheaf for the étale topology (\cite[Theorem $1.5.7$]{1affineness} when $n=1$, \cite[Corollary $14.3.5$]{stefanichthesis} when $n\geqslant2$), it sends colimits of stacks to limits of $(n+1)$-categories. So, we conclude that$$(n+1)\ShvCattwo^n{\lp X_{\mathrm{B}}\rp}\simeq\lim_{X}\hsp(n+1)\ShvCattwo^n{\lp\Spec(\Bbbk)\rp}\simeq \lim_X \hsp(n+1)\LinkPrLUtwo[n],$$where in the second equivalence we used the fact that every affine scheme is tautologically $n$-affine. On the other hand, we already know that$$(n+1)\LocSysCattwo^n{\lp X;\Bbbk\rp}\simeq\lim_X\hsp(n+1)\LinkPrLUtwo[n],$$hence the two expressions match.
\end{proof}

Combining \cref{lemma:shvcatn} with \cref{conj:infinityn} (and \cref{lemma:rectificationofmodules}), we obtain:
\begin{corollary}
	\label{cor:maincorshvcat}
Let $n\geqslant1$ be an integer. For any $n$-connected space $X$ and any $\Einf$-ring spectrum $\Bbbk$ we have equivalences of $(n+1)$-categories
\begin{align*}
(n+1)\ShvCattwo^n{\lp X_{\mathrm{B}}\rp}&\simeq(n+1)\mathbf{LMod}_{n\LocSysCattwo^{n-1}(\Omega_*X;\Bbbk)}{\lp(n+1)\PrLUtwo[n]\rp}\simeq(n+1)\mathbf{LMod}^n_{\operatorname{C}_{\bullet}(\Omega_*^{n+1}X;\Bbbk)}.
\end{align*}
\end{corollary}

\subsection{\texorpdfstring{$n$}{n}-affineness for Betti stacks}
\label{sec:proof_main_thm}
Our main result in this section is a  characterization of $n$-affine Betti stacks of connected spaces, for any integer $n\geqslant1$, in terms of $(n-1)$-affineness of the Betti stacks of their based loop spaces (\cref{prop:naffineifloopn-1}). We start by proving the following result, which will also be used extensively in \cref{sec:nkoszulcoaffine}.
\begin{proposition}
\label{prop:fullyfaithfulness}
 Let $n\geqslant1$ be an integer. For any space $X$, the underlying functor of the global sections $(n+1)$-functor $$(n+1)\bm{\Gamma}^{\mathrm{enh}}(X_{\operatorname{B}},-)\colon (n+1)\ShvCattwo^n(X_{\operatorname{B}})
 \longrightarrow 
(n+1)\mathbf{Lin}_{n\ShvCattwo^{n-1}(X_{\operatorname{B}})}\PrLUtwo[n]$$
 admits a fully faithful left adjoint $\operatorname{Loc}^n_{X_{\operatorname{B}}}$.
\end{proposition}
\begin{remark}
When $n=1$, then $\mathrm{Loc}^n_{X_{\operatorname{B}}}$ obviously agrees with the functor $\mathrm{Loc}_{X_{\operatorname{B}}}$ of \cite{1affineness} because of the essential uniqueness of left adjoints.
\end{remark}
\begin{proof}[Proof of \cref{prop:fullyfaithfulness}]
Writing $(n+1)\ShvCattwo^n(X_{\operatorname{B}})$ as $(n+1)\LocSysCattwo^n(X;\Bbbk)$, we can consider the following diagram of categories.
$$\begin{tikzpicture}[scale=0.75]
\node (a1) at (-4.5,2){$\displaystyle\LocSysCat^n(X;\Bbbk)$};
\node (a2) at (4.5,2){$\displaystyle\operatorname{Mod}_{n\LocSysCattwo^{n-1}(X;\Bbbk)}{\lp\PrLU[n]\rp}$};
\node (b) at (0,0){$\displaystyle\LinkPrLU[n]$};
\draw[->,font=\scriptsize](a1) to[out=280,in=180] node[below left]{$\displaystyle\Gamma(X,-)$} (b);
\draw[->,font=\scriptsize](a2) to[out=260,in=0]node[below right]{$\displaystyle\oblv_{n\LocSysCattwo^{n-1}(X;\Bbbk)}$} (b);
\draw[->,font=\scriptsize](a1) to node[above]{$\displaystyle\Gamma^{\operatorname{enh}}(X_{\operatorname{B}},-)$}(a2);
\end{tikzpicture}$$Here, $\Gamma(X,-)$ is the functor which takes global sections of a local system of $\Bbbk$-linear presentable $n$-categories: this amounts to taking the limit over the diagram of presentable categories of shape $X$ defined by a local system of categories.

Notice that this diagram does commute. Indeed, under the equivalence of \cref{lemma:shvcatn}, the global sections of a quasi-coherent sheaf of $n$-categories $n\bm{\mathscr{F}}$ over the Betti stack $X_{\operatorname{B}}$ correspond to the ``enhanced'' global sections of the corresponding local system of $n$-categories over $X$, taking into account the natural tensor action of $n\LocSysCattwo^{n-1}(X;\Bbbk)$ over them. This is simply a categorification of the fact that global sections of local systems of $\Bbbk$-modules over a space are endowed with an action of the algebra of its $\Bbbk$-cochains. In particular, forgetting such action recovers the underlying $\Bbbk$-linear presentable $n$-category $n\bm{\Gamma}(X,n\bm{\mathscr{F}})$.

We want to prove that this commutative diagram satisfies the assumptions of \cite[Corollary 4.7.3.16]{ha}. This means the following.
\begin{enumerate}
    \item The functor $\Gamma(X,-)$ is a right adjoint which preserves geometric realizations of $\Gamma(X,-)$-split simplicial objects.
    \item The functor $\oblv_{n\LocSysCattwo^{n-1}(X;\Bbbk)}$ is monadic.
    \item The diagram is vertically left adjointable.
\end{enumerate}
The fact that $\oblv_{n\LocSysCattwo^{n-1}(X;\Bbbk)}$ is monadic is obvious. The fact that $\Gamma(X,-)$ is a right adjoint follows from the fact that it computes the limit over an $n$-functor $X\to n\LinkPrLUtwo[n-1]$, and so its left adjoint is given by taking the constant local system of $n$-categories defined by a $\Bbbk$-linear presentable $n$-category. Moreover, since limits and colimits over spaces in $(n+1)\LinkPrLUtwo[n]$ agree (\cref{fact:limits_of_presentable_n}), this functor is also equivalent to the \textit{left} adjoint of the constant local system functor. In particular, it commutes with \textit{all} limits and colimits. Finally, we have to prove that for any $\Bbbk$-linear presentable $n$-category $n\bm{\scrC}$ the $n$-functor
$$n\bm{\scrC}\otimes_{n\mathbf{Lin}_{\Bbbk}\mathbf{Pr}^{\operatorname{L}}_{(\scriptscriptstyle\infty,n-1)}}n\LocSysCattwo^{n-1}(X;\Bbbk)\longrightarrow n\bm{\Gamma}^{\operatorname{enh}}\lp X_{\operatorname{B}},\hsp \operatorname{const}(n\bm{\scrC})\rp,$$
obtained by adjunction from the $n$-functor \begin{align*}
n\bm{\scrC}\longrightarrow n\bm{\Gamma}\lp X,\hsp \operatorname{const}(n\bm{\scrC})\rp\simeq\oblv_{n\LocSysCattwo^{n-1}(X;\Bbbk)}n\bm{\Gamma}^{\operatorname{enh}}\lp X_{\operatorname{B}},\hsp \operatorname{const}(n\bm{\scrC})\rp
\end{align*}which constitutes the unit for the adjunction $\operatorname{const}\dashv\Gamma(X,-)$, is an equivalence. Since forgetting the $n\LocSysCattwo^{n-1}(X;\Bbbk)$-module structure is conservative, we can check whether this map is an equivalence at the level of the underlying $\Bbbk$-linear presentable $n$-categories. On one side, we can write the domain of the above map as\begin{align*}
n\bm{\scrC}\otimes_{n\mathbf{Lin}_{\Bbbk}\mathbf{Pr}^{\operatorname{L}}_{(\scriptscriptstyle\infty,n-1)}}n\LocSysCattwo^{n-1}(X;\Bbbk)&\simeq n\bm{\scrC}\otimes_{n\mathbf{Lin}_{\Bbbk}\mathbf{Pr}^{\operatorname{L}}_{(\scriptscriptstyle\infty,n-1)}}\lim_X n\LinkPrLUtwo[n-1]\\&\simeq n\bm{\scrC}\otimes_{n\mathbf{Lin}_{\Bbbk}\mathbf{Pr}^{\operatorname{L}}_{(\scriptscriptstyle\infty,n-1)}}\colim_X n\LinkPrLUtwo[n-1]\\&\simeq\colim_Xn\bm{\scrC}\otimes_{n\mathbf{Lin}_{\Bbbk}\mathbf{Pr}^{\operatorname{L}}_{(\scriptscriptstyle\infty,n-1)}}n\LinkPrLUtwo[n-1]\\&\simeq\colim_X n\bm{\scrC}\simeq\lim_X n\bm{\scrC},
\end{align*}thanks to the fact that one can swap limits and colimits of presentable $n$-categories indexed by spaces. On the other hand, the codomain of this map is$$\oblv_{n\LocSysCattwo^{n-1}(X;\Bbbk)}n\bm{\Gamma}^{\operatorname{enh}}\lp X_{\operatorname{B}},\hsp \operatorname{const}(n\bm{\scrC})\rp\simeq\lim_X n\bm{\Gamma}(\left\{\ast\right\}_{\operatorname{B}},\hsp \operatorname{const}(n\bm{\scrC}))\simeq\lim_X n\bm{\scrC}.$$So, we can apply \cite[Corollary 4.7.3.16]{ha} and obtain that the functor $\Gamma^{\mathrm{enh}}(X_{\operatorname{B}},-)$ admits a fully faithful left adjoint $\mathrm{Loc}^n_{X_{\operatorname{B}}}$.
\end{proof}
\begin{porism}
\label{porism:conservativity}
The proof of \cref{prop:fullyfaithfulness} crucially relies on \cite[Corollary 4.7.3.16]{ha}, which also implies that the fully faithful left adjoint $\mathrm{Loc}^n_{X_{\operatorname{B}}}$ is an equivalence precisely if the global sections functor is conservative, because this is the only obstruction to its monadicity. 
Note that the proof of \cref{prop:fullyfaithfulness}  provides an identification between the global sections $(n+1)$-functor $(n+1)\bm{\Gamma}(X_{\operatorname{B}},-)$ and  the global sections $(n+1)$-functor for local systems of presentable $n$-categories over $X$; in turn these are realized as a limit of presentable $n$-categories indexed by the space $X$. Since limits and colimits of presentable $n$-categories over diagrams indexed by spaces are canonically equivalent for all $n\geqslant1$, it follows that in this case the global sections \textit{always} commute with \textit{all} colimits. This simplifies greatly the discussion in \cref{remark:on0affineness} when $n\geqslant1$, and implies that for Betti stacks the weak and strong notions of $n$-affineness agree already for $n=1$ (but not for $n=0$).

Using \cref{remark:monadicity}, we can summarize the situation as follows: for any integer $n\geqslant1$ and for any space $X$, the following are equivalent:
\begin{enumerate}
\item the Betti stack $X_{\operatorname{B}}$ is $n$-affine, 
\item the underlying functor of the global sections $(n+1)$-functor
\[
(n+1)\LocSysCattwo^n(X;\Bbbk)\longrightarrow(n+1)\LinkPrLUtwo[n]
\]
is conservative,  
\item the enriched global sections functor yields an equivalence 
$$
(n+1)\LocSysCattwo^{n}(X;\Bbbk) \simeq (n+1)\mathbf{Mod}_{n\LocSysCattwo^{n-1}(X;\Bbbk)}{\lp (n+1)\PrLUtwo[n]\rp}$$
of presentable $(n+1)$-categories.
\end{enumerate}
\end{porism}
We now provide a complete characterization of Betti stacks that are $n$-affine, which generalizes \cite[Proposition 11.2.1]{1affineness} and strengthens the $n$-affineness criterion of \cite[Theorem $14.3.9$]{stefanichthesis}, at least in the topological setting. A drawback of our result is that the necessary and sufficient condition we find is not very easy to verify in practice. For this reason, at the end of this section we shall complement this result by providing more explicit  sufficient conditions for $n$-affineness and its failure. 
\begin{theorem}
\label{prop:naffineifloopn-1}
Let $n\geqslant1$ be an integer and let $X$ be a connected space with a choice of a base point. Then the following are equivalent.
\begin{thmenum}
\item\label{thm:Xnaffine}The Betti stack $X_{\operatorname{B}}$ is $n$-affine.
\item\label{thm:globalsecff}The $n$-functor
\[
n\LinkPrLUtwo[n-1]\otimes_{n\LocSysCattwo^{n-1}(X;\Bbbk)}n\LinkPrLUtwo[n-1]\longrightarrow n\LocSysCattwo^{n-1}(\Omega_*X;\Bbbk)
\]
is an equivalence.
\item\label{thm:n-1monadicity}The global sections $n$-functor
\[
n\bm{\Gamma}(\Omega_*(X_{\operatorname{B}}),-)\colon n\LocSysCattwo^{n-1}(\Omega_*X;\Bbbk)\longrightarrow n\LinkPrLUtwo[n-1]
\]
is monadic.
\end{thmenum}
\end{theorem}
\begin{remark}
\label{remark:omegabetti}
Notice that the Betti stack functor $(-)_{\operatorname{B}}$, being a geometric morphism  of topoi, preserves both colimits and finite limits. In particular, for a point $\eta\colon\left\{\ast\right\}\to X$ (inducing a point $\eta_{\operatorname{B}}\colon\Spec(\Bbbk)\to X_{\operatorname{B}}$) one has an obvious equivalence
\[
\Omega_*(X_{\operatorname{B}})\simeq(\Omega_*X)_{\operatorname{B}}.
\]
We shall simply write $\Omega_*X_{\operatorname{B}}$ for such Betti stack. In particular, if $n\geqslant2$, then \cref{prop:naffineifloopn-1} tells us that $X_{\operatorname{B}}$ is $n$-affine if and only if $\Omega_*X_{\operatorname{B}}$ is $(n-1)$-affine.
\end{remark}
  \begin{remark}
\label{warning:wrongkoszul}
Before proceeding with the proof of \cref{prop:naffineifloopn-1} let us comment on its statement when $n=1$. Under what conditions is the global section functor on the category of local systems over a space $Y$ monadic? Or, using the terminology of \cref{remark:on0affineness}, under what conditions is a Betti stack $Y_{\mathrm{B}}$  \emph{almost $0$-affine}?  It turns out that the answer is quite subtle.  As explained in \cref{remark:on0affineness}, we have the following simple facts: 
  \begin{enumerate} 
\item The functor $\Gamma(Y,-)$ is monadic, i.e., $Y_{\mathrm{B}}$ is almost $0$-affine,  \emph{only if} the trivial local system  $\underline{\Bbbk}_{ Y}$ is a generator of 
 $\LocSys(Y;\Bbbk)$.
\item The functor $\Gamma(Y,-)$ is monadic \emph{if} the trivial local system  $\underline{\Bbbk}_{Y}$ is a compact generator of 
 $\LocSys(Y;\Bbbk)$. Indeed, in this case $Y_{\mathrm{B}}$ is weakly $0$-affine in the sense of Gaitsgory, and therefore in particular almost $0$-affine. In turn,  by Schwede--Shipley (\cite[Proposition $7.1.2.6$]{ha})
  $\underline{\Bbbk}_{Y}$ is a compact generator if and only if the global section functor induces an equivalence 
\begin{align}
\label{bzn}
\LocSys(Y;\Bbbk) \simeq\Mod_{\operatorname{C}^{\bullet}(Y;\Bbbk)}.
\end{align}
\end{enumerate}
  In \cite[Corollary $3.18$]{BZN} it is stated that the equivalence \eqref{bzn} holds for all simply connected and finite spaces. However, as pointed out to us by Y. Harpaz and confirmed by D. Nadler in private communication, this statement is wrong. Consider for example $Y\coloneqq S^2$ to be the sphere. The algebra of $\Bbbk$-chains on its based loop space $\Omega_*Y$ is a free associative algebra $\Bbbk\langle u\rangle$ generated by a variable $u$ lying in homological degree $1$. Then the functor $\Mapin_{\operatorname{C}_{\bullet}(\Omega_*Y;\Bbbk)}(\Bbbk,-)$ cannot be conservative, because the non-trivial $\Bbbk\langle u\rangle$-module $\Bbbk\langle u,u^{-1}\rangle$ is right orthogonal to $\Bbbk$. In particular, $\Bbbk$ cannot be a \textit{generator} of $\LMod_{\operatorname{C}_{\bullet}(\Omega_*Y;\Bbbk)}$: and so in particular it cannot be a \emph{compact generator}.

Equivalence \eqref{bzn}  holds for finite disjoint unions of contractible spaces, and we believe that in fact these might be the only examples. We do not know of any non-trivial space satisfying \eqref{bzn}. 
On the other hand, a characterization of almost $0$-affine Betti stacks would be very interesting, but it seems difficult to achieve, and we have only partial results in this direction: we devote the whole \cref{sec:consequences} to collect them. In \cref{obstruction}, we show that the non-triviality of $\pi_1(Y,y)$ at any base point obstructs the monadicity of the global sections functor of local systems over $Y$, which is perhaps an expected result. In \cref{cor:BCPinfty} we will also show that the Betti stack of $\CC\PP^{\scriptstyle\infty}$ is almost $0$-affine, so we do have non-trivial examples. We leave the further exploration of these questions to future work.
\end{remark}
\begin{proof}[Proof of \cref{prop:naffineifloopn-1}]
For $n=1$, this is just \cite[Proposition 11.2.1]{1affineness}, so we can assume that $n\geqslant2$. In virtue of \cref{prop:fullyfaithfulness}, the Betti stack $X_{\operatorname{B}}$ is $n$-affine if and only if the global sections functor
\[
\Gamma^{\mathrm{enh}}(X_{\operatorname{B}},-)\colon \LocSysCat^n(X;\Bbbk)\longrightarrow\Mod_{n\LocSysCattwo^{n-1}(X;\Bbbk)}\lp\PrLU[n]\rp
\]
is fully faithful. Using \cref{conj:infinityn}, we find an equivalence of functors
\[
\Gamma^{\mathrm{enh}}(X_{\operatorname{B}},-)\simeq n\FuninLtwo_{n\LocSysCattwo^{n-1}(\Omega_*X;\Bbbk)}{\lp n\LinkPrLUtwo[n-1],-\rp},
\]
where $n\LinkPrLUtwo[n-1]$ is equipped with the trivial $n\LocSysCattwo^{n-1}(\Omega_*X;\Bbbk)$-module structure obtained by pulling back along the natural map $\Omega_*X\to\left\{\ast\right\}$. Using ambidexterity of limits and colimits indexed by groupoids in categories of presentable $n$-categories (\cref{fact:limits_of_presentable_n}), we also obtain an equivalence of functors
\[
\Gamma^{\mathrm{enh}}(X_{\operatorname{B}},-)\simeq n\LinkPrLUtwo[n-1]\otimes_{n\LocSysCattwo^{n-1}(\Omega_*X;\Bbbk)}(-).
\]
In particular, $\Gamma^{\mathrm{enh}}(X_{\operatorname{B}},-)$ is fully faithful if and only if, for any categorical $n\LocSysCattwo^{n-1}(\Omega_*X;\Bbbk)$-representation $n\bm{\scrC}$, the counit $n$-functor is an equivalence, i.e., if and only if the $n$-functor
\[
n\bm{\scrC}\otimes_{n\LocSysCattwo^{n-1}(\Omega_*X;\Bbbk)}n\LinkPrLUtwo[n-1]\otimes_{n\LocSysCattwo^{n-1}(X;\Bbbk)}n\LinkPrLUtwo[n-1]\longrightarrow n\bm{\scrC}
\]
is an equivalence. This immediately implies the equivalence between \cref{thm:Xnaffine} and \cref{thm:globalsecff}.

We shall now turn to proving the equivalence between \cref{thm:globalsecff} and \cref{thm:n-1monadicity}. We will need the following lemma.
\begin{lemma}
\label{lemma:actionmonadic}
The underlying functor of the action $n\LocSysCattwo^{n-1}(X;\Bbbk)$-linear $(n+1)$-functor
\begin{equation}
\label{functor:monoidalstructure}
n\LinkPrLUtwo[n-1]\otimes_{n\LocSysCattwo^{n-1}(X;\Bbbk)}n\LinkPrLUtwo[n-1]\longrightarrow n\LinkPrLUtwo[n-1]
\end{equation}
is a monadic functor of categories.
\end{lemma}
\begin{proof}
The tensor product $n\LinkPrLUtwo[n-1]\otimes_{n\LocSysCattwo^{n-1}(X;\Bbbk)}n\LinkPrLUtwo[n-1]$ is computed as a geometric realization of a simplicial diagram of $n$-categories $n\bm{\scrC}_{\bullet}$, whose $i$-th term is described as$$n\bm{\scrC}_i\simeq n\LinkPrLUtwo[n-1]\otimes n\LocSysCattwo^{n-1}(X;\Bbbk)^{\otimes i}\otimes n\LinkPrLUtwo[n-1]\simeq n\LocSysCattwo^{n-1}(X;\Bbbk)^{\otimes i},$$
where the faces and degeneracies are induced by pullback $n$-functors. Here, the tensor product is understood as the tensor product of $\Bbbk$-linear presentable $n$-categories, whose monoidal unit is $n\LinkPrLUtwo[n-1]$. Under the equivalences$$n\LinkPrLUtwo[n-1]\otimes n\LocSysCattwo^{n-1}(X;\Bbbk)^{\otimes i}\otimes n\LinkPrLUtwo[n-1]\simeq n\LocSysCattwo^{n-1}\lp X^{\times i};\Bbbk\rp,$$we can describe such simplicial object in more detail.
\begin{enumerate}
    \item The degeneracy morphisms of such simplicial diagram correspond to pulling back categorical local systems along projections $X^{\times i}\to X^{\times (i-1)}$.
    \item The face morphisms correspond either to pulling back categorical local systems along the extremal inclusions $X^{\times(i-1)}\simeq \left\{*\right\}\times X^{\times(i-1)}\subseteq X^{\times i}$ and $X^{\times(i-1)}\simeq X^{\times(i-1)}\times\left\{*\right\}\subseteq X^{\times i}$ (these are the face morphisms $\partial_0$ and $\partial_i$), or to pulling back categorical local systems along the morphisms $\Delta_p\colon X^{\times(i-1)}\to X^{\times i}$ described informally by $(x_1,\ldots,x_{p-1},x_p,x_{p+1},\ldots, x_{i-1})\mapsto (x_1,\ldots, x_{p-1},x_p,x_p,x_{p+1},\ldots,x_{i-1})$  (these are the face morphisms $\partial_p$, for $p\in\left\{1,\ldots, i-1\right\}$).
\end{enumerate}
Thanks to this description of the faces $n$-functors in this simplicial $n$-category, we immediately see that for any morphism $\alpha\colon [i]\to[j]$ the diagram$$\begin{tikzpicture}[scale=0.75]
\node (a) at (4,-1.5){$\displaystyle n\LocSysCattwo^{n-1}\lp X^{\times i};\Bbbk\rp$};
\node (b) at (-4,-1.5){$\displaystyle n\LocSysCattwo^{n-1}\lp X^{\times (i+1)};\Bbbk\rp$};
\node (c) at (4,1){$\displaystyle n\LocSysCattwo^{n-1}\lp X^{\times j};\Bbbk\rp$};
\node (d) at (-4,1){$\displaystyle n\LocSysCattwo^{n-1}\lp X^{\times (j+1)};\Bbbk\rp$};
\draw[->,font=\scriptsize] (b) to node[below]{$\displaystyle\partial_{i,*}$}(a);
\draw[->,font=\scriptsize] (d) to node[above]{$\displaystyle\partial_{i+1,*}$}(c);
\draw[->,font=\scriptsize] (c) to node[right]{$\displaystyle\alpha^*$} (a);
\draw[->,font=\scriptsize] (d) to node[left]{$\displaystyle(\alpha\star \mathrm{id}_{[0]})^*$} (b);
\end{tikzpicture}$$is commutative. This means that such simplicial diagram of $n$-categories (or better, the underlying simplicial diagram of categories) satisfies the monadic Beck-Chevalley condition (\cite[Definition C.$1.5$]{1affineness}), after applying suitably \cite[Lemma C.$1.6$]{1affineness}. Hence, \cite[Lemma C.$1.8$]{1affineness} implies that such action functor is indeed monadic.
\end{proof}
\cref{lemma:actionmonadic} is what we need in order to apply \cite[Corollary C.$2.3$]{1affineness}, which guarantees that we can compute the monad described by the action functor$$\alpha\colon\LinkPrLU[n-1]\otimes_{\LocSysCat^{n-1}(X;\Bbbk)}\LinkPrLU[n-1]\longrightarrow\LinkPrLU[n-1]$$as the composition$$\eta^*\circ\eta_*\colon\LinkPrLU[n-1]\longrightarrow\LinkPrLU[n-1]$$where $\eta\colon\left\{*\right\}\hookrightarrow X$ is the inclusion of the base point. This implies that the naturally defined functor$$\mathfrak{D}^n\colon\LinkPrLU[n-1]\otimes_{\LocSysCat^{n-1}(X;\Bbbk)}\LinkPrLU[n-1]\longrightarrow\LocSysCat^{n-1}(\Omega_*X;\Bbbk),$$obtained by taking the right adjoint to$$\LinkPrLU[n-1]\otimes\LinkPrLU[n-1]\simeq\LinkPrLU[n-1]\longrightarrow\LinkPrLU[n-1]\otimes_{\LocSysCat^{n-1}(X;\Bbbk)}\LinkPrLU[n-1]$$and then composing with the pullback $n$-functors induced by the two projections $\pi_1,\pi_2\colon\Omega_*X\rightrightarrows\left\{ *\right\}$, makes the diagram$$\begin{tikzpicture}[scale=0.75]
\node (a) at (-5,2){$\displaystyle\LinkPrLU[n-1]\otimes_{\LocSysCat^{n-1}(X;\Bbbk)}\LinkPrLU[n-1]$};
\node (b) at (5,2){$\displaystyle\LocSysCat^{n-1}(\Omega_*X;\Bbbk)$};
\node (c) at (0,0){$\displaystyle\LinkPrLU[n-1]$};
\draw[->,font=\scriptsize] (a) to node[above]{$\mathfrak{D}^n$}(b);
\draw[->,font=\scriptsize] (a) to[out=285,in=180] node[below left]{$\alpha$}(c);
\draw[->,font=\scriptsize] (b) to[out=255,in=0] node[right]{} (c);
\end{tikzpicture}$$
commute. In the above picture, the right-hand side arrow is the push-forward along the natural terminal morphism $\Omega_*X\to\left\{*\right\}$ (i.e., it is the global sections functor).

So, by Barr--Beck--Lurie, the functor $\mathfrak{D}^n$ is an equivalence if and only if $$\Gamma(\Omega_*X_{\operatorname{B}},-)\colon\LocSysCat^{n-1}(\Omega_*X;\Bbbk)\longrightarrow\LinkPrLU[n-1]$$is a monadic functor. But since $n\geqslant2$, this is equivalent to $n\LocSysCattwo^{n-1}(\Omega_*X;\Bbbk)$ being monadic over $n\LinkPrLUtwo[n-1]$ as $n$-categories.
\end{proof}
\subsection{Consequences of \cref{prop:naffineifloopn-1} and examples}
\label{sec:consequences}
In this last subsection, we collect some neat consequences of \cref{prop:naffineifloopn-1}, which determine the $n$-affineness (or the failure to be $n$-affine) of some classes of Betti stacks. We also study in detail the examples of $\CC\PP^{\scriptstyle\infty}$ (\cref{exmp:counterexample}) and of $S^1$ (\cref{exmp:sphere}). We start by showing that \cref{prop:naffineifloopn-1} provides a necessary and sufficient condition for $n$-affineness of \textit{arbitrary} (i.e., not necessarily connected) Betti stacks.
\begin{lemma}
\label{lemma:connected}
Let $n\geqslant0$ be an integer, and let $X$ be a disjoint union of (possibly infinitely many) connected components$$X=\coprod_{\alpha \in \pi_0X}X_{\alpha}.$$Then the functor
\[
\Gamma(X,-)\colon\LocSysCat^n(X;\Bbbk)\longrightarrow\LinkPrLU[n]
\]
is monadic if and only if each functor $\Gamma(X_{\alpha},-)$ is monadic. In particular, for $n\geqslant1$, the Betti stack $X_{\operatorname{B}}$ is $n$-affine if and only if each $(X_{\alpha})_{\operatorname{B}}$ is $n$-affine.
\end{lemma}
\begin{proof}
In virtue of \cref{fact:limits_of_presentable_n}, we have equivalences
\[
\LocSysCat^n(X;\Bbbk)\simeq\prod_{\alpha\in\pi_0X}\LocSysCat^n(X_{\alpha};\Bbbk)\simeq\bigoplus_{\alpha\in\pi_0X}\LocSysCat^n(X_{\alpha};\Bbbk)
\]
and the natural functors $\LocSysCat^n(X_{\alpha};\Bbbk)\to\LocSysCat^n(X;\Bbbk)$ are colimit-preserving fully faithful functors. Since the functor $\Gamma(X,-)$ computes the product of all the global sections functors$$\Gamma(X_{\alpha},-)\colon\LocSysCat^n(X_{\alpha};\Bbbk)\longrightarrow\LinkPrLU[n],$$it is clear that if $\Gamma(X,-)$ is conservative then all the functors $\Gamma(X_{\alpha},-)$ are conservative as well.

For the converse, we only need to check that taking products of conservative functors is conservative. But since $\LinkPrLU[n]$ is pointed for all integers $n\geqslant0$, it is easy to see that given a morphism of $n$-categorical local systems on $X$
$$n\bm{F}\simeq(n\bm{F}_{\alpha})\colon n\bm{\scrF}\simeq(n\bm{\scrF}_{\alpha})\longrightarrow n\bm{\mathscr{G}}\simeq(n\bm{\mathscr{G}}_{\alpha}),$$
then each $n\bm{\Gamma}(X_{\alpha},n\bm{F}_{\alpha})$ is realized as a retract of the morphism $n\bm{\Gamma}(X,n\bm{F})$. In particular, if the latter is an equivalence then the former are equivalences as well. Since each ${\Gamma}(X_{\alpha},-)$ was assumed to be conservative, this implies that $\Gamma(X,-)$ is conservative as well.

Since for $n\geqslant1$ the only obstruction to monadicity is the conservativity of the global sections (\cref{porism:conservativity}), this argument proves the lemma for all non-negative integers. However, when $n=0$ we also need to prove that the functor $\Gamma(X,-)$ preserves colimits of $\Gamma(X,-)$-split simplicial diagrams if and only if each $\Gamma(X_{\alpha},-)$ preserves colimits of $\Gamma(X_{\alpha},-)$-split simplicial diagrams. Since the natural functors
\[
\LocSys(X_{\beta};\Bbbk)\longrightarrow\LocSys(X;\Bbbk)
\]
are fully faithful and commute with all colimits, one direction is obvious. On the converse, assume that each $\Gamma(X_{\alpha},-)$ is monadic, and let $\scrF_{\bullet}$ be an augmented simplicial object in $\LocSys(X;\Bbbk)$ which becomes split after taking global sections. In particular, we can interpret it as a collection of augmented simplicial objects \[
\scrF_{\bullet}\simeq\lp\scrF^{\alpha}_{\bullet}\rp_{\alpha\in\pi_0X}.\]We claim that such product is split because each $\Gamma(X_{\alpha},\scrF^{\alpha}_{\bullet})$ is already a split simplicial object of $\Mod_{\Bbbk}$. Recall from \cite[$\S4.7.2$]{ha} the functor $\rho\colon\bDelta_+\to\bDelta_+$ defined by the assignment$$[n]\mapsto[0]\star[n]\cong[n+1].$$For any category $\scrC$, let us denote by $\mathrm{T}$ the functor
\[
(-)\circ\rho^{\op}\colon\Fun{\lp\bDelta^{\op},\hsp\scrC\rp}\longrightarrow\Fun{\lp\bDelta^{\op},\hsp\scrC\rp}.
\]
The natural inclusion $[n]\subseteq\rho([n])$ defines a natural transformation from $\rho$ to $\operatorname{id}_{\bDelta_+}$, so for any simplicial object $X_{\bullet}$ in $\scrC$ we have a map $\varphi_{\bullet}\colon\mathrm{T}X_{\bullet}\to X_{\bullet}$. In virtue of \cite[Corollary $4.7.2.9$]{ha}, the simplicial object $X_{\bullet}$ is split if and only if $\varphi_{\bullet}$ admits a right homotopy inverse $\psi_{\bullet}\colon X_{\bullet}\to\mathrm{T}X_{\bullet}$.

In our case, we obtain an augmented simplicial $\Bbbk$-module $\mathrm{T}\Gamma\lp X,\scrF_{\bullet}\rp$ such that for all $n\geqslant 0$ one has $$\mathrm{T}\Gamma\lp X,\scrF_n\rp\simeq \Gamma\lp X,\mathrm{T}\scrF_{n}\rp\simeq \Gamma\lp X,\scrF_{n+1}\rp.$$
Since the simplicial local system $\scrF_{\bullet}$ is $\Gamma(X,-)$-split, we have a morphism $\psi_{\bullet}\colon\Gamma\lp X,\hsp \scrF_{\bullet}\rp\longrightarrow \mathrm{T}\Gamma\lp X,\hsp \scrF_{\bullet}\rp$ which provides a right homotopy inverse to the product of the sequence of maps of $\Bbbk$-modules
\[
\lp \varphi_{\bullet}^{\alpha}\colon \Gamma\lp X_{\alpha},\mathrm{T}\scrF^{\alpha}_{\bullet}\rp\longrightarrow\Gamma\lp X_{\alpha},\scrF_{\bullet}^{\alpha}\rp\rp_{\alpha\in\pi_0X}.
\]
So, for any index $\overline{\alpha}\in\pi_0X$ we define
\begin{align*}
\psi_{\bullet}^{\overline{\alpha}}\colon \Gamma \lp X_{\alpha},\scrF^{\overline{\alpha}}_{\bullet}\rp\overset{\iota^{\overline{\alpha}}_{\bullet}}{\longhookrightarrow}\prod_{\alpha\in\pi_0X}\Gamma\lp X_{\alpha},\scrF^{\alpha}_{\bullet}\rp&\overset{\psi_{\bullet}}{\longrightarrow} \prod_{\alpha\in\pi_0X}\Gamma\lp X_{\alpha}, \mathrm{T}\scrF_{\bullet}^{\alpha}\rp\overset{\pi^{\overline{\alpha}}_{\bullet}}{\longrightarrow} \Gamma\lp X_{\alpha},\mathrm{T} \scrF^{\overline{\alpha}}_{\bullet}\rp,
\end{align*}
which is straight-forwardly seen to provide a right homotopy inverse to each $\varphi^{\overline{\alpha}}_{\bullet}$. In particular, each simplicial local system $\scrF^{\alpha}_{\bullet}$ is $\Gamma(X_{\alpha},-)$-split and as such each functor  $\Gamma(X_{\alpha},-)$ preserves its geometric realization. We thus obtain a chain of equivalences \begin{align*}
\colim_{[n]\in\bDelta^{\op}} \Gamma\lp X,\scrF_{\bullet}\rp&\simeq \colim_{[n]\in\bDelta^{\op}}\prod_{\alpha\in\pi_0X}\Gamma\lp X_{\alpha},\scrF^{\alpha}_{\bullet}\rp\\&\simeq \prod_{\alpha\in\pi_0X}\colim_{[n]\in\bDelta^{\op}}\Gamma\lp X_{\alpha},\scrF^{\alpha}_{\bullet}\rp\\&\simeq \prod_{\alpha\in\pi_0X}\Gamma\lp X_{\alpha},\colim_{[n]\in\bDelta^{\op}}\scrF^{\alpha}_{\bullet}\rp\simeq \Gamma\lp X,\hsp \colim_{[n]\in\bDelta^{\op}}\scrF_{\bullet}\rp
 \end{align*}
 where in the second equivalence we used the fact that colimits of split simplicial objects are universal (\cite[Remark $4.7.2.4$]{ha}). This completes the proof.
\end{proof}
\begin{remark}
\label{remark:decategorified}
Let $X$ be a (pointed) $n$-connected space. Applying iteratively \cref{prop:naffineifloopn-1}, we can check the $n$-affineness of $X_{\operatorname{B}}$ by examining the monadicity of the de-categorified global sections functor
\[
\Gamma(\Omega_*^nX,-)\colon\LocSys(\Omega_*^nX;\Bbbk)\longrightarrow\Mod_{\Bbbk}.
\]
When $X$ is not $n$-connected, we can still argue in a similar way, up to choosing a sequence of base points over all connected components of each iterated based loop space $\Omega^k_*X$, for $k\in\left\{0,\ldots,n-1\right\}$. Indeed, applying \cref{prop:naffineifloopn-1}, we can reduce the problem of $n$-affineness for $X_{\operatorname{B}}$ to the problem of the monadicity of the global sections functors
\[
\Gamma(\Omega_*X_{\alpha},-)\colon\LocSysCat^{n-1}(\Omega_*X_{\alpha};\Bbbk)\longrightarrow\LinkPrLU[n-1].
\]
In virtue of \cref{lemma:connected}, the above functor is monadic if and only if the global sections functors over all connected components of $\Omega_*X_{\alpha}$ are monadic. But any H-space is homogeneous (that is, all connected components are homotopy equivalent one to the other), so we deduce an equivalence of spaces
\[
\Omega_*X_{\alpha}\simeq\coprod_{\beta\in\pi_1X}\mathbf{B}(\Omega_*^2X_{\alpha})_{\beta}.
\]
Thus, it is sufficient to check the $(n-1)$-affineness of the Betti stack associated to the (connected) space $\mathbf{B}\Omega_*^2X$. Applying this line of argument iteratively, we obtain that for an arbitrary space $X$ the $n$-affineness of the Betti stack $X_{\mathrm{B}}$ is equivalent to the monadicity of the de-categorified global sections functors
\[
\Gamma(\Omega_*^nX_{\alpha},-)\colon\LocSys(\Omega_*^nX_{\alpha};\Bbbk)\longrightarrow\Mod_{\Bbbk}
\]
for all $\alpha\in\pi_0X$, up to choosing sufficiently many base points in $X$ and its loopings. Under the monodromy equivalence \eqref{lemma:koszulduality}, choosing base points in each connected component of $\Omega^n_*X$, this is in turn equivalent to the monadicity of the functor
\begin{align}
\label{functor:monadicaffine}
\Mapin_{\mathrm{C}_{\bullet}(\Omega^{n+1}_*X;\Bbbk)}(\Bbbk,-)\colon\LMod_{\mathrm{C}_{\bullet}(\Omega^{n+1}_*X;\Bbbk)}\longrightarrow\Mod_{\Bbbk}.
\end{align}
In particular, assuming sufficiently many base points, we can reduce the problem of $n$-affineness of arbitrary Betti stacks to the problem of $n$-affineness of Betti stacks of \textit{$(n+1)$-connected} spaces.
\end{remark}
\cref{remark:decategorified} provides a useful tool in proving that Betti stacks of certain spaces are, or are not, $n$-affine.
\begin{corollary}
\label{thm:naffinenessntruncated}
Let $n\geqslant1$ be an integer, and let $X$ be a pointed $n$-truncated space. Then the Betti stack $X_{\operatorname{B}}$ is $n$-affine.
\end{corollary}
\begin{proof}
Arguing as in \cref{remark:decategorified}, we may reduce ourselves to prove the statement when $X$ is $(n+1)$-connected. But an $n$-truncated, $(n+1)$-connected space must be contractible, and so the assertion is trivial.
\end{proof}
\begin{remark}
\cref{thm:naffinenessntruncated} yields, for any $n\geqslant1$ and for any $n$-truncated space $X$, an equivalence of symmetric monoidal $(n+1)$-categories$$(n+1)\LocSysCattwo^n(X;\Bbbk)\simeq(n+1)\mathbf{Mod}_{n\LocSysCattwo^{n-1}(X;\Bbbk)}\PrLUtwo[n],$$where $n\LocSysCattwo^{n-1}(X;\Bbbk)$ is seen as a symmetric monoidal $n$-category via the natural (point-wise) symmetric monoidal structure. On the other hand, if $X$ is (pointed) connected we have a symmetric monoidal equivalence$$(n+1)\LocSysCattwo^n(X;\Bbbk)\simeq(n+1)\mathbf{Mod}_{n\LocSysCattwo^{n-1}(\Omega_*X;\Bbbk)}\PrLUtwo[n]$$in virtue of \cref{conj:infinityn}. Here, however, we consider the monoidal structure on $n\LocSysCattwo^{n-1}(\Omega_*X;\Bbbk)$ provided by the Day convolution tensor product, which takes into account the $\Ebb_1$-algebra structure of $\Omega_*X$. Combining these two equivalences, we obtain that for any pointed, connected and $n$-truncated space $X$ there is an equivalence between (presentable) $n$-categorical modules for the standard monoidal structure on $n\LocSysCattwo^{n-1}(X;\Bbbk)$ and (presentable) $n$-categorical modules for the convolution monoidal structure on $n\LocSysCattwo^{n-1}(\Omega_*X;\Bbbk)$. As explained in the proof of \cref{prop:naffineifloopn-1}, the explicit equivalence is provided by sending a $n\LocSysCattwo^{n-1}(X;\Bbbk)$-module $n\bm{\scrC}$ to the presentable $n$-category $$n\bm{\scrC}\otimes_{n\LocSysCattwo^{n-1}(X;\Bbbk)}n\LinkPrLUtwo[n-1],$$which inherits an $n\LocSysCattwo^{n-1}(\Omega_*X;\Bbbk)$-action from the one on $n\LinkPrLUtwo[n-1]$. This can be seen as a topological analog of the Morita equivalence for convolution categories of \cite[Theorem $1.3$]{benzvi2012morita}.
\end{remark}
We now provide negative results concerning the $n$-affineness of Betti stacks, which can be interpreted as a dual of \cref{thm:naffinenessntruncated} in some sense.
 \begin{proposition}
 \label{obstruction}
Let $n\geqslant0$. Let $X$ be a space, and assume $\Bbbk$ to be a semisimple commutative ring. If there exists a base point $x$ such that $\pi_{n+1}(X,x)$ contains an element $g$ either of infinite order, or such that the order of $g$ is a unit in $\Bbbk$, then the global sections functor
\[
\Gamma(X,-)\colon\LocSysCat^n(X;\Bbbk)\longrightarrow\LinkPrLU[n]
\]
is never monadic. 
 \end{proposition}
\begin{remark}
As an immediate consequence of \cref{obstruction}, for $n\geqslant1$ we have that the $(n+1)$-th homotopy group of a space $X$ always provides an obstruction to the $n$-affineness of its Betti stack over a field of characteristic zero.
\end{remark}
\begin{proof}[Proof of \cref{obstruction}]
Let us write $\pi\coloneqq\pi_{n+1}(X,x)$. Using \cref{prop:naffineifloopn-1} and arguing as in \cref{remark:decategorified}, we can reduce ourselves to prove the lemma when $X$ is connected and $n=0$. So, we need to prove that the functor of $\Omega^{n+1}_*X$-invariants \eqref{functor:monadicaffine} is not monadic. 

Consider $\pi$ as a discrete space. Notice that $\mathrm{C}_{\bullet}(\pi;\Bbbk)\simeq \pi_0(\mathrm{C}_{\bullet}(\Omega_*^{n+1}X;\Bbbk))$ is isomorphic as a $\Bbbk$-algebra to the group ring $\Bbbk[\pi]$, and the obvious projection
\[
\Omega^{n+1}_*X\longrightarrow \pi_0(\Omega^{n+1}_*X)\simeq \pi
\]
turns $\Bbbk[\pi]$ into a $\mathrm{C}_{\bullet}(\Omega_*^{n+1}X;\Bbbk)$-module. Let $g\in\pi$ be as in the statement: then $(1-g)$ is a non-nilpotent element. Indeed, consider the subgroup $\langle g\rangle$ generated by $g$: the inclusion $\langle g\rangle\subseteq \pi$ induces an inclusion of commutative rings $\Bbbk[\langle g\rangle]\subseteq\Bbbk[\pi]$, so if $(1-g)$ is not nilpotent in $\Bbbk[\langle g\rangle]$ it will automatically be not nilpotent in $\Bbbk[\pi]$ as well. We can therefore reduce ourselves to prove the statement in the case $\pi$ is cyclic and generated by $g$.
\begin{enumerate}
    \item If $g$ has infinite order, then $\Bbbk[\pi]\cong \Bbbk[t,t^{-1}]$ is a domain, hence $(1-g)$ is not nilpotent.
    \item If $g$ has finite order $n$ and $n$ is a unit in $\Bbbk$, then $\Bbbk[\pi]$ is a semisimple algebra because of Maschke's theorem (see for example \cite[Theorem $3.4.7$]{MR1896125}). In particular, $\Bbbk[\pi]$ is reduced and $(1-g)$ is not nilpotent.
\end{enumerate}
It follows that the set $S\coloneqq \left\{(1-t)^n\mid n\geqslant 0\right\}\subseteq\Bbbk[\pi]$ satisfies the Ore conditions in the graded commutative ring $\pi_{\bullet}(\mathrm{C}_{\bullet}(\Omega_*^{n+1}X;\Bbbk))$, hence there exists the localization $(\mathrm{C}_{\bullet}(\Omega_*^{n+1}X;\Bbbk))[(1-g)^{-1}]$ (\cite[Section $7.2.3$]{ha}) and since $S$ does not contain the $0$ element such localization is not trivial. Since $(1-g)$ is an element in the fiber of the map $$\mathrm{C}_{\bullet}(\Omega_*^{n+1}X;\Bbbk)\longrightarrow\Bbbk[\pi]\longrightarrow\Bbbk,$$it follows that $\Bbbk$ is $S$-nilpotent. In particular, \cite[Proposition $7.3.2.14$]{ha} guarantees that $(\mathrm{C}_{\bullet}(\Omega_*^{n+1}X;\Bbbk))[(1-g)^{-1}]$ is right orthogonal to $\Bbbk$ as a $\mathrm{C}_{\bullet}(\Omega_*^{n+1}X;\Bbbk)$-module, so the global sections cannot be conservative.
\end{proof}
We stress that \cref{thm:naffinenessntruncated} does \textit{not} provide a necessary condition for $n$-affineness of Betti stacks. Indeed, the following result provides a countable class of counterexamples.
\begin{theorem}
\label{cor:BCPinfty}  
Let $n\geqslant0$ and $p\geqslant n$ be integers, and let $\Bbbk$ be a field of characteristic zero. Consider $\ZZ$ as a discrete $\Einf$-space. Then the global sections functor
\[
\Gamma(\mathbf{B}^{p}\ZZ,-)\colon\LocSysCat^n(\mathbf{B}^p\ZZ;\Bbbk)\longrightarrow\LinkPrLU[n]
\]
is monadic when $p=n+2k$, while it is not monadic when $p=n+2k+1$. In particular, for $n\geqslant1$ and $k\geqslant0$, the Betti stack $\lp\mathbf{B}^{n+2k}\ZZ\rp_{\operatorname{B}}$ is $n$-affine, while the Betti stack $\lp\mathbf{B}^{n+2k+1}\ZZ\rp_{\operatorname{B}}$ is not.
\end{theorem}
We are deeply thankful to Y. Harpaz and to an anonymous referee for providing key suggestions and remarks in order to prove \cref{cor:BCPinfty}.
\begin{proof}
We use \cref{prop:naffineifloopn-1} and \cref{remark:decategorified} to reduce ourselves to prove the statement for $n=0$: for all $k\geqslant0$, the de-categorified global sections functor $\Gamma(\mathbf{B}^{2k}\ZZ,-)$ is monadic, while $\Gamma(\mathbf{B}^{2k+1}\ZZ,-)$ is not.

For $k=0$, the two statements are a straight-forward consequence of \cref{thm:naffinenessntruncated} and of \cref{obstruction}, respectively. So we can assume $k\geqslant1$; in particular, both $2k$ and $2k+1$ are greater or equal than $2$. For an arbitrary integer $p\geqslant2$ and for $X=\mathbf{B}^p\ZZ$, the global sections functor \eqref{functor:monadicaffine} can be written as
\begin{equation}
\label{functor:monadicaffine2}
\Mapin_{\mathrm{C}_{\bullet}(\mathbf{B}^{p-2}\ZZ;\Bbbk)}(\Bbbk,-)\colon\Mod_{\mathrm{C}_{\bullet}(\mathbf{B}^{p-2}\ZZ;\Bbbk)}\longrightarrow\Mod_{\Bbbk}.
\end{equation}
\begin{enumerate}
    \item When $p=2t+2$ is even, $t\geqslant0$, the functor \eqref{functor:monadicaffine2} cannot be conservative. Indeed, the chain $\Bbbk$-algebra $\mathrm{C}_{\bullet}(\mathbf{B}^{2t}\ZZ;\Bbbk)$ is equivalent to the formal $\Bbbk$-algebra $\Sym_{\Bbbk}(\Bbbk[2t])$, which has a non-nilpotent element $u$ in degree $2t$. Let $M\coloneqq \mathrm{C}_{\bullet}(\mathbf{B}^{2t}\ZZ;\Bbbk)[u^{-1}]$ denote the $\mathrm{C}_{\bullet}(\mathbf{B}^{2t}\ZZ;\Bbbk)$-module obtained by inverting $u$. Since $\Bbbk$ is $u$-nilpotent, it follows that $\Mapin_{\mathrm{C}_{\bullet}(\mathbf{B}^{2t}\ZZ;\Bbbk)}(\Bbbk,M)\simeq0.$
    \item When $p=2t+1$ is odd, $t\geqslant1$, the functor \eqref{functor:monadicaffine2} is always monadic. First, we claim that it is a composition of two monadic functors: a fully faithful Koszul duality embedding \[\Mod_{\mathrm{C}_{\bullet}(\mathbf{B}^{2t-1}\ZZ;\Bbbk)}\subseteq\Mod_{\mathrm{C}^{\bullet}(\mathbf{B}^{2t}\ZZ;\Bbbk)},\]
    followed by the obvious forgetful functor
    \begin{align}
        \label{functor:forgetful}
        \oblv_{\mathrm{C}^{\bullet}(\mathbf{B}^{2t}\ZZ;\Bbbk)}\colon \Mod_{\mathrm{C}^{\bullet}(\mathbf{B}^{2t}\ZZ;\Bbbk)}\longrightarrow\Mod_{\Bbbk}.
    \end{align}
    In particular, it is obviously conservative. Indeed, notice that $\Bbbk$ is a perfect $\mathrm{C}^{\bullet}(\mathbf{B}^{2t}\ZZ;\Bbbk)$-module: letting $u$ be the generator in degree $-2t$, $\Bbbk$ is the cofiber of the map
    \[
    \cdot u\colon\mathrm{C}^{\bullet}(\mathbf{B}^{2t}\ZZ;\Bbbk)[-2t]{\longrightarrow}\mathrm{C}^{\bullet}(\mathbf{B}^{2t}\ZZ;\Bbbk).
    \]
    In particular, the $\mathrm{C}^{\bullet}(\mathbf{B}^{2t}\ZZ;\Bbbk)$-linear dual of $\Bbbk$ is the fiber of the dual of the above map, which is $\Bbbk[2t-1]$. It follows that, up to a shift, the functors
    $\Mapin_{\mathrm{C}^{\bullet}(\mathbf{B}^{2t}\ZZ;\Bbbk)}(\Bbbk,-)$ and $\Bbbk\otimes_{\mathrm{C}^{\bullet}(\mathbf{B}^{2t}\ZZ;\Bbbk)}(-)$ are equivalent, and so we have a string of adjunctions
    \begin{align*}
    \Bbbk\otimes_{{\mathrm{C}_{\bullet}(\mathbf{B}^{2t-1}\ZZ;\Bbbk)}}(-)&\dashv\Mapin_{\mathrm{C}^{\bullet}(\mathbf{B}^{2t}\ZZ;\Bbbk)}(\Bbbk,-)\\&\simeq\Bbbk[2t-1]\otimes_{\mathrm{C}^{\bullet}(\mathbf{B}^{2t}\ZZ;\Bbbk)}(-)\dashv\Mapin_{\mathrm{C}_{\bullet}(\mathbf{B}^{2t-1}\ZZ;\Bbbk)}(\Bbbk,-)[2t-1].
    \end{align*}
    So, in order to prove that the Koszul duality functor is fully faithful, we can reduce ourselves to prove that $\Bbbk\otimes_{{\mathrm{C}_{\bullet}(\mathbf{B}^{2t-1}\ZZ;\Bbbk)}}(-)$ is fully faithful. But this is a left adjoint of a functor which preserves colimits, and so it sufficient to check that the unit
    \[
    M\longrightarrow \Mapin_{ \mathrm{C}^{\bullet}(\mathbf{B}^{2t}\ZZ;\Bbbk)}\lp \Bbbk, \Bbbk\otimes_{{\mathrm{C}_{\bullet}(\mathbf{B}^{2t-1}\ZZ;\Bbbk)}} M\rp.
    \]
    is an equivalence when $M$ is the compact generator $\mathrm{C}_{\bullet}(\mathbf{B}^{2t-1}\ZZ;\Bbbk)$. This, in turn, is an obvious consequence of Koszul duality.

    The essential image of the Koszul duality embedding identifies $\Mod_{\mathrm{C}_{\bullet}(\mathbf{B}^{2t-1}\ZZ;\Bbbk)}$ with the category of \textit{$u$-complete modules}. These are those $\mathrm{C}^{\bullet}(\mathbf{B}^{2t}\ZZ;\Bbbk)$-modules $M$ for which the limit
\[
\lim_{u}M\coloneqq\lim_{\ZZ}\lp\cdots\overset{\cdot u}{\longrightarrow}M\overset{\cdot u}{\longrightarrow}M\overset{\cdot u}{\longrightarrow}M\rp
\]
    vanishes. Thus, to show the monadicity of the functor \eqref{functor:monadicaffine2} it is sufficient to check that if a simplicial diagram $M_{\bullet}$ of $u$-complete $\mathrm{C}^{\bullet}(\mathbf{B}^{2t}\ZZ;\Bbbk)$-modules splits as a simplicial diagram of $\Bbbk$-modules, then its geometric realization is itself $u$-complete. Since the forgetful functor \eqref{functor:forgetful} commutes with all limits and colimits, for any integer $q\in\ZZ$ we have a chain of isomorphisms of $\Bbbk$-vector spaces
\begin{align*}
\pi_q\lp\oblv_{\mathrm{C}^{\bullet}(\mathbf{B}^{2t}\ZZ;\Bbbk)}\lim_{u}\colim_{[n]\in\bDelta^{\op}}M_n\rp&\cong\pi_q\lp\lim_{\beta}\colim_{[n]\in\bDelta^{\op}}\oblv_{\mathrm{C}^{\bullet}(\mathbf{B}^{2t}\ZZ;\Bbbk)}M_n\rp\\
&\cong\colim_{[n]\in\bDelta^{\op}}\pi_q\lp\lim_{\beta}\oblv_{\mathrm{C}^{\bullet}(\mathbf{B}^{2t}\ZZ;\Bbbk)}M_n\rp,
\end{align*}
where we used the fact that $\oblv_{\mathrm{C}^{\bullet}(\mathbf{B}^{2t}\ZZ;\Bbbk)}M_{\bullet}$ is split and so its geometric realization is universal. Since each $M_n$ was complete, this is $0$ for all integers $q$, and so we conclude that $\lim_{u}\colim_{\bDelta^{\op}}M_n\simeq0$ in virtue of the left and right completion of the $t$-structure on $\Mod_{\Bbbk}$.
\end{enumerate}
\end{proof}
 We conclude this section by studying in detail two examples, providing two different (and somewhat more direct) proofs of the $1$-affineness of the Betti stack of the circle $S^1$, and of the non-$1$-affineness of the Betti stack of the infinite projective space $\CC\PP^{\scriptstyle\infty}$. Both examples essentially already appeared in the literature, and can be viewed as instances of the emerging picture of 3d Homological Mirror Symmetry: we refer the reader to \cite{teleman},  \cite{gammage2023perverse} and references therein for additional information. 
\begin{exmp}
\label{exmp:counterexample}
Let $X\coloneqq\CC\PP^{\scriptstyle\infty}$ be the  infinite-dimensional projective space. It is a   simply connected  CW complex whose based loop space is homotopy equivalent to the circle, i.e.,  
$$
 \Omega_*X\simeq S^1,   \quad \text{ and }\quad\Omega^2_*X\simeq \mathbb{Z}.
$$ 
In particular, there is an equivalence of $\mathbb{E}_2$-algebras $
\operatorname{C}_{\bullet}(\Omega^2_*X;\Bbbk) \simeq \Bbbk[t,t^{-1}]
$. This yields an equivalence of monoidal categories 
$$ 
\LMod_{\Omega^2_*X}(\Mod_{\Bbbk}) \simeq 
\Mod_{\Bbbk[t,t^{-1}]}.
$$
We can use this to obtain an interesting alternative description of  $\LocSysCat(X;\Bbbk)$. 
Indeed, we can write a chain of equivalences  
$$
\psi\colon
\LocSysCat(X;\Bbbk) \xrightarrow[\ref{conj:infinityn}]{\simeq} 
\operatorname{Lin}_{\Bbbk[t,t^{-1}]}\PrLU
\overset{\simeq}{\longrightarrow} 
\operatorname{Lin}_{\Qcoh(\Gm)}\PrLU 
\overset{\simeq}{\longrightarrow}
 \operatorname{ShvCat}(\Gm)
$$
where the last equivalence follows from the fact that $\Gm$, being affine, is obviously $1$-affine. Thus categorical local systems on $X$ can be described equivalently as quasi-coherent quasi-coherent sheaves of categories over $\Gm$.

We remark that $\psi$ can be seen as an instance of 3d Homological Mirror Symmetry. Let us briefly explain why this is the case. The pair of spaces 
$$
\mathrm{T}^*\mathbf{B}\Gm \longleftrightarrow \mathrm{T}^*\Gm
$$
is one of the basic examples of 3d mirror partners. At least if $\Bbbk=\mathbb{C}$, the category $\LocSysCat(X;\Bbbk)$ can be viewed as (a subcategory of) the category of 3d A-branes on $\mathrm{T}^{\ast}\mathbf{B}\Gm$. The key observation here is that topologically we have a homotopy equivalence
$$
\BGm(\CC) \simeq \CC\PP^{\scriptstyle\infty}
$$
and the category of 3d A-branes of a cotangent stack is expected to contain local systems of categories over the base; we refer to \cite{teleman} for a fuller discussion of this point. Conversely,   $\operatorname{ShvCat}(\Gm)$  is (a subcategory of)  the category of 3d B-branes on $\mathrm{T}^*\Gm$. From this perspective, the equivalence $\psi$ implements a dictionary relating A-branes on $\mathrm{T}^*\mathbf{B}\Gm$ and B-branes on its mirror, thus paralleling closely the classical 2d HMS story.

Next, let us show that $X$ is not $1$-affine; see also \cite{teleman} for a similar discussion. Note that it is enough to show that the global sections functor 
$$
\Gamma(X,-) \colon\LocSysCat(X;\Bbbk) \longrightarrow \mathrm{Lin}_{\Bbbk}\PrLU
$$ 
is not conservative. This is easily done directly using the equivalences provided by \cref{conj:infinityn} 
\begin{equation}
\label{equival}
\LocSysCat(X;\Bbbk) \simeq
\LMod_{S^1}{\lp\mathrm{Lin}_{\Bbbk}\PrLU\rp}\simeq\ \operatorname{Lin}_{\Bbbk[t,t^{-1}]}\PrLU.
\end{equation}
We can easily classify the $\Bbbk[t,t^{-1}]$-linear structures on $\Mod_{\Bbbk}$: these are equivalently described as characters of $\pi_2(\CC\PP^{\scriptstyle\infty})\cong \ZZ$ (see \cite[Propositions $2.16$ and $2.21$]{Pascaleff_Pavia_Sibilla_Local_Systems}). So, an $S^1$-action on $\Mod_{\Bbbk}$ corresponds to the choice of a non-trivial scalar $\lambda\in\Bbbk^{\times}$, which is the image of $t$ under an $\Ebb_2$-morphism $\Bbbk[t,t^{-1}]\to\Bbbk$. Let us denote by $\Mod_{\Bbbk}(\lambda)$ the corresponding categorical $S^1$-module.

Now, the global section functor on $\LocSysCat(X;\Bbbk)$ is corepresented by the trivial local system, which is the unit with respect to the ordinary monoidal structure on   
$\LocSysCat(X;\Bbbk)$. We denote this object by 
$$
\underline{\smash{\mathrm{LocSys}}}(-) \in \LocSysCat(X;\Bbbk).
$$
Under equivalence \eqref{equival}, the object $\underline{\smash{\mathrm{LocSys}}}(-)$ is mapped to $\Mod_{\Bbbk}(1)$. 

Notice that, for any invertible $\lambda\in\Bbbk^{\times}$, $\Mod_{\Bbbk}(\lambda)$ can be seen as the category of $\Bbbk(\lambda)$-modules inside $\Mod_{\Bbbk[t,t^{-1}]}$, where $\Bbbk(\lambda)$ is the commutative $\Bbbk[t,t^{-1}]$-algebra on the underlying $\Bbbk$-module $\Bbbk$ determined by the evaluation $\operatorname{ev}_{\lambda}\colon\Bbbk[t,t^{-1}]\to\Bbbk$; in other words:$$\Mod_{\Bbbk}(\lambda)\simeq\Mod_{\Bbbk(\lambda)}{\lp\Mod_{\Bbbk[t,t^{-1}]}\rp}.$$
Invoking the categorical Eilenberg--Watts theorem, we have then\begin{align*}
\FuninL_{\Bbbk[t,t^{-1}]}{\lp\Mod_{\Bbbk}(1),\hsp\Mod_{\Bbbk}(\lambda)\rp}&\simeq\FuninL_{\Bbbk[t,t^{-1}]}{\lp\Mod_{\Bbbk(1)}{\lp\Mod_{\Bbbk[t,t^{-1}]}\rp},\hsp\Mod_{\Bbbk}(\lambda)\rp}\\&\simeq{_{\Bbbk(1)}\!\operatorname{BMod}_{\Bbbk(\lambda)}}{\lp\Mod_{\Bbbk[t,t^{-1}]}\rp}\\&\simeq\Mod_{\Bbbk(1)\otimes_{\Bbbk[t,t^{-1}]}\Bbbk(\lambda)}.
\end{align*}
But now an easy homological computation shows that $\Bbbk(1)\otimes_{\Bbbk[t,t^{-1}]}\Bbbk(\lambda)$ is 0 whenever $\lambda\neq1$, and so the $S^1$-fixed points are trivial.

It might be useful to revisit the previous calculation from a geometric standpoint using equivalence $\psi$. Let 
$$
\iota_\lambda: \mathrm{Spec}(\Bbbk) \longrightarrow \Gm
$$
be the $\Bbbk$-rational point $\lambda \in \Gm(\Bbbk)$.  
Under $\psi$, the object $\Mod_{\Bbbk}(\lambda)$ becomes a categorified \emph{skyscraper sheaf}. That is, it is the quasi-coherent sheaf of categories obtained by pushing-forward the unit along the functor
$$
\iota_{\lambda,*}: \operatorname{ShvCat}(\mathrm{Spec}(\Bbbk)) \longrightarrow  \operatorname{ShvCat}(\Gm)
$$ 
The computation above shows that, as expected, skyscraper sheaves at different points are mutually orthogonal.
\end{exmp}

\begin{exmp}
\label{exmp:sphere}
Let us consider next the case $X=S^1\simeq K(\ZZ,1)$. 
We can prove directly   that its Betti stack is $1$-affine as follows.  Note first that $\LocSysCat(S^1;\Bbbk)$ is the same as the category of $\Bbbk$-linear presentable categories equipped with a choice of an autoequivalence $F\colon\scrC\simeq\scrC$. The latter, in turn, is equivalent to the category of $\QCoh(\BGm)$-linear presentable categories (\cite[Example $0.4$]{gammage2023perverse}). Since $\BGm$ is $1$-affine (\cite[Remark $2.5.2$]{1affineness}), it follows that
\begin{equation}
\label{hmssmh}
\ShvCat(S^1_{\operatorname{B}})\simeq\LocSysCat(S^1;\Bbbk)\simeq\operatorname{Lin}_{\QCoh(\BGm)}\PrLU\simeq\ShvCat(\BGm).\end{equation}

The latter category is also equivalent to the category $\operatorname{Lin}_{\QCoh(\Gm)}\PrLU$ of $\QCoh(\Gm)$-linear presentable categories, where now $\QCoh(\Gm)$ is seen as a symmetric monoidal category via the convolution tensor product induced by the $\Einf$-group structure on $\Gm$. This can be deduced by concatenating the equivalences $(10.1)$ and $(10.4)$ in \cite{1affineness}; an earlier proof of this fact can be found in \cite{benzvi2012morita}.

Equipped with this monoidal structure, $\QCoh(\Gm)$ is monoidally equivalent to the category of representations of $\ZZ$ inside $\Mod_{\Bbbk}$, which is in turn equivalent to the category $\LocSys(S^1;\Bbbk)$ equipped with its point-wise monoidal structure. This is precisely $\QCoh(S^1_{\operatorname{B}}).$ We deduce that there is an equivalence 
$$
\ShvCat{\lp S^1_{\operatorname{B}}\rp}\simeq \operatorname{Lin}_{\QCoh(S^1_{\operatorname{B}})}\PrLU 
$$
i.e., $S^1_{\operatorname{B}}$ is $1$-affine, as we wanted to show. We remark that equivalence \eqref{hmssmh} can be viewed as the opposite direction of 3d HMS with respect to the one considered in \cref{exmp:counterexample}, see \cite{gammage2023perverse} for more information.
\end{exmp}

\section{Categorified Koszul duality via coaffine stacks}
\label{sec:nkoszulcoaffine}
This section contains our main contribution to $\mathbb{E}_n$-Koszul duality, at least in the topological setting. First, recall the following classical story, essentially known since \cite[Theorem 2.1]{rothenbergsteenrod}. Given a pointed connected space $X$, then for any base commutative ring spectrum $\Bbbk$ the chain complex
\[
\mathrm{C}_{\bullet}(\Omega_*X;\Bbbk)\coloneqq\Sigma_+^{\scriptstyle\infty}(\Omega_*X)\otimes_{\Sphere}\Bbbk
\]
is an augmented associative $\Bbbk$-algebra under the Pontrjagin product. Taking the bar construction one gets the chain coalgebra
\[
\Bbbk\otimes_{\mathrm{C}_{\bullet}(\Omega_*X;\Bbbk)}\Bbbk\simeq\mathrm{C}_{\bullet}(X;\Bbbk),
\]whose $\Bbbk$-linear dual is the augmented associative cochain $\Bbbk$-algebra
\[
\CX\coloneqq\Mapin_{\Sp}(\Sigma_+^{\scriptstyle\infty}X,\Bbbk).
\]
In other words: $\CX$ is the $\Ebb_1$-Koszul dual of $\mathrm{C}_{\bullet}(\Omega_*X;\Bbbk)$. More generally, for any $n\geqslant1$ and a pointed $(n-1)$-connected space $X$, one can consider the dual of the \textit{$n$-fold} bar construction, and obtain an $\Ebb_n$-Koszul duality statement of the form
\[
\mathrm{C}_{\bullet}(\Omega^n_*X;\Bbbk)^{!_{\Ebb_n}}\simeq\CX.
\]
If moreover $X$ is $n$-connected, $\Bbbk$ is a field, and each vector space $\mathrm{H}^i(X;\Bbbk)$ is finitely generated, then the $\Ebb_n$-Koszul duality is involutive, in the sense that we also have the dual equivalence
\[
\CX^{!_{\Ebb_n}}\simeq \mathrm{C}_{\bullet}(\Omega^n_*X;\Bbbk),
\]
see for example \cite[Theorem 4.4.5]{dagx}.

When $n=1$ and $\Bbbk$ is a field of characteristic zero, there is also a kind of Morita equivalence relating modules over $\operatorname{C}_{\bullet}(\Omega_*X;\Bbbk)$ and modules over
$\operatorname{C}^\bullet(X; \Bbbk)$. The right statement is subtle: it requires to either restrict to appropriately bounded modules; or to change the notion of module we work with. In particular, if we work with \emph{ind-coherent} modules, i.e. if we replace  $\LMod_{\operatorname{C}^\bullet(X;\Bbbk)}$ with $\IndCoh_{\operatorname{C}^\bullet(X;\Bbbk)}$, we do obtain an equivalence
\begin{equation}
\label{eqqeq2}
\LMod_{\operatorname{C}_{\bullet}(\Omega_*X;\Bbbk)} 
\simeq  
\IndCoh_{\operatorname{C}^\bullet(X;\Bbbk)}.\tag{$\star$}
\end{equation}
In this section, we will explain how to define a category that should be viewed as the category of iterated ``ind-coherent'' modules over 
$\operatorname{C}^\bullet(X;\Bbbk)$. This will allow  us to prove our main result \cref{thm:mainkoszuln}, which provides a kind of  $n$-categorical Morita equivalence  relating
$
\mathrm{C}_\bullet(\Omega_*^nX;\Bbbk)$ and $\CX 
$, and recovers \eqref{eqqeq2} when $n=1$.

In fact we will not attempt to define directly a categorification of the notion of ind-coherent module.  Rather, the key idea in our approach is using the theory \textit{coaffine stacks} introduced in \cite{toenchampsaffines} and further studied in \cite{dagviii}. We stress that our approach is new even in the classical case of $\mathbb{E}_1$-Koszul duality, although in that setting it is ultimately equivalent to \eqref{eqqeq2}. 
\subsection{Koszul duality and quasi-coherent sheaves over coaffine stacks}
\label{sec:koszul_coaffine}
\begin{notation}
In this section, $\Bbbk$ always denotes a ground \textit{(discrete) field} of characteristic zero.
\end{notation}
We start by recalling some fundamental results in the theory of stacks associated to coconnective $\Bbbk$-algebras defined over a field $\Bbbk$ of characteristic zero, that will be used extensively in \cref{sec:mainthmkoszul}. We will mostly adopt the conventions from \cite{dagviii}. In particular, as in \cite{dagviii}, we will call these objects \textit{coaffine} rather than \textit{affine} stacks, to stress the difference with affine schemes (which in turn are the spectra of \textit{connective} $\Bbbk$-algebras).

\begin{defn}
\label{def:coaffine}
Let $n\geqslant1$ be an integer.
\begin{defenum}
    \item \label{def:coaffine1}We say that a $\Bbbk$-algebra $A$ is \textit{$n$-coconnective} if the structure morphism $\Bbbk\to A$ induces an isomorphism of abelian groups$$\Bbbk\overset{\cong}{\longrightarrow}\pi_0A$$and the homotopy groups $\pi_iA$ vanish for both $i\geqslant 1$ and $-n< i<0$. If $n=1$, we shall simply say that $A$ is \textit{coconnective}.
    \item \label{def:coaffine2}A \textit{coaffine stack} is a stack $X$ which is equivalent to the stack$$\Map_{\CAlg_{\Bbbk}}{\lp A,\hsp -\rp}\colon\CAlg_{\Bbbk}^{\geqslant0}\longrightarrow\scrS$$for some coconnective $\Bbbk$-algebra $A$. In this case, we shall say that $X$ is the \textit{cospectrum of $A$}, and we shall denote it as $\cSpec(A)$.
    \end{defenum}
\end{defn}
\begin{parag}
\label{parag:mainpropertiescoaffine}    
Coaffine stacks behave in a very similar way to affine schemes: for any stack $\scrY$ defined over $\Bbbk$, giving a morphism $\scrY\to\cSpec(A)$ is equivalent to giving a morphism of commutative $\Bbbk$-algebras $A\to\Gamma(\scrY,\mathscr{O}_{\scrY})$ (\cite[Theorem $4.4.1$]{dagviii}). Moreover, any coaffine stack $\scrX\simeq\cSpec(A)$ can be realized as the left Kan extension of its restriction to \textit{discrete} $\Bbbk$-algebras (i.e., coaffine stacks are \textit{$0$-coconnective} in the sense of \cite[Chapter $2$, §$1.3.4$]{studyindag1}). In particular, defining the category of classical 
 affine schemes$$\operatorname{Aff}^{\operatorname{cl}}_{\Bbbk}\coloneqq\lp\CAlg^{\operatorname{disc}}_{\Bbbk}\rp^{\op}$$as the opposite of the category of discrete $\Bbbk$-algebras, we have that for any coaffine stack $\scrX\simeq\cSpec(A)$ the inclusion
$$\lp\operatorname{Aff}_{\Bbbk}^{\operatorname{cl}}\rp_{/\scrX}\simeq\lp\lp\CAlg^{\operatorname{disc}}_{\Bbbk}\rp_{A/}\rp^{\op}\subseteq\lp\lp\CAlg_{\Bbbk}^{\geqslant0}\rp_{A/}\rp^{\op}\simeq\lp\operatorname{Aff}_{\Bbbk}\rp_{/\scrX}$$is cofinal, as stated in the proof of \cite[Chapter $3$, Lemma $1.2.2$]{studyindag1}. Moreover, coaffine stacks are stable under products and fiber products: we provide a short proof for later reference.
  \begin{lemma}
\label{lemma:cspecmonoidal}
Let $A$ be a coconnective $\Bbbk$-algebra, and let $A\to R$ and $A\to S$ be two $A$-algebras. Assume that $R$ and $S$ are coconnective as $\Bbbk$-algebras. Then we have a natural equivalence of stacks$$\cSpec(R\otimes_AS)\simeq\cSpec(R)\times_{\cSpec(A)}\cSpec(S).$$
 \end{lemma}
 \begin{proof}
 When $\Bbbk$ is a field and $A$ is a coconnective $\Bbbk$-algebras, then coconnective $A$-modules are stable under tensor product over $A$ (\cite[Proposition $4.5.4.(6)$]{dagviii}). So, it is sufficient to notice that the Yoneda embedding $\smash{\lp\CAlg^{\leqslant0}_{\Bbbk}\rp^{\op}}\subseteq\mathrm{St}_{\Bbbk}$ sends pushouts of commutative $\Bbbk$-algebras to fiber products of stacks. 
 \end{proof}
However, contrarily to the case of affine schemes, the category$$\QCoh(\scrX)\coloneqq\lim_{\substack{\Spec(R)\to \scrX\\ R\in \CAlg^{\geqslant0}_{\Bbbk}}}\Mod_R$$of quasi-coherent sheaves over a coaffine stack $\scrX\simeq\cSpec(A)$ does not recover the category of $A$-modules. Rather, we have the following result.
\end{parag}
\begin{proposition}[{\cite[Propositions $3.5.2$ and $3.5.4$; Remark $3.5.6$]{dagviii}}]
\label{prop:luriecoaffine}
Let $A$ be a coconnective $\Bbbk$-algebra, let $\scrX\coloneqq\cSpec(A)$ be its corresponding coaffine stack. Let $\eta\in \scrX(\Bbbk)$ be any $\Bbbk$-point of $\scrX$.
\begin{enumerate}
    \item There exists a right complete $t$-structure on $\Mod_A$ defined as follows.
    \begin{itemize}
    \item The coconnective objects are detected via the forgetful functor$$\oblv_A\colon\Mod_A\longrightarrow\Mod_{\Bbbk}.$$
    \item The connective objects are those $A$-modules such that, for \textit{any} morphism of commutative $\Bbbk$-algebras $A\to R$ with $R$ connective, the $R$-module $M\otimes_A R$ is connective.
    \end{itemize}
    \item There exists a both left and right complete $t$-structure on $\QCoh(\scrX)$, whose heart is equivalent to the ordinary abelian category of algebraic representations of the prounipotent group scheme $\pi_1(\scrX,\eta)$. Both connective and coconnective objects are detected via the pullback functor$$\eta^*\colon\QCoh(\scrX)\longrightarrow\QCoh(\Spec(\Bbbk))\simeq\Mod_{\Bbbk}.$$
    \item Let $F\colon\Mod_A\to\QCoh(\scrX)$ be the natural symmetric monoidal pullback functor. Then $F$ exhibits the $t$-structure of $\QCoh(\scrX)$ as the left completion of the $t$-structure on $\Mod_A$.
\end{enumerate}
\end{proposition}

\begin{defn}[{\cite[Definitions $3.0.1$, $3.1.13$ and $3.4.1$]{dagx}}]
\label{def:smallness}
Let $A$ be an associative $\Bbbk$-algebra, let $M$ be a left $A$-module.
\begin{defenum}
    \item\label{locallysmall}We say that $M$ is \textit{locally small} if $\pi_kM$ is a finite dimensional $\Bbbk$-vector space for any integer $k$. We say that $A$ is \textit{locally small} if its underlying $A$-module is locally small.
    \item\label{smallmodules} We say that $M$ is \textit{small} if$$\pi_{\bullet}M\coloneqq\bigoplus_{k\in\ZZ}\pi_kM$$is a finite dimensional $\Bbbk$-vector space. 
    \item\mylabel{smallalgebras}We say that $A$ is \textit{small} if its underlying $A$-module is connective and small, and the structure morphism $\Bbbk \to A$ induces an isomorphism of discrete $\Bbbk$-algebras $\Bbbk\cong \pi_0A/\mathfrak{n}$, where $\mathfrak{n}$ is the nilradical of $\pi_0A$.
\end{defenum}
\end{defn}
\begin{remark}
\label{remark:stablesmall}
Let $A$ be an associative $\Bbbk$-algebra, and let $\LMod^{\operatorname{sm}}_A$ be the full sub-category of $\LMod_A$ spanned by small objects. We have a Cartesian  diagram of categories$$\begin{tikzpicture}[scale=0.75]
\node(a) at (-2,2){$\LMod^{\operatorname{sm}}_A$};
\node (b) at (2,2){$\Perf_{\Bbbk}$};
\node (c) at (-2,0){$\LMod_A$};
\node (d) at (2,0){$\Mod_{\Bbbk}.$};
\draw[->,font=\scriptsize] (a) to node[above]{$\oblv_A$}(b);
\draw[->,font=\scriptsize] (c) to node[below]{$\oblv_A$}(d);
\draw[right hook->, font=\scriptsize] (a) to node[above]{}(c);
\draw[right hook->, font=\scriptsize] (b) to node[above]{}(d);
\end{tikzpicture}$$
Forgetting the $A$-module structure commutes with all limits and colimits, and $\Perf_{\Bbbk}$ is stable under finite limits and colimits inside $\Mod_{\Bbbk}$. It follows that $\LMod^{\operatorname{sm}}_A$ is a stable (but of course not complete or cocomplete) sub-category of $\LMod_A$. In particular, its ind-completion$$\IndCoh^{\operatorname{L}}_A\coloneqq\Ind\lp\LMod^{\operatorname{sm}}_A\rp$$is stable (\cite[Proposition $1.1.3.6$]{ha}). Moreover, since $\LMod_A^{\operatorname{sm}}$ admits all finite coproducts, it follows that $\IndCohL_A$ admits \textit{all} coproducts (since they are realized as filtered colimits of finite coproducts).  \end{remark}
\begin{defn}
\label{def:indcoh}
Let $A$ be an associative $\Bbbk$-algebra. The category $\IndCohL_A$ is the \textit{category of left ind-coherent modules on $A$}.
\end{defn}
\begin{warning}
\label{warning:indcoh}
The notation can be misleading: in \cite{studyindag1}, \textit{ind-coherent sheaves} over an affine scheme are interpreted as bounded modules with coherent homology. Rather, our definition of ind-coherent modules matches the one in \cite[Definition $3.4.4$]{dagx}. Yet, if $A$ is a discrete local Artinian ring, or a small $\Bbbk$-algebra in the sense of \cref{smallalgebras}, then the two notions coincide.
\end{warning}
In the following, we revisit $\Ebb_1$-Koszul duality for associative algebras and its formulation in terms of correspondences between categories of ind-coherent and quasi-coherent modules. Recall that, over any field $\Bbbk$, an augmented associative $\Bbbk$-algebra $A$ admits a $\Ebb_1$-Koszul dual $A^!$ if and only if there exists a morphism
\[
\mu\colon A\otimes A^!\longrightarrow\Bbbk
\]
which exhibits $A^!$ as the classifying object $\Mapin_{A}(\Bbbk,\hsp\Bbbk)$ of morphisms of left $A$-modules from $\Bbbk$ to itself, see \cite[Remark $3.1.12$]{dagx}.
\begin{proposition}
\label{parag:koszul}
Let $A$ be an augmented associative $\Bbbk$-algebra such that the augmentation of $A$ induces an isomorphism$$\pi_0A\overset{\cong}{\longrightarrow}\Bbbk.$$Suppose that the Koszul dual $A^!$ is locally small as a $\Bbbk$-module. Then there is an equivalence of categories
\begin{align}
\label{equiv:koszuldualitymodules}
\IndCoh^{\operatorname{L}}_{A}\simeq\RMod_{A^!}.
\end{align}
\end{proposition}

\begin{proof}
Using \cite[Remark $3.4.2$]{dagx}, we know that $\LMod^{\operatorname{sm}}_{A}$ is the thick sub-category spanned by $\Bbbk$ inside $\LMod_{A}$ -- i.e., it is the smallest stable category sitting inside $\LMod_{A}$ containing $\Bbbk$ and closed under retracts. In particular, the Koszul duality functor for modules$$\LMod^{\op}_{A}\longrightarrow\LMod_{A^!}$$restricts to an equivalence between $\LMod^{\operatorname{sm}}_{A}$ and $\Perf^{\op}_{A^!}$ (\cite[Proposition $3.5.6$]{dagx}), so applying the $A^!$-linear duality self-functor we have$$\LMod^{\operatorname{sm}}_{A}\overset{\simeq}{\longrightarrow}\Perf_{A^!}$$hence an equivalence on their ind-completions.
\end{proof}

\begin{remark}
At first sight, \cite[Proposition $3.5.2$]{dagx} seems to imply that, for  \cref{parag:koszul} to hold, we  need to assume that   $A^!$ is small. Actually, this is not the case: the smallness is only needed in order to have an equivalence of \textit{functors}
from $\Alg_{\Bbbk}^{\operatorname{sm}}$ to $\CatVinfty[1]$ between $\IndCoh^{\operatorname{L}}_{(-)}$ and $\RMod_{(-)^!}.$ Of course, if we do not assume our algebras to be small the tensor product does not preserve small modules, so the functor $\IndCohL_{(-)}$ is not even well-defined; yet, the \textit{point-wise} equivalence \eqref{equiv:koszuldualitymodules} still applies under our, milder, assumptions on $A$.
\end{remark}
We shall now equip $\IndCoh^{\operatorname{L}}_{A^!}$  with a $t$-structure using the following general recipe.

\begin{lemma}[{\cite[Chapter IV, Lemma $1.2.4$]{studyindag1}}]
\label{lemma:tstructuregr}
Let $\scrC$ be a (non  cocomplete) stable category, endowed with a $t$-structure. Then $\Ind(\scrC)$ carries a unique t-structure which is compatible with filtered colimits (i.e., such that truncation functors commute with filtered colimits), and for which the tautological inclusion $\scrC\subseteq\Ind(\scrC)$ is t-exact.
Moreover:
\begin{lemenum}
\item\label{tstructuregr1} The sub-categories $\Ind(\scrC)_{\geqslant0}$ and $\Ind(\scrC)_{\leqslant 0}$ are compactly generated by $\scrC_{\geqslant0}$ and $\scrC_{\leqslant0}$, respectively.
\item\label{tstructuregr2} Given any other stable category $\scrD$ equipped with a $t$-structure compatible with filtered colimits, any functor $F\colon\Ind(\scrC)\to\scrD$ is $t$-exact if and only if $\restr{F}{\scrC}$ is $t$-exact.
\end{lemenum}
\end{lemma}

\begin{proposition}
\label{prop:tstructureindcoh}
Let $A$ be a connective associative $\Bbbk$-algebra. Then $\IndCohL_A$ admits a right complete $t$-structure. Moreover, if $A$ is locally small the $t$-exact functor$$\Phi_{A}\colon\IndCoh^{\operatorname{L}}_{A}\longrightarrow\LMod_{A}$$induced by the natural inclusion $\LMod^{\operatorname{sm}}_A\subseteq\LMod_A$ exhibits $\LMod_A$ as the left completion of the $t$-structure on $\IndCohL_A$. 
\end{proposition}
\begin{warning}   
If one assumes $A$ to be \textit{small} rather than only locally small, \cref{prop:tstructureindcoh} boils down to \cite[Proposition $3.4.18$]{dagx}. One could be confused by the fact that there the $t$-structure on $ \IndCohL_A$ fails to be \textit{right} complete, but this is easily explained: in \cite{dagx} the functor $\Phi_A$ is replaced by another functor $\Psi_A$, which is more compatible with base change. The functor $\Psi_A$ is closely related to $\Phi_A$ but involves $A$-linear duality as well; in particular, it swaps connective and coconnective objects, and this explains why the $t$-structure on $\IndCoh^{\operatorname{L}}_{A}$ described in \cite[Definition $3.4.16$ and Remark $3.4.17$]{dagx} is left but not right complete. Rather, our definition of the $t$-structure in \cref{prop:tstructureindcoh} resembles the $t$-structure on ind-coherent sheaves over Noetherian schemes defined in \cite[Proposition $1.2.2$]{studyindag1}.
\end{warning}
\begin{proof}[Proof of \cref{prop:tstructureindcoh}]
When $A$ is commutative and such that $\Coh(A)\simeq\Mod_A^{\mathrm{sm}}$, this was essentially proved in \cite[Proposition 1.2.4]{indcoh}. However, that argument works also in our generality: we still spell out the details for the reader's convenience. For any connective associative $\Bbbk$-algebra the restriction of the ordinary $t$-structure on the category of left $A$-modules yields a $t$-structure on $\LMod^{\operatorname{sm}}_A$: this boils down to the fact that this is true for $\Perf_{\Bbbk}$, and that forgetting the $A$-module structure is a conservative operation which preserves all limits and colimits. Thus, the existence of the $t$-structure on $\IndCohL_A$ follows from \cref{lemma:tstructuregr}.

We can easily see that such $t$-structure is right complete as follows. Since $\IndCohL_A$ is stable and admits uncountable coproducts, and coconnective objects are stable under uncountable coproducts (because this is true in $\LMod_A$), we can use the (dual  of the) criterion \cite[Proposition $1.2.1.19$]{dagx} for the right completeness of $t$-structures on stable categories. Indeed, we have that$$\lp\IndCoh^{\operatorname{L}}_{A}\rp_{\leqslant-\scriptstyle\infty}\coloneqq\bigcap_{n\geqslant0}\lp\IndCoh^{\operatorname{L}}_{A}\rp_{\leqslant-n}\simeq\bigcap_{n\geqslant0}\lp\LMod_{A}\rp_{\leqslant-n}\simeq 0.$$

Moreover, the functor $\Phi_A\colon\IndCohL_A\to\LMod_A$ is $t$-exact: this is an obvious consequence of \cref{tstructuregr2} because the inclusion $\LMod^{\operatorname{sm}}_A\subseteq \LMod_A$ is $t$-exact.

To prove the claim about the left completion, we simply need to check that for any integer $n$ the functor $\Phi_A$ induces an equivalence of categories
\begin{equation}
\label{equiv:coconnective}
\lp\IndCohL_A\rp_{\leqslant n}\overset{\simeq}{\longrightarrow}\lp\LMod_A\rp_{\leqslant n}.
\end{equation}
Indeed, the equivalence \eqref{equiv:coconnective} would yield an equivalence between the categories of eventually coconnective objects
\[
\mathrm{IndCoh}_A^{\mathrm{L},+}\coloneqq\bigcup_{n\in\ZZ}\lp\IndCohL_A\rp_{\leqslant n}\simeq\bigcup_{n\in\ZZ}\lp\LMod_A\rp_{\leqslant n}\eqqcolon \LMod_A^+.
\]Restriction to eventually coconnective objects does not alter the left completion of a $t$-structure (\cite[Remark $1.2.1.18$]{ha}), so this implies that the left completion of $\IndCohL_A$ and $\LMod_A$ are equivalent. But the canonical $t$-structure on $\LMod_A$ is left complete (\cite[Proposition $7.1.1.13$]{ha}), so we conclude that it has to be the left completion of the $t$-structure on $\IndCohL_A$ as well.

Since the functor $\Phi_A$ is exact, we can reduce ourselves to consider the case $n=0$. We first prove that any perfect and coconnective left $A$-module $M$ is obtained as a colimit of small coconnective left $A$-modules. Write any such $M$ as a colimit$$M\simeq \colim_{i\in I}A^{\oplus r_i}[n_i]$$over some diagram $I$. Notice that, even if $M$ is perfect, the diagram cannot be assumed to be finite because $M$ could be obtained from $A$ via shifts, finite direct sums or \textit{retracts}, and the latter are only realized as \textit{countably infinite} colimits (\cite[Section $4.4.5$]{htt}). Since $M$ is coconnective, we have$$M\simeq\tau_{\leqslant 0}M\simeq\tau_{\leqslant 0}\lp\colim_{i\in I}A^{\oplus r_i}[n_i]\rp\simeq\colim_{i\in I}\tau_{\leqslant 0}\lp A^{\oplus r_i}[n_i]\rp,$$where in the last equivalence we used the fact the truncation functor $\tau_{\leqslant 0}$ is a left adjoint. Since $A$ is locally small and connective, each $\tau_{\leqslant0}\lp A^{\oplus r_i}[n_i]\rp$ is a small $A$-module. Moreover, by the very same definition of the $t$-structure on $\LMod^{\operatorname{sm}}_A$, we conclude that $\tau_{\leqslant0}\lp A^{\oplus r_i}[n_i]\rp$ is coconnective inside $\LMod^{\operatorname{sm}}_A$. 
Next, we prove that the functor $\Phi_A$ is fully faithful when restricted to $\lp\IndCohL_A\rp_{\leqslant0}$. We will actually prove that for \textit{any} small left $A$-module $M$ (seen trivially as a left ind-coherent $A$-module) and for any coconnective left ind-coherent $A$-module $N$ the map of spaces
\[
\Map_{\IndCohL_A}{\lp M,\hsp N\rp}\longrightarrow\Map_{\LMod_A}{\lp\Phi_A(M),\hsp\Phi_A(N)\rp}
\]
is an equivalence. Writing $N$ as a filtered colimit $\colim_i N_i$, with each $N_i$ small and coconnective, we have that
\begin{align*}
\Map_{\IndCohL_A}{\lp M,\hsp N\rp}&\simeq  \Map_{\IndCohL_A}{\lp M,\hsp \colim_i N_i\rp}\\&\simeq\colim_i\Map_{\IndCohL_A}{\lp M,\hsp N_i\rp}\simeq\colim_i\Map_{\LMod^{\mathrm{sm}}_A}{\lp M,\hsp N_i\rp},
\end{align*}
because each small $A$-module is compact in $\IndCohL_A$. The functor $\Phi_A$ sends $M$ and $N$ to their actual colimits in $\LMod_A$, so $\Phi_A$ is fully faithful on coconnective objects if
\[
\colim_i\Map_{\LMod_A}{\lp M,\hsp N_i\rp}\longrightarrow\Map_{\LMod_A}{\lp M,\hsp \colim_i N_i\rp}
\]
is an equivalence. As we already observed, $\LMod_A^{\mathrm{sm}}$ is the thick stable sub-category of $\LMod_A$ spanned by $\Bbbk$: so it sufficient to write $M$ as a retract of shifts and direct sums of $\Bbbk$
\[
M\simeq\colim_{j\in J} \Bbbk^{\oplus r_j}[n_j],
\]
and observe that the augmentation $A\to\Bbbk$ produces a map
\[
f\colon\colim_j A^{\oplus r_j}[n_j]\longrightarrow M
\]
whose fiber is at least $1$-connective. In particular, recalling that each $N_i$ is coconnective, we obtain
\begin{align*}
\colim_i\Map_{\LMod_A}{\lp M,\hsp N_i\rp}&\simeq\colim_i\fib{\lp \Map_{\LMod_A}{\lp \colim_j A^{\oplus r_j}[n_j],\hsp N_i\rp}\longrightarrow\Map_{\LMod_A}(\fib(f),\hsp N_i)\rp}\\&\simeq \colim_i \Map_{\LMod_A}{\lp\colim_j A^{\oplus r_j}[n_j],\hsp N_i\rp}\\
&\simeq \colim_i\lim_j\Map_{\LMod_A}{\lp A^{\oplus r_j}[n_j],\hsp N_i\rp}.
\end{align*}
Using the fact that the diagram $J$ is a limit over the category $\mathrm{Idem}^+$, and such limits are universal because they are colimits as well, we obtain that
\begin{align*}
 \colim_i\lim_j\Map_{\LMod_A}{\lp A^{\oplus r_j}[n_j],\hsp N_i\rp}&\simeq\lim_j\colim_i\Map_{\LMod_A}{\lp A^{\oplus r_j}[n_j],\hsp N_i\rp}\\&\simeq \lim_j\Map_{\LMod_A}{\lp A^{\oplus r_j}[n_j],\hsp N\rp}
 \simeq \Map_{\LMod_A}{\lp M,\hsp N\rp},
\end{align*}
and this concludes the proof.
\end{proof}
\begin{remark}
The fact that $\LMod_A$ is the left completion of $\IndCohL_A$ implies that such $t$-structure is left complete if and only if $\IndCohL_A$ is equivalent to $\LMod_A$. This cannot be true if small and compact objects are not the same. For this to be true, one needs some quite strong assumptions on $A$: indeed, $A$ must be small itself (since $A$ is always perfect), hence discrete (otherwise, the small left $A$-module $\pi_0A$ cannot admit a finite resolution of semi-free $A$-modules). Moreover, it must satisfy some regularity conditions which assure that every quotient $A/I$ has finite global homological dimension over $A$. In particular, if $A$ is a $\Ebb_2$-$\Bbbk$-algebra, then it must be commutative and it must be a finite product of copies of $\Bbbk$.
\end{remark}
The previous discussion will allow us to reformulate the correspondence between modules over Koszul dual algebras of \cref{parag:koszul} in terms of algebraic geometry, at least in the topological setting. The key ingredient is the following result, which is more or less already known in the literature (see \cite{dagviii,studyindag1}).
\begin{proposition}
\label{prop:qcoh(cspec)}
Let $A$ be a $2$-coconnective and locally small commutative $\Bbbk$-algebra, which as a mere   associative algebra admits a $\Ebb_1$-Koszul dual $A^!$. Then we have an equivalence of categories$$\LMod_{A^!}\simeq \QCoh(\cSpec(A)).$$
\end{proposition}
\begin{proof}
The augmented associative algebra $A^!$ is connective (\cite[Theorem $3.1.14$]{dagx}). If $A$ is locally small and simply connected, one can consider a semi-free resolution of $\Bbbk$ as an $A$-module and using the spectral sequence of \cite[Variant 7.2.1.24]{ha} one can see that $A^!\simeq\Mapin_A(\Bbbk,\hsp\Bbbk)$ is locally small as well. In particular, \cref{prop:luriecoaffine,prop:tstructureindcoh} provide us with the following characters.
\begin{enumerate}
\item The $t$-structure on $\IndCohL_{A^!}$.
    \item The $t$-structure on $\LMod_{A^!}$, which is the left completion of the one on $\IndCohL_{A^!}.$
    \item The $t$-structure on $\Mod_A$.
    \item The $t$-structure on $\QCoh(\cSpec(A))$, which is the left completion of the one on $\Mod_A$. 
\end{enumerate}
Since $A$ is coconnective and locally small, we are in the setting of \cref{parag:koszul} and we obtain the equivalence \eqref{equiv:koszuldualitymodules}. If such equivalence is $t$-exact then we can deduce our statement from the universal property of the left completion. Again, \cref{lemma:tstructuregr} implies that we just need to check whether the restriction of this equivalence to the full sub-category $\LMod^{\operatorname{sm}}_{A^!}$ is $t$-exact.
\begin{enumerate}
    \item First, notice that the duality functor $\IndCoh^{\operatorname{L}}_{A^!}\to\Mod_A$ preserves coconnective objects. Indeed, let $M^!$ be a coconnective small left $A^!$-module. Its image inside $\Mod_A$ is the module$$M\coloneqq\Mapin_{A^!}\lp\Bbbk,\hsp M^!\rp,$$and this mapping $\Bbbk$-module is immediately seen to be coconnective since it is a mapping spectrum from a connective object to a coconnective one.
    \item The duality functor also preserves connective objects. We can see it as follows: notice that, inside $\Mod_A$, a module $M$ is connective precisely if for any (or, equivalently, \textit{one}) map of $\Bbbk$-algebras $A\to R$ where $R$ is connective, the $R$-module $R\otimes_AM$ is connective (\cite[Proposition $4.5.4$]{dagviii}). So we can test whether for a connective small left $A^!$-module $M^!$ the $\Bbbk$-module$$M\otimes_A\Bbbk\coloneqq\Mapin_{A^!}{\lp\Bbbk,\hsp M^!\rp}\otimes_A\Bbbk$$is connective. But this is just the underlying $\Bbbk$-module of the $A^!$-module $M^!$, since the inverse to$$\restr{\Mapin_{A^!}{\lp\Bbbk,\hsp-\rp}}{\LMod^{\operatorname{sm}}_{A^!}}\colon\LMod^{\operatorname{sm}}_{A^!}\longrightarrow\Perf_A$$is realized precisely by its left adjoint $(-)\otimes_A\Bbbk$. So our claim follows from the fact that $M^!$ was assumed to be connective in the first place.
    \end{enumerate}
\end{proof}
\subsection{\texorpdfstring{$\Ebb_n$}{En}-Koszul duality for higher modules in algebraic topology}
\label{sec:mainthmkoszul}


We shall now apply the results and constructions of \cref{sec:koszul_coaffine} to the topological setting. We first set the following definition.
\begin{defn}
\label{def:nkoszul}
Let $n\geqslant0$ be an integer. A space $X$ is \emph{$n$-Koszul} if the following conditions hold.
\begin{defenum}
\item\label{nconnectedness} The space $X$ is $n$-connected.
\item \label{nilpotent}The space $X$ is \emph{nilpotent}: for any choice of a base point $x$ in $X$, the fundamental group $\pi_1(X,x)$ is a nilpotent group which acts nilpotently on every higher homotopy group $\pi_k(X)$ for $k\geqslant2$.
\item\label{finitetype}The space $X$ is \emph{rationally of finite type}: let $X_{\QQ}$ be the rationalization of $X$ (which is well defined, because of \cref{nconnectedness} and \cref{nilpotent}). Then for any $k\geqslant n+1$ the rational vector space $\pi_k(X_{\QQ})\cong\pi_k(X)\otimes_{\ZZ}\QQ$ is finitely generated.
\end{defenum}
\end{defn}
We collect some important features of \cref{def:nkoszul} in the following proposition, for later reference.
\begin{proposition}
\label{remark:nkoszul}
Let $n\geqslant0$ be an integer.
\begin{propenum}
\item\label{remark:nkoszulunderproducts} Let $\left\{X_{\alpha}\right\}_{\alpha\in\mathrm{A}}$ be a collection of $n$-Koszul spaces, and suppose that for any integer $k\geqslant0$ the set of indices $\left\{\alpha\in\mathrm{A}\mid \pi_k(X_{\alpha})\not\cong0\right\}$ is finite. Then the product $\smash{\prod_{\alpha}X_{\alpha}}$ is again $n$-Koszul. 
\item\label{prop:nkoszuleasy}If $X$ is an $n$-Koszul space, then it is also $k$-Koszul for all $0\leqslant k\leqslant n$.
\item\label{prop:nkoszultruncated}If $X$ is an $n$-Koszul space, then for any integer $k\geqslant0$ the Postnikov truncation $\pi_{\leqslant k}X$ is again $n$-Koszul.
\item\label{prop:nkoszuldeloops}Suppose $n\geqslant1$, and let $X$ be a pointed $n$-Koszul space. Then for any integer $0\leqslant k \leqslant n$ the iterated based loop space $\Omega^{k}_*X$ is $(n-k)$-Koszul.
\item\label{remark:cochainsarelocallysmall}Suppose $n\geqslant1$. Let $\Bbbk$ be a field of characteristic zero, and let $X$ be an $n$-Koszul space. Then the cochain $\Bbbk$-algebra $\CX$ is locally small, in the sense of \cref{locallysmall}. 
\end{propenum}
\end{proposition}
\begin{proof}
\cref{remark:nkoszulunderproducts} is a consequence of the fact that taking homotopy groups commutes with arbitrary products of spaces, while \cref{prop:nkoszuleasy} and \cref{prop:nkoszultruncated} are obvious.

Let us prove \cref{prop:nkoszuldeloops}. If $X$ is $n$-connected then $\Omega^k_*X$ is $(n-k)$-connected. Since we have clear isomorphisms
$$\pi_{\bullet}((\Omega^k_*X)_{\QQ},\delta_x)\cong\pi_{\bullet}(\Omega^k_*X,\delta_x)\otimes_{\ZZ}\QQ\cong\pi_{\bullet-k}(X,x)\otimes_{\ZZ}\QQ\cong\pi_{\bullet-k}(X_{\QQ}),$$it follows that $\Omega^k_*X$ is rationally of finite type if and only if $X$ is rationally of finite type. Finally, $\Omega^k_*X$ is a connected based loop space, so it is in particular a connected H-space and hence nilpotent (\cite[Pag. 49]{May_Ponto}).

Finally, we prove \cref{remark:cochainsarelocallysmall}. If $X$ is an $n$-connected space, with $n\geqslant1$, then the rationalization map $X\to X_{\QQ}$ induces isomorphisms
\[
\mathrm{H}^k(X;\QQ)\overset{\cong}{\longrightarrow}\mathrm{H}^k(X_{\QQ};\ZZ)
\]
for all integers $k\geqslant0$, and thus we have that
\begin{align*}
\pi_k\mathrm{C}^{\bullet}(X;\QQ)\cong\mathrm{H}^k(X;\QQ)\cong\mathrm{H}^k(X_{\QQ},\QQ).
\end{align*}
Since $X_{\QQ}$ is a rational space of finite type, each rational vector space $\mathrm{H}^k(X_{\QQ};\QQ)\cong\pi_k\mathrm{C}^{\bullet}(X;\QQ)$ is finitely generated (see for example \cite[Theorem 1.3.1]{dagxiii}). In particular, $\mathrm{C}^{\bullet}(X;\QQ)$ is a locally small $\QQ$-algebra. Faithfully flatness of field extensions and the K\"unneth formula imply that for any field $\Bbbk$ of characteristic zero and for any $n$-Koszul space $X$, with $n\geqslant1$, the commutative $\Bbbk$-algebra $\CX$ is again locally small over $\Bbbk$.
\end{proof}
\begin{remark}\label{remark:K(Q,1)locallyfinite} 
If $X$ is a $0$-Koszul space such that its rationalization $X_{\QQ}$ is a rational Eilenberg--MacLane space $K(\pi,1)$ for some finitely generated $\QQ$-vector space $\pi$, then it is still true that $\CX$ is locally finite over $\Bbbk$ for any field of characteristic zero. When $\Bbbk=\QQ$, this is a consequence of \cite[Corollary 1.1.16]{dagxiii}; the general case follows once again from the universal coefficient theorem and from the fully faithfulness of field extensions.
\end{remark}
\begin{parag}
As a consequence of \cref{remark:cochainsarelocallysmall}, for any $n\geqslant1$ we have that $\Ebb_n$-Koszul duality applies to any (pointed) $n$-Koszul space over a field $\Bbbk$ of characteristic zero, in the sense that we have equivalences
\begin{equation}
\label{equation:n-koszulduality}
\mathrm{C}_{\bullet}(\Omega^n_*X;\Bbbk)^{!_{\Ebb_n}}\simeq\CX\quad\text{and}\quad\CX^{!_{\Ebb_n}}\simeq \mathrm{C}_{\bullet}(\Omega^n_*X;\Bbbk).
\end{equation}
By the mod $\mathcal{C}$ Hurewicz theorem, it follows that if $\mathrm{C}^{\bullet}(X;\QQ)$ is locally small over $\QQ$ then $X$ is also rationally of finite type, so $X$ is $n$-Koszul. In particular, $n$-Koszul spaces are \textit{precisely} those for which equivalences \eqref{equation:n-koszulduality} hold over fields of characteristic zero. So, \cref{prop:qcoh(cspec)} and \cref{conj:infinityn} immediately imply the following.     
\end{parag}
\begin{corollary}
\label{cor:koszulalgebraic}
Let $X$ be a pointed $1$-Koszul space, and let $\Bbbk$ be a field of characteristic zero. Then we have equivalences of presentably $\Bbbk$-linear categories$$\LocSys(X;\Bbbk)\simeq \LMod_{\operatorname{C}_{\bullet}(\Omega_*X;\Bbbk)}\simeq\Qcoh(\cSpec(\CX)).$$
\end{corollary}
\begin{remark}
\label{remark:aff0affineness}If $X$ satisfies the assumptions of \cref{cor:koszulalgebraic}, then such equivalence arises geometrically as follows. Thanks to the equivalence \eqref{lemma:koszulduality}, the category $\LMod_{\operatorname{C}_{\bullet}(\Omega_*X;\Bbbk)}$ can be equivalently described as $\LocSys(X;\Bbbk)$, which is the category of quasi-coherent sheaves over the Betti stack $X_{\operatorname{B}}$ (as already observed in Paragraph \cref{parag:qcohbetti}). Since $\Gamma(X_{\operatorname{B}},\scrO_{X_{\operatorname{B}}})\simeq\CX$, the identity map of $\CX$ induces an affinization map
\[
\mathrm{aff}_X\colon X_{\operatorname{B}}\longrightarrow\cSpec(\CX).
\]
The equivalence of \cref{cor:koszulalgebraic} is then realized by pulling back and pushing forward along $\mathrm{aff}_X$. Indeed, the pullback functor $\mathrm{aff}_X^\ast$ is a functor between stable categories equipped with both left and right complete $t$-structures which is strongly monoidal and right $t$-exact (i.e.,  it preserves connective objects). Therefore, using \cite[Corollary $4.6.18$]{dagx}, we can deduce that it is uniquely determined by the symmetric monoidal and right $t$-exact functor
\[
\smash{\widetilde{\mathrm{aff}}}_X^\ast\colon\Mod_{\CX}\longrightarrow\LMod_{\operatorname{C}_{\bullet}(\Omega_*X;\Bbbk)}
\]
which is obtained by pre-composing $\mathrm{aff}_X^\ast$ with the natural left completion functor 
$$
\Mod_{\CX}\longrightarrow\QCoh(\cSpec(\CX)).$$

So, it will suffice to understand the behaviour of $\smash{\widetilde{\mathrm
{aff}}}^{\ast}_X$. Let $\eta\colon\left\{*\right\}\to X$ be the chosen base point in $X$, and let $\eta_{\mathrm{B}}\colon\Spec(\Bbbk)\to X_{\mathrm{B}}$ be its image under the Betti stack functor. Pullback along $\eta_{\mathrm{B}}$ yields a forgetful functor $\LMod_{\operatorname{C}_{\bullet}(\Omega_*X;\Bbbk)}\to\Mod_{\Bbbk}$, and using again \cite[Corollary $4.6.18$]{dagx} we can see that the symmetric monoidal and right $t$-exact functor
\[
\eta_{\mathrm{B}}^\ast\circ\mathrm{aff}_X^\ast\colon\QCoh(\cSpec(\CX))\longrightarrow \Mod_{\Bbbk}
\]
uniquely corresponds to the natural base change functor $\Mod_{\CX}\to\Mod_{\Bbbk}$ along the augmentation $\CX\to \Bbbk$, induced at the level of $\Bbbk$-cochains by $\eta$. In particular, for any $\CX$-module $M$ the underlying $\operatorname{C}_{\bullet}(\Omega_*X;\Bbbk)$-module of $\smash{\widetilde{\mathrm{aff}}}^\ast_X(M)$ is equivalent to $M\otimes_{\CX}\Bbbk$. This is just the left adjoint of the Koszul duality functor which induces the equivalence of \cref{prop:qcoh(cspec)}.
\end{remark}
\begin{remark}
\label{remark:qcoh(cspec)_good}
Let $X$ be a pointed $1$-Koszul space. Then the equivalence of \cref{cor:koszulalgebraic} implies some important properties of the stable category $\QCoh(\cSpec(\CX))$.
\begin{enumerate}
    \item The category $\QCoh(\cSpec(\CX))$ is always compactly generated, and hence self-dual, as a presentably $\Bbbk$-linear category: indeed, this follows from the fact that $\LocSys( X;\Bbbk)\simeq\LMod_{\mathrm{C}_{\bullet}(\Omega_*X;\Bbbk)}$ has this property. 
    \item Suppose $X$ is a $\Ebb_1$-monoid in spaces; then, \cref{lemma:cspecmonoidal} implies that $\cSpec(\CX)$ is a $\Ebb_1$-monoid in stacks. Thus, both $\LocSys( X;\Bbbk)$ and $\QCoh(\cSpec(\CX))$ come equipped with a convolution monoidal structure induced by the monoid structure of $X$. Since $\aff_X$ is a map of groups, \cref{remark:aff0affineness} implies that the equivalence of \cref{cor:koszulalgebraic} is monoidal. On the other hand, \cite[Proposition 2.10]{Pascaleff_Pavia_Sibilla_Local_Systems} shows that the convolution monoidal structure on $\LocSys(X;\Bbbk)$ corresponds to the relative tensor product monoidal structure of $\LMod_{\mathrm{C}_{\bullet}(\Omega_*X;\Bbbk)}$; the latter is obviously a \textit{rigid} monoidal category. Therefore, under the convolution tensor product, $\QCoh(\cSpec(\CX))$ is a rigid monoidal category as well.
\end{enumerate}
\end{remark}
\cref{cor:koszulalgebraic} allows us to revisit the classical Koszul duality for modules \eqref{equiv:koszuldualitymodules} from a different, though ultimately equivalent, angle.   This point of view has a considerable advantage. Namely, while the concept of $n$-categorical ind-coherent modules is somewhat mysterious and it is far from clear how to define it directly, quasi-coherent sheaves on coaffine stacks can be  categorified in a natural way: that is, we can consider quasi-coherent sheaves of $n$-categories over $\cSpec(\CX)$ (\cref{def:nqcoh}). Hence, $\Ebb_n$-Koszul duality for categorified modules over $\Ebb_n$-Koszul dual algebras in the topological setting can be stated as follows. 
\begin{theorem}
\label{thm:mainkoszuln}
Let $n\geqslant1$ be an integer, let $X$ be a pointed $(n+1)$-Koszul space, and let $\Bbbk$ be a field of characteristic zero. Then the natural $(n+1)$-functor
\[
\mathrm{aff}^\ast_X\colon (n+1)\ShvCattwo^{n}(\cSpec(\CX))\longrightarrow (n+1)\LocSysCattwo^{n}(X;\Bbbk)
\]
is an equivalence of $(n+1)$-categories.
\end{theorem}
\begin{remark}
Using \cref{notation:when_n_is_zero}, when $n=0$ then \cref{thm:mainkoszuln} reduces to the combination of \cref{cor:koszulalgebraic} with \cref{remark:aff0affineness}.
\end{remark}
\begin{remark}
\label{remark:keyenkoszul}
We stress that \cref{thm:mainkoszuln} does provide a generalization of the usual $\Ebb_1$-Koszul duality equivalence between categories of modules. Let us briefly comment on the two characters appearing in the statement: the $(n+1)$-category $(n+1)\ShvCattwo(\cSpec(\CX))$ provides the natural higher categorification of the concept of the category of quasi-coherent sheaves over $\cSpec(\CX)$. On the other hand, if $X$ is  $(n+1)$-Koszul (in particular, $n$-connected), \cref{conj:infinityn} provides an equivalence
\[
(n+1)\LocSysCattwo^{n}(X;\Bbbk)\simeq(n+1)\mathbf{LMod}^n_{\mathrm{C}_{\bullet}(\Omega^{n+1}_*X;\Bbbk)}.
\]
Thus, for $X$ an $(n+1)$-Koszul space, \cref{thm:mainkoszuln} does relate (the categorification of) quasi-coherent sheaves over the coaffine stack $\cSpec(\CX)$, and (the categorification of) left modules over the $\Ebb_n$-algebra $\mathrm{C}_{\bullet}(\Omega_*X;\Bbbk)$.
\end{remark}
\begin{parag} 
Even if the proof of \cref{thm:mainkoszuln} is essentially carried out via an inductive argument, the strategy for proving the $n=1$ case is strikingly different from the one we use in the $n\geqslant2$ case. Indeed, for $n\geqslant2$ the proof is somewhat formal and depends on the general behaviour of pullbacks and pushforwards of sheaves of $(n+1)$-categories along maps of prestacks (\cref{remark:shvcatleftrightkan}); however, the case where $n=1$ requires some more \textit{ad hoc} arguments.
\end{parag} 
We start with proving the following proposition, that allows us to set an inductive argument in the first place.
\begin{proposition}
 \label{prop:cospecdeloop}
Let $X$ be a pointed $1$-Koszul space and let $\Bbbk$ be a field of characteristic zero. Then there is an equivalence of stacks
\[
\cSpec(\COX)\simeq\Spec(\Bbbk)\times_{\cSpec(\CX)}\Spec(\Bbbk).
\]
 \end{proposition}
\begin{proof}
It is clear that $\COX$ is a locally small $\Bbbk$-algebra: if $n\geqslant2$, this follows from a combination of \cref{prop:nkoszuldeloops} and \cref{remark:cochainsarelocallysmall}; for $n=1$, we argue as follows. We observe that the cohomology $\mathrm{H}^{\bullet}(\Omega_*X;\Bbbk)$ is the graded dual of $\mathrm{H}_{\bullet}(\Omega_*X;\Bbbk)$, which in turn is the universal enveloping graded $\Bbbk$-algebra of the graded Lie algebra $\pi_{\bullet}(\Omega_*X,\delta_x)\otimes_{\ZZ}\Bbbk$ endowed with its Whitehead bracket (see for example \cite[Theorem $16.13$]{FHTrational}). Since the latter can be realized as a quotient of the tensor $\Bbbk$-algebra $\mathrm{T}^{\bullet}(\pi_{\bullet}(\Omega_*X,\delta_x)\otimes_{\ZZ}\Bbbk)$, it is sufficient to prove that
\[
\mathrm{T}^{\bullet}(\pi_{\bullet}(\Omega_*X,\delta_x)\otimes_{\ZZ}\Bbbk)_n\coloneqq\bigoplus_{p\geqslant0}\lp\bigoplus_{i_1+\cdots+i_p=n}\pi_{i_1}(\Omega_*X,\delta_x)\otimes_{\ZZ}\cdots\otimes_{\ZZ}\pi_{i_p}(\Omega_*X,\delta_x)\otimes_{\ZZ}\Bbbk\rp
\]
is finite dimensional for all $n\geqslant0$, which is a straightforward computation using the fact that $\pi_{\bullet}(\Omega_*X;\delta_x)\otimes_{\ZZ}\Bbbk$ is bounded and finitely generated as a graded $\Bbbk$-module. So we can apply the Eilenberg--Moore theorem (see for example \cite[Corollary $1.1.10$]{dagxiii}) and deduce the existence of a canonical equivalence$$\Bbbk\otimes_{\CX}\Bbbk\simeq\operatorname{C}^{\bullet}(\Omega_*X;\Bbbk).$$
Applying the cospectrum functor and using \cref{lemma:cspecmonoidal} we deduce our claim.
\end{proof}
 We are now ready to state the main stepping stone in proving \cref{thm:mainkoszuln} when $n=1$. 
 \begin{proposition}
    \label{lemma:locfullyfaithful}
Let $X$ be a pointed $2$-Koszul space over a field $\Bbbk$ of characteristic zero. Then the functor
\[
\mathrm{Loc}_{(\Omega_*X)^{\mathrm{aff}}}\colon\mathrm{Lin}_{\QCoh(\cSpec(\COX))}\PrLU\longrightarrow\ShvCat(\cSpec(\COX))
\]
is fully faithful.
 \end{proposition}
Since the proof of \cref{lemma:locfullyfaithful} is   somewhat convoluted and relies on many auxiliary results coming from rational homotopy theory and algebraic topology, we postpone it until later in \cref{sec:proof_of_key}. For the moment, we simply show how it implies that \cref{thm:mainkoszuln} holds when $n=1$. 
\begin{proposition}
\label{thm:mainkoszultrue}
Let $X$ be a pointed $2$-Koszul space and let $\Bbbk$ a field of characteristic zero. Then the affinization map $\mathrm{aff}_X\colon X_{\operatorname{B}}\to\cSpec(\CX)$ induces an equivalence of $2$-categories
\[
\mathrm{aff}^\ast_X\colon2\ShvCattwo(\cSpec(\CX))\overset{\simeq}{\longrightarrow}2\LocSysCattwo(X;\Bbbk).
\]
\end{proposition}
\begin{proof}
Recall that $\Omega_*X$ is a $\Ebb_1$-monoid and it is moreover $1$-Koszul (\cref{prop:nkoszuldeloops}), therefore we can use \cref{remark:qcoh(cspec)_good} and \cref{lemma:locfullyfaithful} to see that we are in the setting of \cite[$\S10.2$]{1affineness}. Therefore we can write
\[
\ShvCat(\cSpec(\CX))\simeq\mathrm{Lin}_{\QCoh(\cSpec(\COX))}\PrLU.
\]
Here, $\QCoh(\cSpec(\COX))$ is seen as a monoidal category via the convolution tensor product induced by the group structure on $\cSpec(\COX)$. Under the equivalence of \cref{cor:koszulalgebraic}, this monoidal structure corresponds to the Day convolution monoidal structure on $\LocSys(\Omega_*X;\Bbbk)$. Hence, we obtain a chain of equivalences
\begin{align*}
\ShvCat(\cSpec(\CX))&\overset{\simeq}{\longrightarrow}\mathrm{Lin}_{\QCoh(\cSpec(\COX))}{\PrLU}\\&\overset{\simeq}{\longrightarrow}\mathrm{Lin}_{\LocSys(\Omega_*X;\Bbbk)}{\PrLU}\\&\overset{\simeq}{\longrightarrow}\LocSysCat(X;\Bbbk),
\end{align*}
where the second equivalence is obtained by base change along $\mathrm{aff}^\ast_{\Omega_*X}$ and the third equivalence is due to \cref{conj:infinityn}. Let us call $\Psi$ the composition of all these equivalences. To check that this functor is equivalent to $\mathrm{aff}^\ast_X$, we simply notice that $\Psi$ sends a sheaf of categories $\scrF$ over $\cSpec(\CX)$ to a local system of categories over $X$ with the same stalks as $\aff^\ast_X\scrF$. Since local systems are \textit{hyper}sheaves, this implies that $\Psi(\scrF)\simeq\aff_X^\ast\scrF$.

We are left to promote such equivalence to a $2$-categorical equivalence. In order to do this, we just need to check that the equivalence $\mathrm{aff}^\ast_X$ intertwines the coaugmentations from $\LinkPrLU$ on both sides. This is clear since such coaugmentations are induced by pulling back along the terminal morphisms $X_{\operatorname{B}}\to\Spec(\Bbbk)$ and $\cSpec(\CX)\to\Spec(\Bbbk)$, and $\mathrm{aff}_X$ obviously commutes with them.
\end{proof}
\begin{proof}[Proof of \cref{thm:mainkoszuln}]
\cref{thm:mainkoszultrue} proves the case where $n=1$. For a general $n\geqslant2$: assume that we have proved \cref{thm:mainkoszuln} for all integers $1\leqslant k\leqslant n-1$. Let
\[
\cSpec(\eta)\colon\Spec(\Bbbk)\overset{\eta_{\operatorname{B}}}{\longrightarrow} X_{\operatorname{B}}\overset{\mathrm{aff}^\ast_X}{\longrightarrow}\cSpec(\CX)
\]
be the pointing of $\cSpec(\CX)$ induced by the chosen base point $\eta\colon\left\{*\right\}\to X$. This produces a commutative diagram of categories
\begin{align}
\label{fundamentaldiagram}
\begin{tikzpicture}[scale=0.75,baseline=(current  bounding  box.center)]
    \node (a) at (-4,0) {$\ShvCat^n(\cSpec(\CX)$};
    \node (b) at (4,0) {$\LocSysCat^n(X;\Bbbk)$};
    \node (c) at (0,-2) {$\LinkPrLU[n].$};
    \draw[->,font=\scriptsize] (a) to[bend right] node[left]{$\cSpec(\eta)^\ast$} (c);
    \draw[->,font=\scriptsize] (b) to[bend left] node[right]{$\eta^*_{\operatorname{B}}$} (c);
    \draw[->,font=\scriptsize] (a) to node[above]{$\operatorname{aff}_X^*$}(b);
\end{tikzpicture}
\end{align}
We will prove that the diagram \eqref{fundamentaldiagram} satisfies the hypotheses of \cite[Corollary $4.7.3.16$]{ha}. This will allow us to apply the Barr--Beck--Lurie's monadicity theorem, and then conclude that the $n$-categorical equivalence holds as well thanks to \cref{remark:monadicity}.
\begin{enumerate}
\item The functor $\eta^*_{\operatorname{B}}$ is both monadic and comonadic: it is conservative, it commutes with all colimits, and is part of  an ambidextrous adjunction. Its adjoint is computed as a left Kan extension along the pointing $\eta\colon\left\{*\right\}\to X$, which is the same as a right Kan extension in virtue of \cref{remark:shvcatleftrightkan}. With our connectedness assumptions on $X$, this adjoint is extremely simple to describe: under the equivalence$$\LocSysCat^n(X;\Bbbk)\simeq\operatorname{Lin}_{n\LocSysCattwo^{n-1}(\Omega_*X;\Bbbk)}\PrLU[n]$$the functor $\eta^*_{\operatorname{B}}$ corresponds to forgetting the $\LocSysCat^{n-1}(\Omega_*X;\Bbbk)$-module structure, and the adjoint is given by$$n\bm{\scrC}\mapsto n\bm{\scrC}\otimes_{n\LinkPrLUtwotiny[n-1]}n\LocSysCattwo^{n-1}(\Omega_*X;\Bbbk).$$
 \item The functor $\cSpec(\eta)^\ast$ is conservative. Indeed, suppose that a morphism of two quasi-coherent sheaves of $n$-categories $F\colon n\bm{\scrF}\to n\bm{\scrG}$ over $\cSpec(\CX)$ is an equivalence when considering local sections over $\Spec(\Bbbk)$: we want to prove that it is actually an equivalence on \textit{all} local sections. Since $\cSpec(\CX)$ is a coaffine stack, we can write it as a geometric realization
 \[
 \cSpec(\CX)\simeq\colim_{[m]\in\bDelta^{\op}}\Spec(A^m),
 \]
 where each $A^m$ is discrete and $A^0\simeq \Bbbk$ (\cite[Proposition 4.4.4]{dagviii}). Therefore, we have
 \[
\ShvCat^n(\cSpec(\CX))\simeq\lim_{[m]\in\bDelta}\mathrm{Lin}_{A^m}\PrLU[n].
\]
Therefore, we have that
\[n\bm{\Gamma}(F,\Spec({A^m}))\simeq n\bm{\Gamma}(F,\Spec({\Bbbk}))\otimes_{n\LinkPrLUtwotiny[n-1]}\mathrm{
id}_{n\mathbf{Lin}_{A^m}\PrLUtwotiny[n-1]}
\]
and this implies our claim.
\item The functor $\cSpec(\eta)^\ast$ commutes with all limits and colimits. Indeed, as observed in \cref{remark:shvcatleftrightkan}, it admits a both left and right adjoint $\cSpec(\eta)_\ast$. 
 \item For any $\Bbbk$-linear presentable $n$-category $n\bm{\scrC}$, the natural $n$-functor$$\mathrm{aff}_X^\ast\lp\cSpec(\eta)_\ast(n\bm{\scrC})\rp\longrightarrow\eta_{\operatorname{B},\ast}(n\bm{\scrC})$$obtained via adjunction from the counit $n$-functor $$\cSpec(\eta)^\ast(\cSpec(\eta)_\ast(n\bm{\scrC}))\simeq\eta^\ast_{\operatorname{B}}\lp\mathrm{aff}_X^\ast\lp\cSpec(\eta)_\ast(n\bm{\scrC})\rp\rp\longrightarrow n\bm{\scrC}$$is an equivalence. Since both $\cSpec(\eta)^\ast$ and $\eta_{\operatorname{B}}^\ast$ are conservative, we can reduce ourselves to check whether the $n$-functor at the level of local sections over $\Spec(\Bbbk)$
\begin{align}
\label{mustbeequiv}
n\bm{\Gamma}\lp\Spec(\Bbbk),\mathrm{aff}_X^\ast\lp\cSpec(\eta)_\ast(n\bm{\scrC})\rp\rp\longrightarrow n\bm{\Gamma}\lp\Spec(\Bbbk),\eta_{\operatorname{B},\ast}(n\bm{\scrC})\rp
\end{align}
is an equivalence. Under the equivalence$$\LocSysCat^n(X;\Bbbk)\simeq\LMod_{n\LocSysCattwo^{n-1}(\Omega_*X;\Bbbk)}{\lp \PrLU[n]\rp},$$the codomain of the functor \eqref{mustbeequiv} can be written as
\[
n\bm{\Gamma}(\Spec(\Bbbk),\eta_{\operatorname{B},\ast}(n\bm{\scrC}))\simeq n\bm{\scrC}\otimes_{n\LinkPrLUtwotiny[n-1]}n\LocSysCattwo^{n-1}(\Omega_*X;\Bbbk).
\]
The left hand side, using \cref{remark:shvcatleftrightkan} and \cref{prop:cospecdeloop}, can be instead described as
\[
n\bm{\Gamma}\lp\Spec(\Bbbk),\mathrm{aff}_X^\ast\lp\cSpec(\eta)_\ast(n\bm{\scrC})\rp\rp\simeq\colim_{\substack{\Spec(R)\to\cSpec(\COX)\\R\in\CAlg^{\operatorname{disc}}_{\Bbbk}}}n\bm{\scrC}\otimes_{n\LinkPrLUtwotiny[n-1]}n\mathbf{Lin}_R\PrLUtwo[n-1].
\]
Since the tensor product of presentable $n$-categories is compatible with colimits, we can swap the tensor product and the colimit and using once again \cref{remark:shvcatleftrightkan} we can write
\begin{align*}
n\bm{\Gamma}\lp\Spec(\Bbbk),\mathrm{aff}_X^\ast\lp\cSpec(\eta)_\ast(n\bm{\scrC})\rp\rp&\simeq n\bm{\scrC}\otimes_{n\LinkPrLUtwotiny[n-1]}\lp\colim_{\substack{\Spec(R)\to\cSpec(\COX)\\R\in\CAlg^{\operatorname{disc}}_{\Bbbk}}}n\mathbf{Lin}_R\PrLUtwo[n-1]\rp\\&\simeq n\bm{\scrC}\otimes_{n\LinkPrLUtwotiny[n-1]}n\ShvCattwo^{n-1}(\cSpec(\COX)).
\end{align*}
Therefore, the $n$-functor \eqref{mustbeequiv} can be interpreted as the tensor product over $n\LinkPrLUtwo[n-1]$ of the affinization $n$-functor
\[
\mathrm{aff}_{\Omega_*X}^\ast\colon n\ShvCattwo^{n-1}(\cSpec(\COX))\longrightarrow n\LocSysCattwo^{n-1}(\Omega_*X;\Bbbk)
\]
with the identity $n$-functor of $n\bm{\scrC}$. Since $\Omega_*X$ is $n$-Koszul such that $\pi_q(\Omega_*X)\otimes_{\ZZ}\QQ$ is a finitely generated $\QQ$-vector space for all $q\geqslant0$ (\cref{prop:nkoszuldeloops}), the $n$-functor \eqref{mustbeequiv} is an equivalence because of the inductive hypothesis, as desired.
\end{enumerate}
So, Barr--Beck--Lurie's monadicity theorem allows us to conclude. 
\end{proof}
\begin{remark}
The reason why the above proof does not extend straightforwardly to the case where $n=1$ is that, in this case, the functor $\cSpec(\eta)^\ast$ can only be proved to be comonadic. Indeed, it is not obvious that $\cSpec(\eta)_\ast$ is both a left and right adjoint to $\cSpec(\eta)^\ast$. In particular, it is not obvious how to check that the natural functor
\begin{align*}
\scrC\otimes_{\Mod_{\Bbbk}}{\lp\lim_{\substack{\Spec(R)\to\cSpec(\COX)\\ R\in\CAlg^{\geqslant0}_{\Bbbk}}}\Mod_R\rp}\longrightarrow\lim_{\substack{\Spec(R)\to\cSpec(\COX)\\ R\in\CAlg^{\geqslant0}_{\Bbbk}}}\scrC\otimes_{\Mod_{\Bbbk}}\Mod_R
\end{align*}
is an equivalence. When $X$ is $2$-Koszul this is true, \textit{a posteriori}, because of \cref{thm:mainkoszultrue}.
\end{remark}
\begin{remark}
\label{remark:E2-Koszuldeloops}
Let $n\geqslant1$ be an integer, and suppose that $X$ is a pointed $(n+1)$-Koszul space over a field $\Bbbk$ of characteristic zero. Then, since $X$ is in particular $n$-Koszul, one can expect to recover the $\Ebb_{n}$-Koszul duality equivalence between $(n-1)$-categorical modules by ``delooping'' $\Ebb_{n+1}$-Koszul duality between $n$-categorical modules. This is indeed the case: notice that the unit for the monoidal structure on $(n+1)\ShvCattwo^{n}(\cSpec(\CX))$ is the sheaf $n\underline{\smash{\mathbf{ShvCat}}}^{n-1}(-)$ whose global sections are precisely $n\ShvCattwo^{n-1}(\cSpec(\CX))$. So we have an equivalence of mapping $n$-categories between
\[
n\FuninLtwo_{(n+1)\ShvCattwo^{n}(\cSpec(\CX))}{\lp n\underline{\smash{\mathbf{ShvCat}}}^{n-1}(-),\hsp n\underline{\smash{\mathbf{ShvCat}}}^{n-1}(-)\rp}
\]
and
\[
n\FuninLtwo_{(n+1)\LinkPrLUtwotiny[n]}{\lp n\LinkPrLUtwo[n-1],\hsp n\ShvCattwo^{n-1}(\cSpec(\CX))\rp}.
\]
Indeed, both objects are naturally equivalent to $n\ShvCattwo^{n-1}(\cSpec(\CX))$ because $n\LinkPrLUtwo[n-1]$ is the monoidal unit inside $(n+1)\LinkPrLUtwo[n]$.

Similarly, the monoidal unit for $(n+1)\LocSysCattwo^n(X;\Bbbk)$ is the trivial local system$$n\underline{\smash{{\mathbf{LocSysCat}}}}^{n-1}(-)\coloneqq\mathrm{const}{\lp n\LinkPrLUtwo[n-1]\rp},$$and so we have an equivalence of mapping $n$-categories between 
\[
n\FuninLtwo_{(n+1)\LocSysCattwo^n(X;\Bbbk)}{\lp n\underline{\smash{{\mathbf{LocSysCat}}}}^{n-1}(-),\hsp n\underline{\smash{{\mathbf{LocSysCat}}}}^{n-1}(-)\rp}
\]
and
\[
n\FuninLtwo_{(n+1)\LinkPrLUtwotiny[n]}{\lp n\LinkPrLUtwo[n-1],\hsp n\LocSysCattwo^{n-1}(X;\Bbbk)\rp}\simeq n\LocSysCattwo^{n-1}(X;\Bbbk).
\]
The $(n+1)$-functor $\mathrm{aff}^\ast_X$ sends $n\underline{\smash{\mathbf{ShvCat}}}^{n-1}(-)$ to $n\underline{\smash{{\mathbf{LocSysCat}}}}^{n-1}(-)$, because it is strongly monoidal and hence preserves the monoidal unit. Since $\mathrm{aff}^\ast_X$ is also an equivalence of $(n+1)$-categories, it induces an equivalence at the level of mapping $n$-categories, and so it recovers the $\Ebb_{n-1}$-Koszul duality equivalence for $(n-2)$-categorical modules. Applying iteratively this argument, we recover the $\Ebb_k$-Koszul duality equivalence for modules for all $k\leqslant n$, up to the classical $\Ebb_1$-Koszul duality for modules of \cref{cor:koszulalgebraic}.
\end{remark}
\subsection{Proof of \cref{lemma:locfullyfaithful}}
\label{sec:proof_of_key}
We dedicate this entire section to the proof of \cref{lemma:locfullyfaithful}. We will deduce it as a straight-forward corollary of the following more general result.
 \begin{proposition}
\label{prop:key}
Let $Y$ be a pointed $0$-Koszul space. Assume $Y$ is a $\Ebb_1$-monoid, and let $\Bbbk$ be a field of characteristic zero. Then the functor
\[
\mathrm{Loc}_{Y^{\mathrm{aff}}}\colon\mathrm{Lin}_{\QCoh(\cSpec(\CY))}\PrLU\longrightarrow\ShvCat(\cSpec(\CY))
\]
is fully faithful.
\end{proposition}
The proof of \cref{prop:key} involves many steps and reductions to the case in which our $0$-Koszul $\Ebb_1$-monoid $Y$ has only finitely many non-trivial rational homotopy groups. We start with the following.
\begin{lemma}
\label{lemma:key1}
Let $n\geqslant1$ be an integer, let $\pi$ be an abelian group of finite rank, and let $K(\pi,n)$ be the corresponding Eilenberg--MacLane space. Let $\Bbbk$ be a field of characteristic zero. Then the functor
\[
\mathrm{Loc}_{K(\pi,n)^{\mathrm{aff}}}\colon\mathrm{Lin}_{\QCoh(\cSpec(\CK[n]))}\PrLU\longrightarrow\ShvCat(\cSpec(\CK[n]))
\]
is fully faithful.
\end{lemma}
\begin{proof}
We argue by induction on $n$. If $n=1$ or $n=2$, then we have an equivalence of $\Bbbk$-algebras
\[\CK[n]\simeq\Sym_{\Bbbk}((\pi\otimes_\ZZ\Bbbk)^{\vee}[-1]).\]When $\Bbbk=\QQ$ this is \cite[Corollary 1.1.16]{dagxiii}, and for any field extension of $\QQ$ it follows from the universal coefficient theorem. Thus, $\cSpec(\CK[n])$ is $1$-affine (\cite[Remark 2.5.2 and Theorem 2.5.7.(a)]{1affineness}), so the claim is clear. 

For $n\geqslant3$, we can consider any base point of $K(\pi,n)$ (for example, the one induced by the group morphism $0\to \pi$) and consider the based loop space $\Omega_*K(\pi,n)\simeq K(\pi,n-1)$. Notice that for any $\Bbbk$-module $M$ we have
\[
\Bbbk\otimes_{\Sym_{\Bbbk}(M[-n])}\Bbbk\simeq\Sym_{\Bbbk}(M[-n+1]),
\]
so \cref{lemma:cspecmonoidal} implies that $\cSpec(\CK[n-1])\simeq\Omega_*\cSpec(\CK[n])$. By the inductive hypothesis, we know that $\mathrm{Loc}_{K(\pi,n-1)^{\aff}}$ is fully faithful. Moreover, for $n\geqslant3$ the space $\CK[n-1]$ is always $1$-Koszul. So using \cref{remark:qcoh(cspec)_good} and \cite[$\S10$]{1affineness} we can write a chain of equivalences
\begin{align*}
\ShvCat(\cSpec(\CK[n]))&\simeq\mathrm{Lin}_{\QCoh(\cSpec(\CK[n-1])))^{\circledast}}\PrLU\\&\simeq\mathrm{Lin}_{\LocSys(K(\pi,n-1);\Bbbk)}\PrLU.
\end{align*}
Concatenating with the equivalence of \cref{conj:infinityn} we finally obtain
\begin{align*}
\ShvCat(\cSpec(\CK[n]))\simeq\LocSysCat(K(\pi,n);\Bbbk).
\end{align*}
The latter equivalence is induced by pulling back along the affinization map $\aff_{K(\pi,n)}$: indeed, one can notice that both functors yield local systems of categories over $K(\pi,n)$ with equivalent stalks. It follows that we have a commutative diagram of categories
\[
\begin{tikzpicture}[scale=0.75,baseline=(current  bounding  box.center)]
\node (a) at (-5,-2.5){$\ShvCat(\cSpec(\CK[n]))$};
\node (b) at (5,-2.5){$\LocSysCat(K(\pi,n);\Bbbk)$};
\node (c) at (-5,0){$\mathrm{Lin}_{\QCoh(\cSpec(\CK[n]))}{\PrLU}$};
\node (d) at (5,0){$\mathrm{Lin}_{\LocSys(K(\pi,n);\Bbbk)}{\PrLU}.$};
\draw[->,font=\scriptsize] (c) to node[left]{$\mathrm{Loc}_{K(\pi,n)^{\mathrm{aff}}}$}(a);
\draw[->,font=\scriptsize] (d) to node[right]{$\mathrm{Loc}_{K(\pi,n)_{\operatorname{B}}}$}(b);
\draw[->,font=\scriptsize] (a) to node[below]{$\mathrm{aff}^\ast_{K(\pi,n)}$}node[above]{$\simeq$}(b);
\draw[->,font=\scriptsize] (c) to node[above]{$\simeq$}(d);
\end{tikzpicture}
\]
where the horizontal functors are equivalences and the right vertical functor is fully faithful because of \cref{prop:fullyfaithfulness}. So the left vertical functor is fully faithful as well.
\end{proof}
\begin{lemma}
\label{lemma:key2}
Let $n_1,\ldots,n_k$ be integers such that $n_i\geqslant 1$ for all $1\leqslant i\leqslant k$. Let $\pi_1,\ldots,\pi_k$ be abelian groups of finite rank, and let $K(\pi_1,n_1),\ldots,K(\pi_k,n_k)$ be the corresponding Eilenberg--MacLane spaces. Let $$K\coloneqq\prod_{i=1}^k K(\pi_i,n_i)$$denote their product and let $\Bbbk$ a field of characteristic zero.. Then the functor $$\mathrm{Loc}_{K^{\aff}}\colon\mathrm{Lin}_{\QCoh(\cSpec(\mathrm{C}^{\bullet}(K;\Bbbk))}\PrLU\longrightarrow\ShvCat(\cSpec(\mathrm{C}^{\bullet}(K;\Bbbk)))$$ is fully faithful.
\end{lemma}
\begin{proof}
We can assume that $n_i\neq n_j$ when $i\neq j$ (otherwise, we can replace $K(\pi_1,n_1)\times K(\pi_j,n_j)$ with the homotopically equivalent $K(\pi_i\times\pi_k,n_i)$). Next, notice that the Eilenberg--Moore theorem yields an equivalence of commutative $\Bbbk$-algebras
\[
\bigotimes_{i=1}^k\mathrm{C}^{\bullet}(K(\pi_i,n_i);\Bbbk)\simeq\mathrm{C}^{\bullet}(K;\Bbbk).
\]
It follows from \cref{lemma:cspecmonoidal} that
\[
\cSpec(\mathrm{C}^{\bullet}(K;\Bbbk))\simeq\prod_{i=1}^k\cSpec(\mathrm{C}^{\bullet}(K(\pi_i,n_i);\Bbbk)).
\]
Since for any prestacks $\scrX$ and $\scrY$ the functor
\[
\ShvCat(\scrX)\otimes_{\mathrm{Pr}^{\mathrm{L}}_{(\scriptscriptstyle\infty,1)}}\ShvCat(\scrY)\longrightarrow\ShvCat(\scrX\times\scrY)
\]
is an equivalence (\cite[Proposition 14.2.15]{stefanichthesis}), it will suffice to prove that 
\[
\QCoh(\cSpec(\mathrm{C}^{\bullet}(K;\Bbbk)))\simeq\bigotimes_{i=1}^k\QCoh(\cSpec(\mathrm{C}^{\bullet}(K(\pi_i,n_i);\Bbbk))).
\]
In this way, $\mathrm{Loc}_{K^{\aff}}$ is realized as the tensor product of the functors $\mathrm{Loc}_{K(\pi_i,n_i)^{\aff}}$ which are fully faithful (\cref{lemma:key1}), and so it is fully faithful as well. By induction, we can reduce ourselves to the case where $k=2$. In this case, at least one of $n_1$ and $n_2$ is bigger than $1$ (let us assume that $n_2\geqslant2$), so \cref{cor:koszulalgebraic} yields
\[
\QCoh(\cSpec(\mathrm{C}^{\bullet}(K(\pi_2,n_2);\Bbbk)))\simeq\LocSys(K(\pi_2,n_2);\Bbbk)
\]
that is self-dual as a presentable stable category. Writing $K_i\coloneqq \cSpec(\mathrm{C}^{\bullet}(K(\pi_i,n_i);\Bbbk))$, we have
\begin{align*}
\QCoh\lp K_1\times K_2\rp&\simeq\QCoh\lp\colim_{\mathrm{Aff}_{/K_1}}\Spec(R)\times K_2\rp\\&\simeq\lim_{\mathrm{Aff}^{\op}_{/K_1}}
\QCoh(\Spec(R)\times K_2)\\&\simeq\lim_{\mathrm{Aff}^{\op}_{/K_1}}\LocSys(K_2;R)\\
&\simeq\LocSys(K_2;\QCoh(K_1))\\&\simeq\LocSys(K_2;\Bbbk)\otimes_{\Mod_{\Bbbk}}\QCoh(K_1)\simeq\QCoh(K_1)\otimes_{\Mod_{\Bbbk}}\QCoh(K_2).
\end{align*}
\end{proof}
\begin{corollary}
\label{cor:prop_key_finite}
Let $Y$ be a pointed $0$-Koszul space. Assume $Y$ is a $\Ebb_1$-monoid and that the graded rational vector space $\pi_{\bullet}(Y_{\QQ},y)$ is finitely generated. Let $\Bbbk$ be a field of characteristic zero. Then the functor
\[
\mathrm{Loc}_{Y^{\aff}}\colon\mathrm{Lin}_{\QCoh(\cSpec(\CY))}\PrLU\longrightarrow\ShvCat(\cSpec(\CY))
\]
is fully faithful.
\end{corollary}
\begin{proof}
Since $\CY$ and $\mathrm{C}^{\bullet}(Y_{\QQ};\Bbbk)$ are equivalent $\Bbbk$-algebras for any field $\Bbbk$ of characteristic zero, we can safely reduce ourselves to the case where $Y$ is a \textit{rational} $\Ebb_1$-monoid (i.e., a rational H-space).

First, assume that $\Bbbk=\QQ$. In this case, $Y$ splits as a product of rational Eilenberg--MacLane spaces
\begin{align}
\label{equiv:rational-h-spaces}
Y\simeq\prod_{n\geqslant0}K(\pi_n(Y,y),n),
\end{align}
see for example \cite[Theorem 9.1.1]{May_Ponto}. Since $\pi_{\bullet}(Y,y)$ is a finitely generated rational vector space, this product has to be finite. So we can just conclude in virtue of \cref{lemma:key2}. For a general field extension $\QQ\subseteq\Bbbk$, we observe that the universal coefficient theorem implies that
\[
\cSpec(\CY)\simeq\cSpec(\mathrm{C}^{\bullet}(Y;\QQ))\times_{\Spec(\QQ)}\Spec(\Bbbk).
\]
Arguing as in the proof of \cref{cor:prop_key_finite}, we see that both functors involved send this product to a tensor products of $\QQ$-linear presentable $2$-categories, and thus $\mathrm{Loc}_{Y^{\aff}}$ can be written as a tensor product of two fully faithful functors.
\end{proof}
\begin{warning}
If $Y$ is a pointed rational $0$-Koszul H-space, then the equivalence \eqref{equiv:rational-h-spaces} is \textit{not} an equivalence of H-spaces, unless $Y$ is homotopy abelian. So, the equivalence
\[
\QCoh(\cSpec(\CY))\simeq\bigotimes_{n\geqslant0}\QCoh(\cSpec(\mathrm{C}^{\bullet}(K(\pi_nY,n);\Bbbk)))
\]
proved in \cref{lemma:key2} will not preserve the convolution monoidal structure in general.
\end{warning}
One would be tempted to deduce \cref{lemma:locfullyfaithful} from \cref{cor:prop_key_finite}, arguing by induction along the Postnikov tower
\[
Y\simeq\lim_{n\in\NN}\lp\cdots\longrightarrow\tau_{\leqslant n+1}Y\longrightarrow\tau_{\leqslant n}Y\longrightarrow\cdots\longrightarrow\left\{\ast\right\}\rp,
\]
since each $\tau_{\leqslant n}Y$ is again $0$-Koszul (\cref{prop:nkoszultruncated}). However, the functors $\QCoh(-)$ and $\ShvCat(-)$ do not interact well with limits.

Yet, when $Y$ is a rational $\Ebb_1$-monoid, it admits a curiously different description as a \textit{colimit} of its Postnikov tower.
\begin{construction}
\label{construction:h-spaces}
Let $Y$ be a connected rational $\Ebb_1$-monoid, and consider the decomposition as a product of rational Eilenberg--MacLane spaces \eqref{equiv:rational-h-spaces}. Write $Y(n)$ for the $n$-th step in the Postnikov tower of $Y$: it is clear that we have
\[
Y(n)\coloneqq \tau_{\leqslant n}Y\simeq\prod_{k=1}^nK(\pi_k(Y,y),k).
\]
For any $n\geqslant1$, let $\left\{\ast\right\}\to K(\pi_n(Y,y),n)$ be the map of spaces induced by the homomorphism of groups $0\to\pi_n(Y,y)$. This induces an obvious map $Y(n)\to Y(n+1)$ that is a section of the fibration $Y(n+1)\to Y(n)$. In particular, we have a morphism
\begin{align}
\label{map:postnikov}
\colim_{n\geqslant0}Y(n)\longrightarrow Y.
\end{align}
\end{construction}
\begin{proposition}
\label{prop:postnikov}
For any pointed connected rational $\Ebb_1$-monoid $Y$, the map \eqref{map:postnikov} is an equivalence of spaces.
\end{proposition}
\begin{proof}
Let us leave implicit all base points in the notation. It is sufficient to check that for any $p\geqslant0$, the map of abelian groups
\[
\pi_p{\lp\colim_n Y(n)\rp}\cong\pi_p{\lp\colim_n\prod_{k=1}^nK(\pi_k(Y),k)\rp}\longrightarrow\pi_p(Y)
\]
is an isomorphism. Since taking homotopy groups obviously commutes with both finite products and filtered colimits, we can write the above map as
\[
\colim_n\prod_{k=1}^n\pi_p(K(\pi_k(Y),k)\longrightarrow\pi_p(Y).
\]
Now, it is clear that $\pi_p(K(\pi_k(Y),k))$ is $0$ for all $p\neq k$, and is $\pi_p(Y)$ otherwise. In particular, the source of the above map is the colimit (computed in the abelian category of abelian groups) over the eventually constant sequence 
\[
0\longrightarrow0\longrightarrow\cdots\longrightarrow
0\longrightarrow\underset{p\text{-th term}}{\underbrace{\pi_p(Y)}}\overset{\cong}{\longrightarrow}\pi_p(Y)\overset{\cong}{\longrightarrow}\cdots
\]
which is clearly $\pi_p(Y)$ itself. So the map \eqref{map:postnikov} is an equivalence, as desired.
\end{proof}
\begin{corollary}
\label{corollary:colimit-postnikov}
Let $Y$ be a pointed rational $1$-Koszul $\Ebb_1$-monoid, and let $\Bbbk$ be a field of characteristic zero. Then we have equivalences of stacks
\[
Y_{\operatorname{B}}\simeq\colim_n Y(n)_{\operatorname{B}}\quad\text{ and }\quad \cSpec(\CY)\simeq\colim_n \cSpec(\mathrm{C}^{\bullet}(Y(n);\Bbbk)).
\]
\end{corollary}
\begin{proof}
The claim for the Betti stack of $Y$ is obvious from \cref{prop:postnikov}, because the functor $(-)_{\operatorname{B}}$ preserves colimits. For the coaffine stack, we first assume $\Bbbk=\QQ$. In this case, we observe that \cref{prop:postnikov} actually implies that $Y$ is a sequential colimit of $Y(n)$'s in the category of \textit{rational} spaces; in turn, this is equivalent to the category of \textit{rational homotopy types}, in the sense of \cite[Definition 1.2.1]{dagxiii}. When restricted to the full sub-category of rational spaces of finite type, such equivalence is precisely realized by the functor $\cSpec(\mathrm{C}^{\bullet}(-;\QQ))$ (\cite[$\S1.3$]{dagxiii}). Since rational homotopy types are closed under filtered colimits inside the category of prestacks (\cite[Lemma 1.2.16]{dagxiii}), it follows that 
\[
\cSpec(\mathrm{C}^{\bullet}(Y,\QQ))\simeq\colim_n\cSpec(\mathrm{C}^{\bullet}(Y(n),\QQ)).
\]
For a general field extension $\QQ\subseteq\Bbbk$, it is sufficient to observe that the base change along $\Spec(\Bbbk)\to\Spec(\QQ)$ preserves filtered colimits, because colimits in categories of stacks are universal.
\end{proof}
\begin{lemma}
\label{lemma:2-connected}
Let $Y$ be a pointed $2$-Koszul space which is a $\Ebb_1$-monoid. Let $\Bbbk$ be a field of characteristic zero. Then the functor
\[
\mathrm{Loc}_{Y^{\mathrm{aff}}}\colon\mathrm{Lin}_{\QCoh(\cSpec(\CY))}\PrLU\longrightarrow\ShvCat(\cSpec(\CY))
\]
is fully faithful.
\end{lemma}
\begin{proof}
Notice that for all integers $n\geqslant0$, the $n$-th step $Y_{\QQ}(n)$ in the Postnikov tower of the rationalization of $Y$ is the rationalization of $Y(n)$. Thus, as in the proof of \cref{cor:prop_key_finite}, we can assume that $Y$ is rational itself. Using \cref{corollary:colimit-postnikov}, we can write
\[
Y_{\operatorname{B}}\simeq\colim_n Y(n)_{\operatorname{B}}\quad\text{ and }\quad \cSpec(\CY)\simeq\colim_n \cSpec(\mathrm{C}^{\bullet}(Y(n);\Bbbk)).
\]
Then we have\[
\ShvCat(\cSpec(\CY))\simeq\lim_n \ShvCat(\cSpec(\mathrm{C}^{\bullet}(Y(n);\Bbbk)))
\]
and
\[
\LocSysCat(Y;\Bbbk)\simeq\lim_n \LocSysCat(Y(n);\Bbbk).\]
Indeed, over simply connected spaces the functor $\LocSysCat(-;\Bbbk)$ is rationally invariant as well, thanks to \cref{conj:infinityn}.) Since each $Y(n)$ is again $2$-Koszul (\cref{prop:nkoszultruncated}), and thus $1$-Koszul (\cref{prop:nkoszuldeloops}), \cref{cor:koszulalgebraic} applies. Moreover, each $Y(n)$ is $n$-truncated, so the based loop space $\Omega_*Y(n)\simeq Y(n-1)$ is a $(n-1)$-truncated rational $1$-Koszul $\Ebb_1$-monoid. Applying \cref{cor:prop_key_finite}, we obtain a chain of equivalences
\begin{align*}
\ShvCat(\cSpec(\CY)&)\simeq \lim_n \ShvCat(\cSpec(\mathrm{C}^{\bullet}(Y(n);\Bbbk)))\\&\simeq\lim_n\LocSysCat(Y(n);\Bbbk)\\&\simeq\LocSysCat(Y;\Bbbk),
\end{align*}
which is again induced by taking the limit of the pullback functors $\aff^\ast_{Y(n)}$. We now argue as we did at the end of the proof of \cref{lemma:key1}: the functor $\mathrm{Loc}_{Y^{\aff}}$ is equivalent to the fully faithful functor $\mathrm{Loc}_{Y_{\operatorname{B}}}$ and so it must be fully faithful as well.
\end{proof}
\begin{porism}
Notice that \cref{lemma:2-connected} already proves that \cref{thm:mainkoszuln} holds when $n=1$ and $X$ is a $\Ebb_1$-monoid. 
\end{porism}
\begin{proof}[Proof of \cref{prop:key}]
Let $Y$ be an arbitrary pointed $0$-Koszul space which is an $\Ebb_1$-monoid, and let $\Bbbk$ be a field of characteristic zero. As in the previous proofs, we can safely assume $Y$ to be a rational space, so we have a decomposition of $Y$ as a product of rational Eilenberg--Maclane spaces. Set
\[
Y(2)\coloneqq K(\pi_1(Y,y),2)\times K(\pi_2(Y,y),2)\quad\text{ and }\quad Y_{\geqslant3}\coloneqq \prod_{n\geqslant3}K(\pi_n(Y,y);n),
\]
so that $Y\simeq Y(2)\times Y_{\geqslant3}$. The functor $\mathrm{Loc}_{Y(2)^{\aff}}$ is fully faithful because of \cref{lemma:key2}, while the functor $\mathrm{Loc}_{Y_{\geqslant3}^{\aff}}$ is fully faithful because of \cref{lemma:2-connected}. Using the fact that $Y_{\geqslant3}$ is $2$-Koszul, and thus that \cref{cor:koszulalgebraic} applies, we deduce that its category of quasi-coherent sheaves is dualizable. So, arguing as in the proof of \cref{lemma:key2}, we obtain that
\[
\QCoh(\cSpec(\CY))\simeq\QCoh(\cSpec(\mathrm{C}^{\bullet}(Y(2);\Bbbk))\otimes_{\Mod_{\Bbbk}}\QCoh(\cSpec(\mathrm{C}^{\bullet}(Y_{\geqslant3};\Bbbk))). 
\]
Therefore, once again $\mathrm{Loc}_{Y^{\aff}}$ is realized as the tensor product of two fully faithful functors, hence it is fully faithful as well.
\end{proof}
\subsection{Consequences of \cref{thm:mainkoszuln}}
In the last section of this paper we provide some interesting consequences of \cref{thm:mainkoszuln}.
\begin{corollary}
\label{cor:1affinenessga}
Let $\Bbbk$ be a field of characteristic zero. For any $n\geqslant1$ and $k\geqslant 1$, the coaffine stack $\mathbf{B}^{n+2k+1}\Ga$ is not $n$-affine, while the Betti stack $\mathbf{B}^{n+2k}\Ga$ is $n$-affine.
\end{corollary}
\begin{proof}
Combine \cref{cor:BCPinfty} and \cref{thm:mainkoszuln}, using the fact that for any integer $p\geqslant1$ the stacks $\mathbf{B}^p\Ga$ and $\cSpec(\mathrm{C}^{\bullet}(\mathbf{B}^{p}\ZZ;\Bbbk))$ are equivalent.
\end{proof}
\begin{remark}
\cref{cor:1affinenessga} allows us to interpret \cite[Theorem 2.5.7]{1affineness} from a different perspective. Indeed, there it is shown that $\mathbf{B}\Ga$, $\mathbf{B}^2\Ga$ and $\mathbf{B}^3\Ga$ are all $1$-affine, while $\mathbf{B}^4\Ga$ is not: not only it is implied that the failure of $1$-affineness of $\mathbf{B}^4\Ga$ is somewhat surprising, but the algebraic techniques are not sufficient to study the problem for further deloopings. Under the lens of Koszul duality, this negative result acquires a new meaning. Even deloopings of $\Ga$ are \textit{almost never} $1$-affine, and $\mathbf{B}^2\Ga$ is $1$-affine only because it is not connected enough for Koszul duality to apply to this case as well. 
\end{remark}
\begin{corollary}
Let $n> m\geqslant 0$ be integers, and let $\Bbbk$ be a field of characteristic zero. For $X$ a pointed $(n+1)$-Koszul space, the fiber at the base point of the affinization map $$\aff_X\colon X_{\operatorname{B}}\longrightarrow\cSpec(\CX)$$ is equivalent to the quotient stack $[\cSpec(\COX)/\Omega_*X_{\operatorname{B}}]$, and we have an equivalence of $(m+1)$-categories
\[
(m+1)\ShvCattwo^{m}([\cSpec(\COX)/\Omega_*X_{\operatorname{B}}])\simeq (m+1)\LinkPrLUtwo[m].
\]
In particular, if $n>m+1\geqslant2$, then $\aff_X$ is an $m$-affine morphism in the sense of \cite[Definition 14.3.7]{stefanichthesis}.
\end{corollary}
\begin{proof}
Let $\eta\colon \left\{\ast\right\}\to X$ be the base point. Since $X$ is sufficiently connected, we can interpret $\aff_X$ to be the delooping of the affinization map
\[
\aff_{\Omega_*X}\colon\Omega_*X_{\operatorname{B}}\longrightarrow\Omega_*\cSpec(\CX)\simeq\cSpec(\COX).
\]
In other words: $\aff_X\simeq\mathbf{B}(\aff_{\Omega_*X})$. Since both $\Omega_*X_{\operatorname{B}}$ and $\cSpec(\COX)$ are group stacks and the affinization map is compatible with the group structure, we deduce that
\[
\fib(\aff_X)\simeq [\cSpec(\COX)/\Omega_*X_{\operatorname{B}}].
\]
The right hand side is the quotient stack of $\cSpec(\COX)$ under the action of $\Omega_*X_{\operatorname{B}}$ via the map of group stacks $\aff_{\Omega_*X}$. In particular, we have
\[
\fib(\aff_X)\simeq\colim_{\bDelta^{\op}}\lp\cdots\stack{5}\Omega_*X_{\operatorname{B}}\times\cSpec(\COX)\stack{3}\cSpec(\COX)\rp.
\]
Therefore, for all integers $m < n$ we have
\[
(m+1)\ShvCattwo^{m}(\fib(\aff_X))\simeq\lim_{[k]\in\bDelta}(m+1)\ShvCattwo^{m}\lp\Omega_*X_{\operatorname{B}}^{\times k}\times\cSpec(\COX)\rp.
\]
If $n=1$, then $m=0$ and
\[
1\ShvCattwo^0(\Omega_*X_{\operatorname{B}})\coloneqq\LocSys(\Omega_*X;\Bbbk)
\]is dualizable, therefore the categorical K\"unneth formula holds; for $n\geqslant2$ and $m\geqslant1$, it holds without any assumption of dualizability. In any case, we can write
\[
(m+1)\ShvCattwo^{m}(\fib(\aff_X))\simeq \lim_{k\in\bDelta}(m+1)\LocSysCattwo^{m}(\Omega_*X;\Bbbk)^{\otimes k}\otimes (m+1)\ShvCattwo^{m}(\cSpec(\COX)).
\]
Since $\Omega_*X$ is $m$-Koszul, \cref{thm:mainkoszuln} applies. In particular, the right hand side is the limit over the split coaugmented cosimplicial diagram 
\[
(m+1)\LinkPrLUtwo[m]\longrightarrow (m+1)\LocSysCattwo^{m}(\Omega_*X;\Bbbk)\stack{3}(m+1)\LocSysCattwo^{m}(\Omega_*X\times\Omega_*X;\Bbbk)\stack{5}\cdots
\]
which is $(m+1)\LinkPrLUtwo[m]$ itself. Since this argument works for any $m<n$, when $m<n-1$ it is clear that the adjunction $\mathrm{Loc}^m_{\fib(\aff_X)}\dashv \Gamma^{\mathrm{enh}}(\fib(\aff_X),-)$ corresponds to the autoequivalence
\[
(m+1)\Modtwo_{m\LinkPrLUtwotiny[m-1]}{\lp(m+1)\PrLUtwo[m]\rp}\simeq(m+1)\LinkPrLUtwo[m].\]
The fact that $\aff_X$ is an $m$-affine morphism follows from the above computation using the fact that the stack $\cSpec(\CX)$ is connected, so any $R$-rational point factors through $\eta$.
\end{proof}

\printbibliography
\end{document}